\documentclass[10pt]{amsart}
\usepackage[mathscr]{euscript}
\usepackage[utf8]{inputenc}

\usepackage[margin=1.1in]{geometry}
\usepackage[T1]{fontenc}
\usepackage[utf8]{inputenc}

\usepackage{adjustbox}

\usepackage{amsmath}
\usepackage{latexsym,amsfonts,amssymb,mathrsfs}
\usepackage{amsthm}
\usepackage{mathdots}
\usepackage{color}
\usepackage{enumitem} 
\usepackage{xspace} 
 \usepackage{array}
 \usepackage{multicol} 
 \usepackage{mathtools} 
 \usepackage[pagebackref]{hyperref}
\usepackage[capitalise]{cleveref}

\usepackage{quiver}
\usetikzlibrary{nfold}
\usetikzlibrary{matrix,arrows, patterns, positioning, calc, intersections}
\usepackage{tikz}
\usepackage{tikz-3dplot}

\usepackage{tikz}
\usetikzlibrary{cd}
\usetikzlibrary{calc}
\usetikzlibrary{math}
\usetikzlibrary{arrows}
\usetikzlibrary{fit}
\usetikzlibrary{positioning}
\usetikzlibrary{decorations.pathmorphing}
\usetikzlibrary{decorations.pathreplacing}
\usetikzlibrary{ decorations.markings} 
\tikzset{Rightarrow/.style={double equal sign distance,>={Implies},->},
triple/.style={-,preaction={draw,Rightarrow}},
quadruple/.style={preaction={draw,Rightarrow,shorten >=0pt},shorten >=1pt,-,double,double
distance=0.2pt}}

\usepackage{graphicx} 
\usepackage{comment} 

\usepackage[
textwidth=2cm,
textsize=small,
colorinlistoftodos]
{todonotes}

\tikzset{%
    symbol/.style={%
        draw=none,
        every to/.append style={%
            edge node={node [sloped, allow upside down, auto=false]{$#1$}}}
    }
}

\tikzset{%
scalearrow/.style n args={3}{
  decoration={
    markings,
    mark=at position (1-#1)/2*\pgfdecoratedpathlength
      with {\coordinate (#2);},
    mark=at position (1+#1)/2*\pgfdecoratedpathlength
      with {\coordinate (#3);},
    },
  postaction=decorate,
  }
}

\theoremstyle{plain}   

\theoremstyle{definition}
\newtheorem{thm}{Theorem}[section] 
\makeatletter\let\c@thm\c@thm\makeatother

\makeatletter\let\c@cor\c@thm\makeatother
\newtheorem{lem}{Lemma}[section]
\makeatletter\let\c@lem\c@thm\makeatother
\newtheorem{prop}{Proposition}[section]
\makeatletter\let\c@prop\c@thm\makeatother

\makeatletter\let\c@claim\c@thm\makeatother

\makeatletter\let\c@conjecture\c@thm\makeatother

\makeatletter\let\c@wconjecture\c@thm\makeatother

\newtheorem*{unnumberedtheorem}{Theorem}

\newtheorem{defn}{Definition}[section]
\makeatletter\let\c@defn\c@thm\makeatother
\newtheorem{const}{Construction}[section]
\makeatletter\let\c@const\c@thm\makeatother
\newtheorem{notn}{Notation}[section]
\makeatletter\let\c@notn\c@thm\makeatother

\makeatletter\let\c@convention\c@thm\makeatother

\makeatletter\let\c@convention\c@thm\makeatother

\theoremstyle{remark}

\newtheorem{rmk}{Remark}[section]
\makeatletter\let\c@rmk\c@thm\makeatother
\newtheorem{ex}{Example}[section]
\makeatletter\let\c@ex\c@thm\makeatother

\makeatletter\let\c@observation\c@thm\makeatother

\makeatletter\let\c@warning\c@thm\makeatother

\makeatletter\let\c@digression\c@thm\makeatother

\makeatletter\let\c@answ\c@thm\makeatother

\makeatletter\let\c@answ\c@thm\makeatother

\makeatletter\let\c@aside\c@thm\makeatother

\makeatletter
\let\c@equation\c@thm
\numberwithin{equation}{section}
\makeatother

\crefname{lem}{Lemma}{Lemmas}
\crefname{thm}{Theorem}{Theorems}
\crefname{defn}{Definition}{Definitions}
\crefname{notn}{Notation}{Notations}
\crefname{const}{Construction}{Constructions}
\crefname{prop}{Proposition}{Propositions}
\crefname{rmk}{Remark}{Remarks}
\crefname{cor}{Corollary}{Corollaries}
\crefname{equation}{Display}{Displays}
\crefname{ex}{Example}{Examples}
\crefname{thmalph}{Theorem}{Theorems}
\crefname{answ}{Answer}{Answers}
\crefname{question}{Question}{Questions}

\makeatletter
\newcommand{\@bbify}[1]{
  \ifcsname b#1\endcsname
  \message{WARNING: Overwriting b#1 with blackboard letter!}
  \fi
  \expandafter\edef\csname b#1\endcsname
  {\noexpand\ensuremath{\noexpand\mathbb #1}\noexpand\xspace}}
\newcommand{\@calify}[1]{
  \ifcsname c#1\endcsname
  \message{WARNING: Overwriting c#1 with calligraphic letter!}
  \fi 
  \expandafter\edef\csname c#1\endcsname
  {\noexpand\ensuremath{\noexpand\mathcal #1}\noexpand\xspace}}
\newcommand{\@bfify}[1]{
  \ifcsname bf#1\endcsname
  \message{WARNING: Overwriting c#1 with bold letter!}
  \fi
  \expandafter\edef\csname bf#1\endcsname
  {\noexpand\ensuremath{\noexpand\mathbf #1}\noexpand\xspace}}
\newcounter{@letter}\stepcounter{@letter}
\loop\@bbify{\Alph{@letter}}\@calify{\Alph{@letter}}\@bfify{\Alph{@letter}}
\ifnum\the@letter<26\stepcounter{@letter}\repeat
\makeatother

\newcommand{\set}{\cS\!\mathit{et}}
\newcommand{\sset}{\mathit{s}\set}

\DeclareMathOperator{\colim}{colim}

\DeclareMathOperator{\id}{id}

\DeclareMathOperator{\op}{op}

\AtBeginDocument{%
   \def\MR#1{}
}

\newcommand{\mrt}{\mathfrak M}

\author{Arghan Dutta}
\address{Department of Mathematics and Statistics,
University of Massachusetts Amherst, 
Amherst,
USA
}
\email{arghandutta@umass.edu} 

\author{Stefano Luneia}
\address{Department of Mathematics and Statistics,
University of Massachusetts Amherst, 
Amherst,
USA
}
\email{sluneia@umass.edu} 

\author{Martina Rovelli}
\address{Department of Mathematics and Statistics,
University of Massachusetts Amherst, 
Amherst,
USA
}
\email{mrovelli@umass.edu}
\address{Department of Mathematics and Statistics,
University of Ottawa, 
Ottawa,
Canada
}
\email{mrovelli@uottawa.ca}

\author{Sam Silver}
\address{Department of Mathematics and Statistics,
University of Massachusetts Amherst, 
Amherst,
USA
}
\email{samuelsilver@umass.edu} 

\title[The Morita $(\infty,2)$-category of a monoidal category
as a $2$-complicial set]{The Morita $(\infty,2)$-category of a monoidal category\\
as a $2$-complicial set}

\begin{document}

\maketitle

\begin{abstract}
We provide an explicit and elementary construction of the Morita $(\infty,2)$-category of a monoidal category which satisfies minimal conditions. We construct it as a $3$-coskeletal $2$-complicial set, in which the vertices encode the monoids, the edges encode the bimodules, the triangles encode the bimodule maps out of a balanced tensor product, and tetrahedra encode composition of bimodule maps. The marked edges encode invertible bimodules, and the marked triangles encode bimodule isomorphisms with a balanced tensor product.
\end{abstract}

\tableofcontents

\section*{Introduction}

Two primary examples of non-strict $2$-dimensional categorical structures include the one where the objects, $1$-morphisms and $2$-morphisms are given by rings, bimodules and bimodule maps,
and the one where they are given by sets, cospans and cospan maps. One of the reasons why the two examples do not form a strict $2$-category is that composition of $1$-morphisms, which is given by bimodule tensor product and pushout, respectively, is not strictly associative. Each forms a \emph{bicategory}, which is discussed e.g.~in \cite{SPclassification} in the case of abelian groups and bimodules and in \cite{benabou} for the case of sets and cospans.

The two constructions are special cases of a more general pattern.
Provided that a monoidal category satisfies satisfies specific assumptions, which we indicate in \cref{Axioms} (following \cite{Vitale}), there is a typically non-strict $2$-dimensional category in which the objects, $1$-morphisms and 2-morphisms are monoids, bimodules and bimodule maps.
The two examples are recovered when considering the monoidal category of abelian groups with respect to tensor product, and the category of sets with respect to disjoint union.

This is a well-known construction in the literature, often regarded as a folklore fact. Various constructions exist in greater generality, though they typically rely on significantly more advanced machinery. Alternatively, elementary proofs exist, but they are restricted to specific monoidal categories. We provide in this note a self-contained, explicit, and elementary proof which works for this particular level of generality.

We bypass the verification of the bicategory axioms, and instead we construct a truncated $(\infty,2)$-category, in the form of a $3$-coskeletal $2$-complicial set. A \emph{$2$-complicial set} is a particular model for the notion of an $(\infty,2)$-category due to Verity, based on simplicial sets with marking (see \cite{VerityComplicial,RiehlComplicial,ORms}). The interpretation is that $k$-simplices encode $k$-dimensional morphisms (with appropriate boundary decompositions), and marked $k$-simplices encode $k$-morphisms which are weakly invertible.

Given a monoidal category satisfying appropriate conditions, we introduce in \cref{MonoidalNerve} a $3$-coskeletal marked simplicial set, and we show as \cref{maintheorem} that it is a $2$-complicial set. The construction is completely explicit and does not rely on any prior result.

\begin{unnumberedtheorem}
 Given a monoidal category $\cC^{\otimes}$ which admits a calculus of balanced tensor products, there is a $(\infty,2)$-category in which the objects are the monoids in $\cC^{\otimes}$, the $1$-morphisms are the bimodules and the $2$-morphisms are the bimodule maps.    
\end{unnumberedtheorem}

Our approach essentially repackages the task of verifying the bicategory axioms using the combinatorics of simplicial sets. While the resulting effort is comparable overall, we hope this approach offers a proof of concept for a more efficient formulation of the conditions, and one that can be more readily scaled to higher dimensional structures (for instance when treating braided monoidal categories instead of monoidal categories).

Assuming having checked independently that monoids, bimodules, and bimodule maps form a bicategory (a well-believed fact, albeit one whose verification remains somewhat lengthy),
one could alternatively consider the marked nerve of this bicategory, which is considered and shown to be a complicial set by Gurski in \cite{Gurski} (building on work by Duskin \cite{Duskin}).
We believe that the resulting $2$-complicial set would agree with the one presented here, though the description of the marked simplices appears to be presented differently, and we did not investigate a direct comparison.

In terms of related literature, various groups of authors have already successfully constructed $(\infty,n)$-categories in which the $k$-morphisms are $E_k$-algebras in a monoidal $\infty$-category. For instance, some approaches in this direction include the works by Haugseng \cite{HaugsengMorita}, Johnson-Freyd--Scheimbauer \cite{JFS}, and Gwilliam--Scheimbauer \cite{GS}. Lurie formalizes the construction of the $\infty$-category of bimodules in a monoidal $\infty$-category in \cite[\textsection4.3]{LurieHA}.

\subsubsection*{Acknowledgments}
MR is grateful for support from the National Science Foundation under Grant No.~DMS-2203915. MR would also like to thank the Isaac Newton Institute for Mathematical Sciences, Cambridge, for support and hospitality during the programme \emph{Equivariant homotopy theory in context} (supported by EPSRC grant EP/Z000580/1), where work on this paper was undertaken. We are thankful to Viktoriya Ozornova for valuable conversations on the topic of this note.

\section{Monoids and bimodules in a monoidal category}

We recall the notions of monoidal categories, as well as those of monoids, bimodules and bimodule maps in a monoidal category.

The definition of monoidal category is classical; see e.g.~\cite[\textsection VII.1]{MacLane}:

\begin{defn}
\label{MonoidalCategory}
A \emph{monoidal category} $(\cC,\otimes,I,\alpha,\lambda,\rho)$, often denoted just $\cC^{\otimes}$, consists of a category $\cC$ together with a \emph{tensor product} functor $\otimes\colon \cC\times \cC\to \cC$ and a \emph{unit object} $I\in \cC$, and the following natural isomorphisms:
\begin{itemize}[leftmargin=*]
    \item For $X$ in $\cC$, the \emph{left unitor} and the \emph{right unitor}
\[\lambda_X: I\otimes X\xrightarrow{\cong} X\quad\text{ and }\quad\rho_X: X\otimes I\xrightarrow{\cong} X;\]
    \item For $X,Y,Z$ in $\cC$, the \emph{associator}
\[
\alpha_{X,Y,Z}: (X\otimes Y)\otimes Z\xrightarrow{\cong} X\otimes (Y\otimes Z);
\]
\end{itemize}
satisfying the following coherence conditions:
\begin{itemize}[leftmargin=*]
    \item For $X$ and $Y$ in $\cC$, the the following diagram commutes:
\[
\begin{tikzcd}
(X\otimes I)\otimes Y\arrow[rr,"\alpha_{X,I,Y}"]\arrow[rd,"\rho_X\otimes Y"'] & & X\otimes(I\otimes Y)\arrow[ld,"X\otimes\lambda_Y"] \\
& X\otimes Y &
\end{tikzcd}
\]
\item and, for $X,Y,Z,W$ in $\cC$, the following diagram commutes:
\[
\begin{tikzcd}
& (X\otimes Y)\otimes (Z\otimes W)\arrow[ld,"\alpha_{X,Y,Z\otimes W}" swap] & \\
X \otimes (Y \otimes (Z\otimes W)) & & ((X\otimes Y)\otimes Z)\otimes W\arrow[ul,"\alpha_{X\otimes Y,Z,W}" swap]\arrow[d,"\alpha_{X,Y,Z}\otimes W"] \\
X\otimes ((Y\otimes Z)\otimes W)\arrow[u,"X\otimes \alpha_{Y,Z,W}" ] & & (X\otimes (Y\otimes Z))\otimes W\arrow[ll,"\alpha_{X,Y\otimes Z,W}"]
\end{tikzcd}
\]
\end{itemize}
\end{defn}

\begin{prop}[\cite{KellyCoherence}]
\label{AssociatorAndUnitor}
   Let $\cC^{\otimes}$ be a monoidal category. Given $X$ and $Y$ in $\cC$, the following diagrams commute
   \[
   \begin{tikzcd}
       (I\otimes X)\otimes Y\arrow[rd,"\lambda_X\otimes Y"swap]\arrow[r,"\alpha_{I,X,Y}"]&I\otimes (X\otimes Y)\arrow[d,"\lambda_{X\otimes Y}"]\\
       &X\otimes Y
   \end{tikzcd}
   \quad
   \begin{tikzcd}
       (X\otimes Y)\otimes I\arrow[d,"\rho_{X\otimes Y}"swap]\arrow[r,"\alpha_{X,Y,I}"]&X\otimes (Y\otimes I)\arrow[ld,"X\otimes \rho_{Y}"]\\
       X\otimes Y&
   \end{tikzcd}
   \]
\end{prop}

We recall a variety of contexts of relevance in which monoidal categories arise.

\begin{ex}
The following are known examples of monoidal categories (cf.~e.g.~\cite[\textsection~VII.1]{MacLane}):
\begin{itemize}[leftmargin=*]
    \item The monoidal category $\set^\times$, given by the category $\set$ of sets with the cartesian product $\times\colon\set\times\set\to\set$; more generally, the monoidal category $\cC^\times$ with $\cC$ a category that admits finite products;
    \item The monoidal category $\set^\amalg$, given by the category $\set$ of sets with the disjoint union $\amalg\colon\set\times\set\to\set$; more generally, the monoidal category $\cC^\amalg$ with $\cC$ a category which admits finite coproducts.
    \item The monoidal category $\cA b^\otimes$, given by the category $\cA b$ of abelian groups with the ordinary tensor product $\otimes\colon\cA b\times\cA b\to\cA b$; more generally, the monoidal category $R\cM od^{\otimes_R}$ with $R$ a commutative ring.
\end{itemize} 
\end{ex}

\begin{ex}
New monoidal categories can be obtained from existing ones through standard procedures:
\begin{itemize}[leftmargin=*]
    \item Given a monoidal category $\cC^{\otimes}$, the opposite monoidal category $\cC^{\op}$ can be endowed with the structure of a monoidal category (cf.~e.g.~\cite[Example 2.1.3.4]{LurieKerodon}).
    \item Given monoidal categories $\cC^{\otimes}$ and $\cD^{\boxtimes}$, the product category $\cC\times\cD$ can be endowed with the structure of a monoidal category (cf.~e.g.~\cite[\textsection VII.1]{MacLane}).
\end{itemize}
\end{ex}

\subsection{Monoids}

The definition of monoid in a monoidal category is classical; see e.g.~\cite[\textsection VII.3]{MacLane}:

\begin{defn}
Let $\cC^{\otimes}$ be a monoidal category. A \emph{monoid} $(A,m_A,e_A)$, often denoted just $A$, consists of an object $A$ in $\cC$, together with a \emph{multiplication} $m_A\colon A\otimes A\to A$ and a \emph{unit} $e_A\colon I\to A$ in $\cC$, such that:
\begin{enumerate}[leftmargin=*]
\item The multiplication map is \emph{associative}; that is, the following diagram commutes:
\[
\begin{tikzcd} A \otimes (A \otimes A) \arrow[d, "A \otimes m_A"'] &\arrow[l, "\alpha_{A,A,A}" swap] (A \otimes A) \otimes A \arrow[r, "m_A \otimes A"] & A \otimes A \arrow[d, "m_A"] \\ A \otimes A \arrow[rr, "m_A"'] & & A \end{tikzcd}
\]
\item The multiplication map is \emph{unital}; that is, the following diagram commutes:
\[
\begin{tikzcd}
I\otimes A\arrow[rd,"\lambda_A" swap]\arrow[r,"e_A\otimes A"]&A\otimes A\arrow[d]&A\otimes I\arrow[ld,"\rho_A"]\arrow[l,"A\otimes e_A" swap]\\
&A&
\end{tikzcd}
\]
\end{enumerate}
\end{defn}

The notion of monoid in established monoidal categories recovers known structures (cf.~\cite[\textsection~VII.3]{MacLane}):

\begin{ex}
\begin{enumerate}[leftmargin=*]
        \item A monoid in $\set^{\times}$ is an odinary monoid;
        \item A monoid in $\set^{\amalg}$ is equivalent to a set;
        \item A monoid in $\cA b^{\otimes}$ is a ring, while a monoid in $\mathbb K\cV ect$ is a $\mathbb K$-algebra.
    \end{enumerate}
\end{ex}

\subsection{Bimodules}

The definition of bimodule over monoid in a monoidal category is classical; see e.g.~\cite[\textsection VII.4]{MacLane}:

\begin{defn}
Let $\cC^{\otimes}$ be a monoidal category, and $(A,m_A,e_A)$, $(B,m_B,e_B)$ monoids in $\cC^{\otimes}$. An \emph{$(A,B)$-bimodule} $(A,B,M,\ell_M,r_M)$, often denoted just $M$, consists of an object $M$ in $\cC$, together with a \emph{left action} $\ell_M\colon A\otimes M\to M$, and a \emph{right action} $r_M\colon M\otimes B\to M$, such that:
\begin{enumerate}[leftmargin=*]
    \item The actions are \emph{unital}; that is, the following diagrams commute:
\[
\begin{tikzcd}
I \otimes M \arrow[rd, "\lambda_M" swap] \arrow[r, "e_A\otimes M"] 
  & A \otimes M \arrow[d, "\ell_M"] \\
& M
\end{tikzcd}
\quad\text{ and }\quad
\begin{tikzcd}
M \otimes B \arrow[d, "r_M" swap] 
  & M \otimes I \arrow[l, "M \otimes e_B" swap] \arrow[ld, "\rho_M"] \\
M &
\end{tikzcd}
\]
    \item The left action is associative; that is, the following diagram commutes:
\[
\begin{tikzcd} A \otimes (A \otimes M)  \arrow[d, "A \otimes \ell_M"'] &\arrow[l, "\alpha_{A,A,M}" swap] (A \otimes A) \otimes M \arrow[r, "m_A \otimes M"] & A \otimes M \arrow[d, "\ell_M"] \\ A \otimes M \arrow[rr, "\ell_M"'] & & M \end{tikzcd}\]
and the right action is associative; that is, the following diagram commutes:
\[
\begin{tikzcd} (M \otimes B) \otimes B \arrow[r, "\alpha_{M,B,B}"] \arrow[d, "r_M \otimes B"'] & M \otimes (B \otimes B) \arrow[r, "M \otimes m_B"] & M \otimes B \arrow[d, "r_M"] \\ M \otimes B \arrow[rr, "r_M"'] & & M \end{tikzcd}
\]

\item The actions are compatible; that is, the following diagrams commute:
\[
\begin{tikzcd} A \otimes (M \otimes B) \arrow[d, "A \otimes r_M"'] & \arrow[l, "\alpha_{A,M,B}"  swap] (A \otimes M) \otimes B \arrow[r, "\ell_M \otimes B"] & M \otimes B \arrow[d, "r_M"] \\ A \otimes M \arrow[rr, "\ell_M"'] & & M \end{tikzcd}
\]
\end{enumerate}
\end{defn}

The following constructions allow one to produce bimodules:

\begin{rmk}
Let $\cC^{\otimes}$ be a monoidal category.
\begin{itemize}[leftmargin=*]
\item
\label{IdentityBimodule}
Given a monoid $A$ in $\cC$, the map
\[
m_A\colon A\otimes A\to A
\]   
endows $A$ with an $(A,A)$-bimodule structure.
 \item \label{FreeModule}
 Given monoids $A$ and $C$, a left $A$-module $M$, and a right $C$-module $N$,
 the maps
 \[
 \ell_{M\otimes N}\colon A\otimes (M\otimes N)\xrightarrow{\alpha^{-1}_{A,M,N}}(A\otimes M)\otimes N\xrightarrow{\ell_M\otimes N}M\otimes N\]
 \[
 \quad\text{ and }\quad
 r_{M\otimes N}\colon (M\otimes N)\otimes C\xrightarrow{\alpha_{M,N,C}}M\otimes(N\otimes C)\xrightarrow{M \otimes r_N} M\otimes N
 \]
 endow $M\otimes N$ with the structure of an $(A,C)$-bimodule.
\end{itemize}
\end{rmk}

The notion of bimodule in established monoidal categories recovers known structures:

\begin{ex}
\begin{enumerate}[leftmargin=*]
    \item Given monoids $A$ and $B$, an $(A,B)$-bimodule $M$ in $\set^{\times}$ is equivalent to an $A\times B^{\op}$-set;
    \item
    Given sets $A$ and $B$, an $(A,B)$-bimodule $M$
    in $\set^{\amalg}$ is equivalent to a span
        \[A\xrightarrow{f}M\xleftarrow{g}B;
        \]
        \item
        Given rings $A$ and $B$, an $(A,B)$-bimodule $M$
        in $\cA b^{\otimes}$ is an $A\otimes B^{\op}$-module, or an $(A,B)$-bimodule in the traditional sense.
    \end{enumerate}
\end{ex}

\subsection{Bimodule maps}
The definition of bimodule map in a monoidal category is classical:

\begin{defn}
Let $\cC^{\otimes}$ be a monoidal category, $A$, $B$ be monoids and $M$, $N$ be $(A,B)$-bimodules. An $(A,B)$-\emph{bimodule map} consists of a morphism $\varphi\colon M\to N$ in $\cC$ that is compatible with the right actions and with the left actions of $M$ and $N$; that is, the following diagrams commute:
\[
\begin{tikzcd}
    A\otimes M\arrow[r,"A\otimes\varphi"]\arrow[d,"\ell_M" swap]&A\otimes N\arrow[d,"\ell_{N}"]\\
    M\arrow[r,"\varphi" swap]&N
\end{tikzcd}
\quad\text{ and }\quad
\begin{tikzcd}
    M\otimes B\arrow[r,"\varphi\otimes B"]\arrow[d,"r_M" swap]&N\otimes B\arrow[d,"r_N"]\\
    M\arrow[r,"\varphi" swap]&N
\end{tikzcd}
\]
\end{defn}

The following constructions allow one to produce bimodule maps:

\begin{rmk}
Let $\cC^{\otimes}$ be a monoidal category.
\begin{itemize}[leftmargin=*]
    \item Given monoids $A$ and $B$, three $(A,B)$-bimodules $M$, $N$ and $P$, and $(A,B)$-bimodule maps $\varphi\colon M\to N$ and $\psi\colon N\to P$, the composite\footnote{We use the convention $\varphi\bullet\psi\coloneqq\psi\circ\varphi$.} morphism $\varphi\bullet\psi\colon M\to N$ is an $(A,B)$-bimodule map.
    \item Given a monoid $A$, the identity map $\id_A\colon A\to A$ is an $(A,A)$-bimodule map.
    \item Given monoids $A$ and $B$, two $(A,B)$-bimodules $M$ and $N$, and an $(A,B)$-bimodule map $\varphi\colon M\to N$, if $\varphi$ is an isomorphism in $\cC$ then the inverse map $\varphi^{-1}\colon N\to M$ is an $(A,B)$-module map.
    \item 
    \label{ActionIsmap}
    Given monoids $A,B$ and $M$ an $(A,B)$-bimodule, the left action map $\ell_M\colon A\otimes M\to M$
is a map of $(A,B)$-bimodules with respect to the bimodule structure from \cref{FreeModule}. Similarly, the right action map $r_M\colon M\otimes B\to M$ is a map of $(A,B)$-bimodules.
\end{itemize}
\end{rmk}

\section{Monoids and Bimodules in a monoidal category with balanced tensor products}

\begin{defn}
\label{BalancedTensor}
Let $\cC^{\otimes}$ be a monoidal category.
Suppose we are given monoids $A$, $B$, $C$, an $(A,B)$-bimodule $M$ and a $(B,C)$-bimodule $N$. The \emph{balanced tensor product} of $M$ and $N$ over $B$ is the coequalizer
    \[M\otimes_BN\coloneqq\mathrm{coeq}
    \left[
    \begin{tikzcd}
    {(M \otimes B )\otimes N} && {M \otimes N}
    \arrow["r_M \otimes N", shift left=2, from=1-1, to=1-3]
    \arrow["\alpha_{M,B,N}\bullet(M \otimes \ell_N)"', shift right=2, from=1-1, to=1-3]
    \end{tikzcd}
    \right]
    \]
\end{defn}

In presence of all balanced tensor products, we can perform various constructions on bimodules.

\begin{prop}
\label{TensorFunctorial}
Let $\cC^{\otimes}$ be a monoidal category admitting all balanced tensor products. Suppose we are given monoids $A$, $B$, $C$, two $(A,B)$-bimodules $M$, $M'$, two $(B,C)$-bimodules $N$, $N'$, a map of $(A,B)$-bimodules $\varphi\colon M\to M'$ and a map of $(B,C)$-bimodules $\psi\colon N\to N'$. Then there is an induced map in $\cC$
\[\varphi\otimes_B\psi\colon M\otimes_BN\to M'\otimes_BN'.\]
\end{prop}

\begin{proof}
Since $\varphi$ and $\psi$ are bimodule maps there are commutative diagrams in $\cC$
\[
\begin{tikzcd}
(M\otimes B)\otimes N\arrow[rr,"(\varphi\otimes B)\otimes\psi"]\arrow[d,"r_M\otimes N" swap]&&(M'\otimes B)\otimes N'\arrow[d,"r_{M'}\otimes N'"]\\
M\otimes N\arrow[rr,"\varphi\otimes\psi" swap]&&M'\otimes N'
\end{tikzcd}
\quad
\begin{tikzcd}
M\otimes (B\otimes N)\arrow[rr,"\varphi\otimes( B\otimes\psi)"]\arrow[d,"M\otimes \ell_N" swap]&&M'\otimes( B\otimes N')\arrow[d,"M'\otimes \ell_{N'}"]\\
M\otimes N\arrow[rr,"\varphi\otimes\psi" swap]&&M'\otimes N'
\end{tikzcd}
\]
By naturality of the associator $\alpha$ on the triple $(\varphi,B,\psi)$, the second diagram can be rewritten as follows
\[
\begin{tikzcd}
(M\otimes B)\otimes N\arrow[rr,"(\varphi\otimes B)\otimes\psi"]\arrow[d,"r_M\otimes N" swap]&&(M'\otimes B)\otimes N'\arrow[d,"r_{M'}\otimes N'"]\\
M\otimes N\arrow[rr,"\varphi\otimes\psi" swap]&&M'\otimes N'
\end{tikzcd}
\quad
\begin{tikzcd}
(M\otimes B)\otimes N\arrow[rr,"(\varphi\otimes B)\otimes\psi"]\arrow[d,"\alpha_{M,B,N}\bullet(M\otimes \ell_N)" swap]&&(M'\otimes B)\otimes N'\arrow[d,"\alpha_{M',B,N'}\bullet(M'\otimes \ell_{N'})"]\\
M\otimes N\arrow[rr,"\varphi\otimes \psi" swap]&&M'\otimes N'
\end{tikzcd}
\]
By the universal property of coequalizers we then obtain an induced map in $\cC$
    \[
\varphi\otimes_B\psi\colon M\otimes_BN\to M'\otimes_BN',
    \]
as desired.
\end{proof}

\begin{rmk}
\label{CompatibilityTensorOnMaps}
    In the context of \cref{TensorFunctorial}, there is a commutative diagram in $\cC$:
    \[
    \begin{tikzcd}
        M\otimes N\arrow[r,"\varphi\otimes\psi"]\arrow[d,"\pi_{M,N}"swap]&M'\otimes N'\arrow[d,"\pi_{M',N'}"]\\
        M\otimes_B N\arrow[r,"\varphi\otimes_B\psi"swap]&M'\otimes_B N'
    \end{tikzcd}
    \]
\end{rmk}

\begin{lem}
\label{RmkEpi}
Let $\cC^{\otimes}$ be a monoidal category admitting all balanced tensor products. Suppose we are given monoids $A$, $B$, $C$, an $(A,B)$-bimodule $M$, three $(B,C)$-bimodules $N$, $N'$, $N''$, and maps of $(B,C)$-bimodules $\gamma\colon N\to N'$,  $\gamma'\colon N'\to N''$, $\gamma''\colon N\to N''$. If the left diagram commutes then so does the right diagram.
\[
\begin{tikzcd}
    N\arrow[rd,"\gamma" swap]\arrow[rr,"\gamma''"]&& N''\\
    & N'\arrow[ru,"\gamma'" swap]&
\end{tikzcd}
\quad\rightsquigarrow\quad
\begin{tikzcd}
    M\otimes_B N\arrow[rd,"M\otimes_B \gamma" swap]\arrow[rr,"M\otimes_B \gamma''"]&&M\otimes_B N''\\
    &M\otimes_B N'\arrow[ru,"M\otimes_B \gamma'" swap]&
\end{tikzcd}
\]
\end{lem}

\begin{proof}
Consider the diagram in $\cC$ given by
\[\begin{tikzcd}
	{M\otimes N} && {M\otimes N''} \\
	&&& {M\otimes_BN} && {M\otimes_BN''} \\
	& {M\otimes N'} \\
	&&&& {M\otimes_BN'}
	\arrow["{M\otimes\gamma''}", from=1-1, to=1-3]
	\arrow["{M\otimes\gamma'}"', from=3-2, to=1-3]
	\arrow["{\pi_{M,N}}"{description, pos=0.3}, from=1-1, to=2-4,crossing over]
	\arrow["{M\otimes\gamma}"', from=1-1, to=3-2]
	\arrow["{\pi_{M,N''}}"{description}, from=1-3, to=2-6]
	\arrow["{M\otimes_B\gamma''}"', from=2-4, to=2-6]
	\arrow["{M\otimes_B\gamma}"', from=2-4, to=4-5]
	\arrow["{\pi_{M,N'}}"{description}, from=3-2, to=4-5]
	\arrow["{M\otimes_B\gamma'}"', from=4-5, to=2-6]
\end{tikzcd}\]
The top triangular face commutes by functoriality of the tensor product in $\cC$, and the three square faces commute by \cref{TensorFunctorial}.
By \cite[Proposition~1.3]{EH2} or \cite[Proposition~3.19]{awodeyBook},
the map $\pi_{M,N}\colon M\otimes N\to M\otimes_BN$
is an epimorphism in $\cC$, so the front triangle commutes too, as desired.
\end{proof}

\begin{prop}
\label{InducedAlpha}
Let $\cC^\otimes$ be a monoidal category. Suppose we are given monoids $A,B,C,D$, an $(A,B)$-bimodule $M$, a $(B,C)$-bimodule $N$ and a $(C,D)$-bimodule $P$.
 \begin{enumerate}[leftmargin=*]
 \item When considering $M\otimes N$ with the $(A,C)$-bimodule structure from \cref{FreeModule}, the composite
 \[
 (M\otimes N)\otimes P\xrightarrow{\alpha_{M,N,P}}M\otimes(N\otimes P)\xrightarrow{M\otimes\pi_{N,P}} M\otimes (N\otimes_CP)\]
 induces an map in $\cC$
 \[
\overline\alpha_{M,N|P}\colon(M\otimes N)\otimes_CP\xrightarrow{}M\otimes(N\otimes_CP).
 \]
   \item When considering $N\otimes P$ with the $(B,D)$-bimodule structure from \cref{FreeModule}, the map
 \[ M\otimes (N\otimes P)\xrightarrow{\alpha^{-1}_{M,N,P}}(M\otimes N)\otimes P\xrightarrow{\pi_{M,N}\otimes P}(M\otimes_BN)\otimes P\]
 induces a map in $\cC$
        \[\overline{\alpha^{-1}}_{M|N,P}\colon M\otimes_B(N\otimes P)
        \xrightarrow{}(M\otimes_BN)\otimes P.
        \] 
 \end{enumerate}
\end{prop}

\begin{proof}
We prove (1), (2) being analogous. To this end, using the naturality of the associator $\alpha$, the pentagon identity, and the module structure from \cref{FreeModule}, we observe that we have the following equality:
\[
\begin{array}{lll}
 \alpha_{M \otimes N, C, P} \bullet (M \otimes N) \otimes \ell_{P} \bullet \alpha_{M,N,P} \bullet M \otimes \pi_{N,P}\\
 \quad\quad = \alpha_{M \otimes N, C, P} \bullet \alpha_{M,N, C \otimes P} \bullet M \otimes (N \otimes \ell_{P}) \bullet M \otimes \pi_{N,P}\\
 
 \quad\quad = \alpha_{M \otimes N, C, P} \bullet \alpha_{M,N, C \otimes P} \bullet M \otimes ((N \otimes \ell_{P}) \bullet \pi_{N,P}) \\
 \quad\quad = \alpha_{M \otimes N, C, P} \bullet \alpha_{M,N, C \otimes P} \bullet M \otimes ((r_N\otimes P) \bullet \pi_{N,P}) \\

 \quad\quad = \alpha_{M \otimes N, C, P} \bullet \alpha_{M,N, C \otimes P} \bullet M \otimes \alpha_{N,C,P}^{-1} \bullet M \otimes (r_{N} \otimes P) \bullet M \otimes \pi_{N,P} \\
 \quad\quad= \alpha_{M \otimes N, C, P} \bullet \alpha_{M,N, C \otimes P} \bullet M \otimes \alpha_{N,C,P}^{-1} \bullet \alpha_{M, N \otimes C, P}^{-1} \bullet (M \otimes r_{N}) \otimes P \bullet \alpha_{M,N,P} \bullet M \otimes \pi_{N,P} \\
\quad\quad= \alpha_{M,N,C} \otimes P \bullet (M \otimes r_{N}) \otimes P \bullet \alpha_{M,N,P} \bullet M \otimes \pi_{N,P} &\\
 \quad\quad= r_{M \otimes N} \otimes P \bullet \alpha_{M,N,P} \bullet M \otimes \pi_{N,P}
\end{array}
\]
Hence, there is an induced map
\[
\overline\alpha_{M,N|P}\colon(M\otimes N)\otimes_CP\xrightarrow{}M\otimes(N\otimes_CP),\] 
as desired.  
\end{proof}

\begin{rmk}
\label{CompatibilityOverlineAlpha}
    In the context of \cref{InducedAlpha},
    there are commutative diagrams
    \[
    \begin{tikzcd}
        (M\otimes N)\otimes P\arrow[r,"\alpha_{M,N,P}"]\arrow[d,"\pi_{M\otimes N,P}"swap]&M\otimes (N\otimes P)\arrow[d,"M\otimes\pi_{N,P}"]\\
        (M\otimes N)\otimes_C P\arrow[r,"\overline\alpha_{M,N|P}"swap]&M\otimes (N\otimes_C P)
    \end{tikzcd}
    \quad\text{ and }\quad
    \begin{tikzcd}
   M\otimes (N\otimes P)\arrow[r,"\alpha^{-1}_{M,N,P}"]\arrow[d,"\pi_{M,N\otimes P}"swap]&(M\otimes N)\otimes P\arrow[d,"\pi_{M,N}\otimes P"]\\
   M\otimes_B (N\otimes P) \arrow[r,"\overline{\alpha^{-1}}_{M|N,P}"swap]&(M\otimes_B N)\otimes P
    \end{tikzcd}
    \]
\end{rmk}

We impose conditions on the monoidal category $\cC^{\otimes}$ which guarantee that bimodules can be composed over monoids. These conditions are essentially from \cite[Axioms~1.3]{Vitale}.

\begin{defn}
Let $\cC^{\otimes}$ be a monoidal category. We say that $\cC^{\otimes}$
\emph{admits a calculus of balanced tensor products}
if the following conditions are satisfied for all monoids and for all $(A,B)$-bimodule $M$, $(B,C)$-bimodule $N$, and $(C,D)$-bimodule $P$.
\begin{enumerate}[leftmargin=*]
    \item[(0)] The coequalizer
    \[M\otimes_BN\coloneqq\mathrm{coeq}
    \left[
    \begin{tikzcd}
    {(M \otimes B )\otimes N} && {M \otimes N}
    \arrow["r_M \otimes N", shift left=2, from=1-1, to=1-3]
    \arrow["\alpha_{M,B,N}\bullet(M \otimes \ell_N)"', shift right=2, from=1-1, to=1-3]
    \end{tikzcd}
    \right]
    \]
    exists in $\cC$. Let $\pi_{M,N}$ denote the canonical map $\pi_{M,N}\colon M\otimes N\to M\otimes_BN$.
    \item[($1'$)] The functor $M\otimes(-)\colon\cC\to\cC$ preserves the coequalizer of diagrams of the form
 \[
    \begin{tikzcd}
    {(N \otimes C )\otimes P} && {N \otimes P}
    \arrow["r_N \otimes P", shift left=2, from=1-1, to=1-3]
    \arrow["\alpha_{N,C,P}\bullet(N \otimes \ell_P)"', shift right=2, from=1-1, to=1-3]
    \end{tikzcd}
    \]
    \item[($2'$)] The functor $(-)\otimes P\colon\cC\to\cC$ preserves the coequalizer of diagrams of the form
 \[
    \begin{tikzcd}
    {(M \otimes B )\otimes N} && {M \otimes N}
    \arrow["r_M \otimes N", shift left=2, from=1-1, to=1-3]
    \arrow["\alpha_{M,B,N}\bullet(M \otimes \ell_N)"', shift right=2, from=1-1, to=1-3]
    \end{tikzcd}
    \]
\end{enumerate}
\end{defn}

\begin{prop}
\label{Axioms}
In presence of (0), Axioms ($1'$) and ($2'$) imply the following:
\begin{enumerate}[leftmargin=*]
    \item The map from \cref{InducedAlpha} is an isomorphism in $\cC$
 \[
\overline\alpha_{M,N|P}\colon(M\otimes N)\otimes_CP\xrightarrow{\cong}M\otimes(N\otimes_CP).
 \]
   \item The map from \cref{InducedAlpha} is an isomorphism in $\cC$
        \[\overline{\alpha^{-1}}_{M|N,P}\colon M\otimes_B(N\otimes P)
        \xrightarrow{\cong}(M\otimes_BN)\otimes P.
        \] 
\end{enumerate}
\end{prop}

\begin{proof}
We briefly describe how to address the case of $\overline\alpha$, the case of $\overline\alpha^{-1}$ being analogous.
Consider the following diagram:
\[\begin{tikzcd}
	{((M\otimes N)\otimes C)\otimes P} &&& {M\otimes ((N\otimes C)\otimes P)} \\
	&&& {} \\
	{(M\otimes N)\otimes P} &&& {M\otimes (N\otimes P)}
	\arrow["{\alpha_{M,N,C}\otimes P\bullet\alpha_{M,N\otimes C,P}}", from=1-1, to=1-4]
	\arrow["{(\alpha_{M,N,C}\bullet M\otimes r_N)\otimes P}"', dashed, from=1-1, to=3-1]
	\arrow["{\alpha_{M\otimes N,C,P}\bullet(M\otimes N)\otimes\ell_P}", shift left=3, dotted, from=1-1, to=3-1]
	\arrow["{M\otimes(r_N\otimes P)}"', dashed, from=1-4, to=3-4]
	\arrow["{M\otimes(\alpha_{N,C,P}\bullet N\otimes \ell_P)}", shift left=3, dotted, from=1-4, to=3-4]
	\arrow["{\alpha_{M,N,P}}", from=3-1, to=3-4]
\end{tikzcd}\]
One can see, using repeatedly the pentagon axiom from \cref{MonoidalCategory}(2), the naturality of the associator $\alpha$, and the bifunctoriality of the tensor product $\otimes$, that the square obtained using the dashed arrows and the square obtained using the dotted arrows commute individually, so there is an induced arrow at the level of their coequalizers. Given that the two horizontal maps are isomorphisms, it follows that the induced map, which is by assumption ($1'$) can be identified with $\overline\alpha_{M,N|P}$, is also one.
\end{proof}

Many examples of interest satisfy these properties:

\begin{ex}
\begin{itemize}[leftmargin=*]
     \item If $\cC^{\otimes}$ is a cocomplete closed monoidal category, then $\cC^{\otimes}$ admits a calculus of balanced tensor products. In particular, the monoidal categories $\set^{\times}$ and $\cA b^{\otimes}$ admit a calculus of balanced tensor products.
     \item If $\cC$ is a cocomplete category, then $\cC^{\amalg}$ admits a calculus of balanced tensor products. In particular, the monoidal category $\set^{\amalg}$ admits a calculus of balanced tensor products.
     \item If $\cC^{\otimes}$ and $\cD^{\boxtimes}$ admit a calculus of balanced tensor products, then their product is a monoidal category which admits a calculus of balanced tensor products when endowed with the usual monoidal product.
 \end{itemize}   
\end{ex}

\begin{prop}
\label{InducedAlpha2}
Let $\cC^\otimes$ be a monoidal category. Suppose we are given monoids $A,B,C,D$, an $(A,B)$-bimodule $M$, a $(B,C)$-bimodule $N$ and a $(C,D)$-bimodule $P$. Then the maps from \cref{InducedAlpha} induce inverse isomorphisms in $\cC$
\[
\overline\alpha_{M|N|P}\colon(M\otimes_BN)\otimes_CP\xrightarrow{\cong} M\otimes_B(N\otimes_CP).
\]
and
\[
\overline{\alpha^{-1}}_{M|N|P}\colon M\otimes_B(N\otimes_CP)\xrightarrow{\cong}(M\otimes_BN)\otimes_CP .
\]
\end{prop}

\begin{proof}
There are commutative squares in $\cC$
\[\begin{tikzcd}
	{((M \otimes B) \otimes N) \otimes (C \otimes P)} &&& {(M \otimes N) \otimes (C \otimes P)} \\
	{(M \otimes B) \otimes (N \otimes (C \otimes P))} &&& {M \otimes (N \otimes (C \otimes P))} \\
	{(M \otimes B) \otimes (N \otimes P)} &&& {M\otimes (N \otimes P)}
	\arrow["{(r_M \otimes N) \otimes (C \otimes P)}", shift left=2, from=1-1, to=1-4]
	\arrow["{\alpha_{M \otimes B,N,C \otimes P}}"', from=1-1, to=2-1]
	\arrow["{\alpha_{M,N, C \otimes P}}", from=1-4, to=2-4]
	\arrow["{(M \otimes B) \otimes (\alpha_{N,C,P}^{-1} \bullet r_N \otimes P)}"', from=2-1, to=3-1]
	\arrow["{M \otimes (\alpha_{N,C,P}^{-1} \bullet r_N \otimes P)}", from=2-4, to=3-4]
	\arrow["{r_M \otimes (N \otimes P)}"', from=3-1, to=3-4]
\end{tikzcd}\]

\bigskip 

\[\begin{tikzcd}
	{((M \otimes B) \otimes N) \otimes (C \otimes P)} &&& {(M \otimes N) \otimes (C \otimes P)} \\
	{(M \otimes B) \otimes (N \otimes (C \otimes P))} &&& {M \otimes (N \otimes (C \otimes P))} \\
	{(M \otimes B) \otimes (N \otimes P)} &&& {M\otimes (N \otimes P)}
	\arrow["{(r_M \otimes N) \otimes (C \otimes P)}", from=1-1, to=1-4]
	\arrow["{\alpha_{M \otimes B,N,C \otimes P}}"', from=1-1, to=2-1]
	\arrow["{\alpha_{M,N, C \otimes P}}", from=1-4, to=2-4]
	\arrow["{(M \otimes B) \otimes (N \otimes \ell_P)}"', from=2-1, to=3-1]
	\arrow["{M \otimes (N \otimes \ell_P)}", from=2-4, to=3-4]
	\arrow["{r_M \otimes (N \otimes P)}"', from=3-1, to=3-4]
\end{tikzcd}\]

\bigskip   

\[\begin{tikzcd}
	{((M \otimes B) \otimes N) \otimes (C \otimes P)} &&& {(M \otimes N) \otimes (C \otimes P)} \\
	{(M \otimes B) \otimes (N \otimes (C \otimes P))} &&& {M \otimes (N \otimes (C \otimes P))} \\
	{(M \otimes B) \otimes (N \otimes P)} &&& {M\otimes (N \otimes P)}
	\arrow["{(\alpha_{M,B,N} \bullet M \otimes \ell_N) \otimes (C \otimes P)}", from=1-1, to=1-4]
	\arrow["{\alpha_{M \otimes B,N,C \otimes P}}"', from=1-1, to=2-1]
	\arrow["{\alpha_{M,N, C \otimes P}}", from=1-4, to=2-4]
	\arrow["{(M \otimes B) \otimes (N \otimes \ell_P)}"', from=2-1, to=3-1]
	\arrow["{M \otimes (N \otimes \ell_P)}", from=2-4, to=3-4]
	\arrow["{\alpha_{M,B,N \otimes P} \bullet M \otimes (\alpha_{B,N,P}^{-1} \bullet \ell_N \otimes P)}"', from=3-1, to=3-4]
\end{tikzcd}\]

which commute by naturality of $\alpha$, pentagon axiom and bifunctoriality of $- \otimes -$, and
\[\begin{tikzcd}
	{((M \otimes B) \otimes N) \otimes (C \otimes P)} &&& {(M \otimes N) \otimes (C \otimes P)} \\
	{(M \otimes B) \otimes (N \otimes (C \otimes P))} &&& {M \otimes (N \otimes (C \otimes P))} \\
	{(M \otimes B) \otimes (N \otimes P)} &&& {M\otimes (N \otimes P)}
	\arrow["{(\alpha_{M,B,N} \bullet M \otimes \ell_N) \otimes (C \otimes P)}", from=1-1, to=1-4]
	\arrow["{\alpha_{M \otimes B,N,C \otimes P}}"', from=1-1, to=2-1]
	\arrow["{\alpha_{M,N, C \otimes P}}", from=1-4, to=2-4]
	\arrow["{(M \otimes B) \otimes (\alpha_{N,C,P}^{-1} \bullet r_N \otimes P)}"', from=2-1, to=3-1]
	\arrow["{M \otimes (\alpha_{N,C,P}^{-1} \bullet r_N \otimes P)}", from=2-4, to=3-4]
	\arrow["{\alpha_{M,B,N \otimes P} \bullet M \otimes (\alpha_{B,N,P}^{-1} \bullet \ell_N \otimes P)}"', from=3-1, to=3-4]
\end{tikzcd}\]
which commutes by $N$ being a $(B,C)$-bimodule, naturality of $\alpha$ and pentagon axiom.

Then, if $\mathcal{A}$ is the category with two objects and two parallel morphisms, these four diagrams determine the datum of a diagram $F\colon\cA\times\cA\to\cC$.
By Fubini's theorem for colimits (see \cite[\textsection IX.2]{MacLane}) we have
\[
\colim_{j\in\cA}\colim_{i\in\cA} F_{i,j}\cong\colim_{(i,j)\in\cA\times\cA} F_{i,j}\cong\colim_{i\in\cA}\colim_{j\in\cA} F_{i,j}.
\]
Following linear algebra convention ($i$ for row index and $j$ for column index), the object $\colim_{j\in\cA}\colim_{i\in\cA} F_{i,j}$ should be computed by taking the colimit of each column and then the colimit of the result, which gives $M\otimes_B(N\otimes_CP)$. Similarly, the object $\colim_{i\in\cA}\colim_{j\in\cA} F_{i,j}$ is computed by taking the colimit of each row and then the colimit of the result, which gives $(M\otimes_BN)\otimes_CP$.

So there is a canonical isomorphism induced by $\alpha_{M,N,P}$ in $\cC$:
\[
\overline\alpha_{M|N|C}\colon (M\otimes_BN)\otimes_CP\xrightarrow\cong M\otimes_B(N\otimes_CP),
\]
as desired. With a similar argument one can construct $\overline{\alpha^{-1}}_{M|N|P}$, as well as verify that $\overline{\alpha^{-1}}_{M|N|P}$ and $\overline{\alpha}_{M|N|P}$ are inverse to each other.
\end{proof}

\begin{prop}
\label{NaturalityAssociatorTensor}
Let $\cC^{\otimes}$ be a monoidal category which admits a calculus of balanced tensor products. Suppose we are given monoids $A$, $B$, $C$, $D$, a map of $(A,B)$-bimodules $\varphi\colon M\to M'$, a map of $(B,C)$-bimodules $\psi\colon N\to N'$, and a map of $(C,D)$-bimodules $\rho\colon P\to P'$. Then the following diagram commutes:
\[
\begin{tikzcd}
   (M\otimes_BN)\otimes_CP \arrow[r,"\overline\alpha_{M|N|P}"]\arrow[d,"(\varphi\otimes_B\psi)\otimes_C\rho" swap]&M\otimes_B(N\otimes_CP)\arrow[d,"\varphi\otimes_B(\psi\otimes_C\rho)"]\\
    (M'\otimes_BN')\otimes_CP'\arrow[r,"\overline\alpha_{M'|N'|P'}"swap]&M'\otimes_B(N'\otimes_CP')
\end{tikzcd}
\]
The same holds when considering the analog diagram in which any of the balanced tensor products ($\otimes_B$, $\otimes_C$) is replaced with an ordinary tensor product $\otimes$, and replacing $\overline\alpha$ appropriately.
\end{prop}

\begin{proof}
  Consider the diagram in $\cC$:
\[\begin{tikzcd}
	{(M \otimes N) \otimes P} && {M \otimes (N \otimes P)} \\
	\\
	{(M' \otimes N') \otimes P'} && {M' \otimes (N' \otimes P')} \\
	& {(M \otimes_B N) \otimes_C P} && {M \otimes_B (N \otimes_C P)} \\
	\\
	& {(M' \otimes_B N') \otimes_C P'} && {M' \otimes_B (N' \otimes_C P')}
	\arrow["{\alpha_{M,N,P}}", from=1-1, to=1-3]
	\arrow["{(\varphi\otimes\psi)\otimes\rho}"', from=1-1, to=3-1]
	\arrow["{(\pi_{M,N} \otimes P) \bullet \pi_{M\otimes_B N, P}}", from=1-1, to=4-2]
	\arrow["{\varphi\otimes(\psi\otimes\rho)}"{description}, from=1-3, to=3-3]
	\arrow["{(M \otimes \pi_{N,P}) \bullet \pi_{M, N \otimes_C P}}", from=1-3, to=4-4]
	\arrow["{\alpha_{M',N',P'}}"{description}, from=3-1, to=3-3]
	\arrow["{(\pi_{M',N'} \otimes P') \bullet \pi_{M'\otimes_B N', P'}}"{description, pos=0.2}, from=3-1, to=6-2]
	\arrow["{(M' \otimes \pi_{N',P'}) \bullet \pi_{M', N' \otimes_C P'}}"{description}, from=3-3, to=6-4]
	\arrow["{\overline{\alpha}_{M|N|P}}"{description}, from=4-2, to=4-4]
	\arrow["{(\varphi\otimes_B\psi)\otimes_C\rho}"{description, pos=0.3}, from=4-2, to=6-2]
	\arrow["{\varphi\otimes_B(\psi\otimes_C\rho)}"{description, pos=0.4}, from=4-4, to=6-4]
	\arrow["{\overline{\alpha}_{M'|N'|P'}}"', from=6-2, to=6-4]
\end{tikzcd}\]
The back square commutes because of the naturality of $\alpha$ on $(\varphi,\psi,\rho)$, the top square and bottom square commute by construction of $\overline\alpha$ (cf.~\cref{CompatibilityOverlineAlpha}), the the left and right squares commute by \cref{TensorFunctorial}.
By \cite[Proposition~1.3]{EH2},
the map $(\pi_{M,N} \otimes P) \bullet \pi_{M\otimes_B N, P}: (M \otimes N) \otimes P \to (M \otimes_B N) \otimes_C P$ is an epimorphism in $\cC$, so the front square commutes too, as desired.
\end{proof}

\begin{prop}
\label{PentagonTensor}
Let $\cC^\otimes$ be a monoidal category. Suppose we are given monoids $A,B,C,D,E$, an $(A,B)$-bimodule $M$, a $(B,C)$-bimodule $N$, a $(C,D)$-bimodule $P$ and a $(D,E)$-bimodule $Q$.
Then the following diagram commutes:
\[
\begin{tikzcd}
& (M\otimes_B N)\otimes_C (P\otimes_D Q)\arrow[ld,"\overline{\alpha}_{M|N|P\otimes_D Q}" swap] & \\
M \otimes_B (N \otimes_C (P\otimes_D Q)) & & ((M\otimes_B N)\otimes_C P)\otimes_D Q\arrow[ul,"\overline{\alpha}_{M\otimes_B N|P|Q}" swap]\arrow[d,"\overline{\alpha}_{M|N|P}\otimes_D Q"] \\
M\otimes_B ((N\otimes_C P)\otimes_D Q)\arrow[u,"M\otimes_B \overline{\alpha}_{N|P|Q}" ] & & (M\otimes_B (N\otimes_C P))\otimes_D Q\arrow[ll,"\overline{\alpha}_{M|N\otimes_C P|Q}"]
\end{tikzcd}
\]
The same holds when considering the analog diagram in which any of the balanced tensor products ($\otimes_B$, $\otimes_C$, $\otimes_D$) is replaced with an ordinary tensor product $\otimes$, and replacing $\overline\alpha$ appropriately.
\end{prop}

\begin{proof}
Consider the following diagram in $\cC$, which is built using (various iterations of) \cref{InducedAlpha2}:
\[
\adjustbox{width=\textwidth,center}{
\begin{tikzcd}
	& {(M \otimes N) \otimes (P \otimes Q)} & {((M\otimes N)\otimes P)\otimes Q} \\
	{M \otimes (N \otimes (P\otimes Q)) } \\
	&& {(M\otimes (N\otimes P))\otimes Q} \\
	{M\otimes ((N\otimes P)\otimes Q)} && {(M \otimes _B N) \otimes_C (P \otimes_D Q)} & {((M\otimes_B N)\otimes_C P)\otimes_D Q} \\
	& {M \otimes_B (N \otimes_C (P\otimes_D Q)) } \\
	&&& {(M\otimes_B (N\otimes_C P))\otimes_D Q} \\
	& {M\otimes_B ((N\otimes_C P)\otimes_D Q)}
	\arrow["{\alpha_{M,N,P \otimes Q}}"', from=1-2, to=2-1]
	\arrow[from=1-2, to=4-3]
	\arrow["{\alpha_{M \otimes N, P, Q}}"', from=1-3, to=1-2]
	\arrow["{\alpha_{M,N,P} \otimes Q}"', from=1-3, to=3-3]
	\arrow[from=1-3, to=4-4]
	\arrow[crossing over, from=2-1, to=5-2]
	\arrow["{\alpha_{M,N \otimes P, Q}}", from=3-3, to=4-1]
	\arrow[from=3-3, to=6-4]
	\arrow["{M \otimes \alpha_{N,P,Q}}", from=4-1, to=2-1]
	\arrow[from=4-1, to=7-2]
	\arrow["{\overline{\alpha}_{M|N|P \otimes_D Q}}"{description}, from=4-3, to=5-2]
	\arrow["{\overline{\alpha}_{M \otimes_B N| P| Q}}"', from=4-4, to=4-3]
	\arrow["{\overline{\alpha}_{M|N|P} \otimes Q}", from=4-4, to=6-4]
	\arrow["{\overline{\alpha}_{M|N \otimes_C P| Q}}", from=6-4, to=7-2]
	\arrow["{M \otimes \overline{\alpha}_{N|P|Q}}"', from=7-2, to=5-2]
\end{tikzcd}
}
\]

The back square commutes by \cref{MonoidalCategory}(2), and
all the ``back-to-front'' squares commute as a consequence of (various iterations of) \cref{NaturalityAssociatorTensor}.

By \cite[Proposition~1.3]{EH2},
the map
\[(\pi_{M,N} \otimes P) \otimes Q \bullet \pi_{M \otimes_{B} N,P} \otimes Q \bullet \pi_{(M \otimes_{B} N) \otimes_{C } P, Q} : ((M \otimes N) \otimes P) \otimes Q \to ((M \otimes_B N) \otimes_C P) \otimes_D Q\]
is an epimorphism in $\cC$, so the front square commutes too, as desired.
\end{proof}

\begin{prop}
\label{AssociatorAndUnitorTensor}
Let $\cC^\otimes$ be a monoidal category. Suppose we are given monoids $A,B,C$, an $(A,B)$-bimodule $M$ and a $(B,C)$-bimodule $N$.
Then the following diagrams commute
   \[
   \begin{tikzcd}
       (I\otimes M)\otimes_B N\arrow[rd,"\lambda_M\otimes_B N"swap]\arrow[r,"\overline{\alpha}_{I,M|N}"]&I\otimes (M\otimes_B N)\arrow[d,"\lambda_{M\otimes_B N}"]\\
       &M\otimes_B N
   \end{tikzcd}
   \quad
     \begin{tikzcd}
       (M\otimes_B N)\otimes I\arrow[d,"\rho_{M\otimes_B N}"swap]&M\otimes_B (N\otimes I)\arrow[ld,"M\otimes_B \rho_{N}"]\arrow[l,"\overline{\alpha^{-1}}_{M|N,I}"']\\
       M\otimes_B N&
   \end{tikzcd}
   \]
\end{prop}

\begin{proof}    
Consider the diagram in $\cC$.
\[\begin{tikzcd}[ampersand replacement=\&,cramped]
	{(I \otimes M) \otimes N} \&\& {I  \otimes (M \otimes N)} \\
	\&\& {M \otimes N} \\
	\& {(I \otimes M) \otimes_B N} \&\& {I \otimes (M \otimes_B N)} \\
	\&\&\& {M \otimes_BN}
	\arrow["{\alpha_{I,M,N}}", from=1-1, to=1-3]
	\arrow["{\lambda_M \otimes N}"', from=1-1, to=2-3]
	\arrow[from=1-1, to=3-2]
	\arrow["{\lambda_{M \otimes N}}"', from=1-3, to=2-3]
	\arrow[from=1-3, to=3-4]
	\arrow[{pos=0.3}, from=2-3, to=4-4]
	\arrow["{\overline{\alpha}_{I, M | N}}"{description},crossing over, from=3-2, to=3-4]
	\arrow["{\lambda_M \otimes_B N}"', from=3-2, to=4-4]
	\arrow["{\lambda_{M \otimes_B N}}", from=3-4, to=4-4]
\end{tikzcd}\]
   
The back triangle commutes by \cref{{AssociatorAndUnitor}}, the top square commute by \cref{CompatibilityOverlineAlpha}, the right square because of the naturality of $I \otimes -$, the bottom square \cref{CompatibilityTensorOnMaps}.

By \cite[Proposition~1.3]{EH2},
the map  $\pi_{I \otimes M,N} : (I \otimes M) \otimes N \rightarrow (I \otimes M) \otimes_{B} N$ is an epimorphism in $\cC$, so the front triangle commutes too, as desired.
\end{proof}

\subsection{Relevant bimodules}

\begin{prop}
\label{CompositeBimodule}
Let $\cC^{\otimes}$ be a monoidal category which admits a calculus of balanced tensor products. Given $A$, $B$, $C$ monoids, $M$ an $(A,B)$-bimodule and $N$ a $(B,C)$-bimodule, the maps from \cref{FreeModule}
induce maps in $\cC$    \[
   \ell_{M\otimes_BN}\colon A\otimes(M\otimes_BN)\to M\otimes_BN\quad\text{ and }\quad
   r_{M\otimes_BN}\colon (M\otimes_BN)\otimes C\to M\otimes_BN
    \]
   that endow $M\otimes_BN$ with an $(A,C)$-bimodule structure.
\end{prop}

\begin{proof}
\begin{enumerate}[leftmargin=*]
\item[(0)] We show that the maps $\ell_{M\otimes_BN}$ and $r_{M\otimes_BN}$ exist in $\cC$.
By \cref{TensorFunctorial,ActionIsmap}
we obtain maps in $\cC$
    \[
    \ell_M\otimes_BN\colon (A\otimes M)\otimes_B N\longrightarrow M\otimes_BN
\quad \text{ and }\quad
    M\otimes_B r_M\colon M\otimes_B(N\otimes C)\longrightarrow M\otimes_BN.
    \]
After precomposing with the isomorphisms $\overline{\alpha^{-1}}_{A,M|N}$
and $\overline\alpha_{M|N,C}$ from \cref{Axioms}, we obtain maps in $\cC$
    \[
    \ell_{M\otimes_BN}\colon A\otimes (M\otimes_BN)\longrightarrow M\otimes_BN
\quad \text{ and }\quad
    r_{M\otimes_BN}\colon(M\otimes_BN)\otimes C\longrightarrow M\otimes_BN.
    \]
\item We show associativity of the left action $\ell_{M\otimes_BN}$. To this end, we have to show that the following diagram commutes in $\cC$:
\[
\begin{tikzcd}
(A\otimes A)\otimes(M\otimes_B N)\arrow[rr,"m_A\otimes(M\otimes_BN)"]\arrow[d,"\alpha_{A,A,M\otimes_BN}\bullet (A\otimes\ell_{M\otimes_BN})" swap]&&A\otimes (M\otimes_B N)\arrow[d,"\ell_{M\otimes_BN}"]\\
A\otimes (M\otimes_BN)\arrow[rr,"\ell_{M\otimes_BN}" swap]&&M\otimes_B N 
\end{tikzcd}
\]
Using \cref{NaturalityAssociatorTensor,PentagonTensor},
this can be rewritten as
\[
\begin{tikzcd}
((A\otimes A)\otimes M)\otimes_B N\arrow[rr, "(m_A \otimes M) \otimes_B N"]\arrow[d,"(\alpha_{A,A,M}\bullet A\otimes\ell_M)\otimes_B N" swap]&&(A\otimes M)\otimes_B N\arrow[d, "\ell_M \otimes_B N"]\\
(A\otimes M)\otimes_B N\arrow[rr,"\ell_M\otimes_B N" swap]&&M\otimes_B N 
\end{tikzcd}
\]
By \cref{RmkEpi}, it is enough to check that the following diagram commutes in $\cC$:
\[
\begin{tikzcd}
(A\otimes A)\otimes M\arrow[rr, "m_A \otimes M"]\arrow[d,"\alpha_{A,A,M}\bullet\ell_M" swap]&&A\otimes M\arrow[d, "\ell_M "]\\
A\otimes M\arrow[rr,"\ell_M" swap]&&M
\end{tikzcd}
\]
This diagram does commute in $\cC$ because $M$ is an $(A,B)$-bimodule.
Similarly, we obtain that the following diagram commutes in $\cC$:
\[
\begin{tikzcd}
(M\otimes_B N)\otimes (C\otimes C)\arrow[rr,"(M\otimes_BN)\otimes m_C"]\arrow[d,"\alpha^{-1}_{M\otimes_BN,C,C}\bullet((M\otimes_BN)\otimes m_C)" swap]&& (M\otimes_B N)\otimes C\arrow[d,"r_{M\otimes_BN}"]\\
(M\otimes_B N)\otimes C\arrow[rr,"r_{M\otimes_BN}" swap]&&M\otimes_B N 
\end{tikzcd}
\]
so the right action $r_{M\otimes_BN}$ is associative, too.
\item
We show that the actions $\ell_{M\otimes_BN}$ and $r_{M\otimes_BN}$ are compatible. To this end, we need to show that the following diagram commutes in $\cC$:
\[
\begin{tikzcd}
(A\otimes (M\otimes_B N))\otimes C\arrow[rr,"\ell_{M\otimes_BN}\otimes C"]\arrow[d,"\alpha_{A,M\otimes_BN,C}\bullet(A\otimes r_{M\otimes_BN})"swap]&& (M\otimes_B N)\otimes C\arrow[d,"r_{M\otimes_BN}"]\\
A\otimes (M\otimes_B N)\arrow[rr,"\ell_{M\otimes_BN}"swap]&&M\otimes_B N 
\end{tikzcd}
\]
By \cref{PentagonTensor,NaturalityAssociatorTensor}, it is enough to check that the following diagram, built using \cref{Axioms}, commutes in $\cC$:
\[\begin{tikzcd}
(A\otimes M)\otimes_B (N\otimes C)\arrow[rr,"\ell_M\otimes_B(N\otimes C)"]\arrow[d,"(A\otimes M)\otimes_Br_M"swap]&& M\otimes_B (N\otimes C)\arrow[d,"M\otimes_Br_N"]\\
(A\otimes M)\otimes_B N\arrow[rr,"\ell_M\otimes_BN"swap]&&M\otimes_B N 
\end{tikzcd}
\]
By \cref{CompatibilityTensorOnMaps}, it is enough to check that the following diagram commutes in $\cC$:
\[
\begin{tikzcd}
(A\otimes M)\otimes (N\otimes C)\arrow[rr,"\ell_M\otimes(N\otimes C)"]\arrow[d,"(A\otimes M)\otimes r_M"swap]&& M\otimes (N\otimes C)\arrow[d,"M\otimes r_N"]\\
(A\otimes M)\otimes N\arrow[rr,"\ell_M\otimes N"swap]&&M\otimes N 
\end{tikzcd}
\]
and this commutes because $\otimes$ is a bifunctor,
so the actions are compatible.
\item
We show that the left action $\ell_{M\otimes_BN}$ is unital.
To this end, we need to show that the following diagram commutes in $\cC$:
\[
\begin{tikzcd}
I\otimes(M\otimes_BN)\arrow[rr,"e_A\otimes(M\otimes_BN)"]\arrow[d,"\lambda_{M\otimes_BN}" swap]&&A\otimes(M\otimes_BN)\arrow[lld,"\ell_{M\otimes_BN}"]\\
M\otimes_BN
\end{tikzcd}
\]
By \cref{AssociatorAndUnitorTensor,NaturalityAssociatorTensor}, it is enough to check that the following diagram, built using \cref{Axioms}, commutes in $\cC$:
\[
\begin{tikzcd}
(I\otimes M)\otimes_BN\arrow[rr,"(e_A\otimes_B M)\otimes_B N"]\arrow[d,"\lambda_M\otimes_B N" swap]&&(A\otimes M)\otimes_BN\arrow[lld,"\ell_M\otimes_B N"]\\
M\otimes_BN
\end{tikzcd}
\]
By \cref{RmkEpi}, it is enough to check that the following diagram commutes in $\cC$:
\[
\begin{tikzcd}
I\otimes M\arrow[rr,"e_A\otimes M"]\arrow[d,"\lambda_M" swap]&&A\otimes M\arrow[lld, "\ell_M "]\\
M
\end{tikzcd}
\]
and this commutes because $M$ is an $A$-module, so the left action $\ell_{M\otimes_BN}$ is unital. One can similarly show that the right action $r_{M\otimes_BN}$ is also unital, concluding the proof.\qedhere
\end{enumerate}
\end{proof}

\begin{prop}
\label{UnitalityForBimodules}
Let $\cC^{\otimes}$ be a monoidal category which admits a calculus of balanced tensor products.
Given monoids $A$ and $B$ and $M$ an $(A,B)$-bimodule, the left and right actions $\ell_M$ and $r_M$
induce isomorphisms of $(A,B)$-bimodules
in $\cC$:
    \[
    \overline{\ell}_M\colon A\otimes_A M\xrightarrow{\cong}M\quad\text{ and }\quad
    \overline r_M\colon M\otimes_BB\xrightarrow{\cong}M.
    \]
\end{prop}

\begin{proof}
We first check that the maps exist and are isomorphisms in $\cC$. For this, we claim that the diagram
\[
 \begin{tikzcd}
    {(A \otimes A) \otimes M} && {A \otimes M}
    \arrow["{\alpha_{A,A,M}\bullet (A\otimes\ell_M)}"', shift right=2, from=1-1, to=1-3]
    \arrow["{m_A\otimes M}", shift left=2, from=1-1, to=1-3]\arrow[r,"\ell_M"]&M
    \end{tikzcd}
\]
can be completed to a split coequalizer (in the sense of \cite[\textsection IV.6]{MacLane}) with the maps
\[
s_M\colon M\xrightarrow{\lambda_M^{-1}}I\otimes M\xrightarrow{e_A\otimes M}A\otimes M\quad\text{ and }
\quad
\]
\[
s'_M\colon A\otimes M\xrightarrow{A\otimes \lambda_M^{-1}}A\otimes (I\otimes M)\xrightarrow{A\otimes(e_A\otimes M)}A\otimes (A\otimes M)\xrightarrow{\alpha^{-1}_{A,A,M}}(A\otimes A)\otimes M.
\]
The fact that this is indeed a split coequalizer amounts to arguing that we have the following four commutative diagrams:
\begin{enumerate}[leftmargin=*]
    \item[(0)]
This diagram (showing that $\ell_M$ coequalizes the two parallel maps)
\[
\begin{tikzcd}
(A\otimes A)\otimes M\arrow[r,"m_A\otimes M"]\arrow[d,"\alpha_{A,A,M}\bullet (A\otimes\ell_M)" swap]&A\otimes M\arrow[d,"\ell_M"]\\
A\otimes M\arrow[r,"\ell_M" swap]&M
\end{tikzcd}
\]
commutes because the left action of $M$ is associative.
\item This diagram (showing $s_M$ is a section of $\ell_M$)
\[
\begin{tikzcd}
M\arrow[r,"\lambda_M^{-1}"]\arrow[d,"\id_M" swap]
&I\otimes M\arrow[d,"e_A\otimes M"]\\
M&A\otimes M\arrow[l,"\ell_M"]
\end{tikzcd}
\]
commutes because the left action of $M$ is unital.
\item This diagram (showing $s'_M$ is a section of $\alpha_{A,A,M}\bullet (A\otimes\ell_M)$)
\[
\begin{tikzcd}
A\otimes M\arrow[rr,"A\otimes \lambda_M^{-1}"]\arrow[d,"A\otimes\id_{M}" swap]
&&A\otimes (I\otimes M)\arrow[d,"(A\otimes (e_A\otimes M))\bullet\alpha^{-1}_{A,A,M}"]\\
A\otimes M&&(A\otimes A)\otimes M\arrow[ll,"(A\otimes\ell_M)\bullet \alpha_{A,A,M}"]
\end{tikzcd}\]
can be replaced with the diagram
\[
\begin{tikzcd}
A\otimes M\arrow[rr,"A\otimes \lambda_M^{-1}"]\arrow[d,"A\otimes\id_{M}" swap]
&&A\otimes (I\otimes M)\arrow[d,"A\otimes (e_A\otimes M)"]\\
A\otimes M&&A\otimes (A\otimes M)\arrow[ll,"A\otimes\ell_M"]
\end{tikzcd}\]
which commutes because the left action $\ell_M$ is unital.
\item This diagram
(showing $s'_M$ is a section of $m_A\otimes M$)
\[
\begin{tikzcd}
A\otimes M\arrow[rr,"A\otimes \lambda_M^{-1}"]\arrow[d,"A\otimes\id_{M}" swap]
&&A\otimes (I\otimes M)\arrow[d,"(A\otimes (e_A\otimes M))\bullet\alpha^{-1}_{A,A,M}"]\\
A\otimes M&&(A\otimes A)\otimes M\arrow[ll,"m_A\otimes_M"]
\end{tikzcd}\]
can be replaced by \cref{AssociatorAndUnitorTensor} with the diagram
\[
\begin{tikzcd}
A\otimes M\arrow[rr,"\lambda_A^{-1}\otimes M"]\arrow[d,"\id_{A\otimes M}" swap]
&&(A\otimes I)\otimes M\arrow[d,"(A\otimes e_A)\otimes M"]\\
A\otimes M&&(A\otimes A)\otimes M\arrow[ll,"m_A\otimes M"]
\end{tikzcd}
\]
which commutes because the multiplication of $A$ is unital.
\end{enumerate}
So $\ell_M$ gives a split coequalizer.
It now follows by \cite[\textsection IV.6]{MacLane} that $\ell_M$ induces an isomorphism in $\cC$
\[
\overline \ell_M\colon A\otimes_AM\xrightarrow{\cong} M.
\]
We now show that $\overline\ell_M$ is a map of $(A,B)$-bimodules. In order to show that $\overline\ell_M$ is compatible with the left actions we have to prove that the following diagram commutes in $\cC$:
\[
\begin{tikzcd}
    A\otimes(A\otimes_AM)\arrow[r,"A\otimes\overline \ell_{M}"]\arrow[d,"\ell_{A\otimes_AM}"swap]&A\otimes M\arrow[d,"\ell_M"]\\
    A\otimes_AM\arrow[r,"\overline\ell_M" swap]&M
\end{tikzcd}
\]
By \cref{Axioms},
it suffices to prove that the following diagram commutes:
\[\begin{tikzcd}
    (A\otimes A)\otimes_AM\arrow[rrr,"\overline\alpha_{A,A|M}\bullet(A\otimes\overline \ell_{M})"]\arrow[d,"\overline\alpha_{A,A|M}\bullet(\ell_{A\otimes_AM})"swap]&&&A\otimes M\arrow[d,"\ell_M"]\\
    A\otimes_AM\arrow[rrr,"\overline\ell_M"swap]&&&M
\end{tikzcd}
\]
By \cref{CompatibilityTensorOnMaps,ActionIsmap},
it suffices to prove that the following diagram commutes:
\[
\begin{tikzcd}
    (A\otimes A)\otimes M\arrow[rrr,"\alpha_{A,A,M}\bullet (A\otimes\ell_M)"]\arrow[d,"\alpha_{A,A,M}\bullet(\ell_{A\otimes M})=m_A\otimes M"swap]&&&A\otimes M\arrow[d,"\ell_M"]\\
    A\otimes M\arrow[rrr,"\ell_M"swap]&&&M
\end{tikzcd}
\]
This diagram commutes because the left action of $A$ on $M$ is associative.

 In order to show that $\overline\ell_M$ is compatible with the right actions we have to prove that the following diagram commutes in $\cC$:
\[
\begin{tikzcd}
    (A\otimes_AM)\otimes B\arrow[r,"\overline \ell_{M}\otimes B"]\arrow[d,"r_{A\otimes_AM}"swap]&M\otimes B\arrow[d,"r_M"]\\
    A\otimes_AM\arrow[r,"\overline\ell_M" swap]&M
\end{tikzcd}
\]
By \cref{CompositeBimodule,Axioms},
it suffices to prove that the following diagram commutes:
\[
\begin{tikzcd}
    A\otimes_A(M\otimes B)\arrow[rrr,"\overline{\alpha^{-1}}_{A|M,B}\bullet(\overline \ell_{M}\otimes B)"]\arrow[d,"A\otimes_Ar_M"swap]&&&M\otimes B\arrow[d,"r_M"]\\
    A\otimes_AM\arrow[rrr,"\overline\ell_M" swap]&&&M
\end{tikzcd}
\]
By \cref{CompatibilityTensorOnMaps,ActionIsmap},
it suffices to prove that the following diagram commutes:
\[
\begin{tikzcd}
    A\otimes(M\otimes B)\arrow[rrr,"\alpha^{-1}_{A,M,B}\bullet( \ell_{M}\otimes B)"]\arrow[d,"A\otimes r_M"swap]&&&M\otimes B\arrow[d,"r_M"]\\
    A\otimes M\arrow[rrr,"\ell_M" swap]&&&M
\end{tikzcd}
\]
This diagram commutes because $M$ is an $(A,B)$-bimodule.
\end{proof}

\begin{prop}
\label{AssociativityForBimodules}
Let $\cC^{\otimes}$ be a monoidal category which admits a calculus of balanced tensor products. Given monoids $A$, $B$, $C$, $D$, and $M$ an $(A,B)$-bimodule, $N$ a $(B,C)$-bimodule, and $P$ a $(C,D)$-bimodule, the map $\alpha_{M,N,P}$
induces a bimodule isomorphism in $\cC$
\[
\overline\alpha_{M|N|P}\colon(M \otimes_BN) \otimes_CP\xrightarrow{\cong}M \otimes_B (N \otimes_CP).
\]  
\end{prop}

\begin{proof}
The existence was shown in \cref{InducedAlpha2}.
We need to show it's a map of bimodules. Let's look at the right action. Need
\[
\begin{tikzcd}
    ((M\otimes_BN)\otimes_C P)\otimes D\arrow[rr,"\overline\alpha_{M|N|P}\otimes D"]\arrow[d,"r_{(M\otimes_BN)\otimes_CP}"swap]&&(M\otimes_B(N\otimes_C P))\otimes D\arrow[d,"r_{M\otimes_B(N\otimes_CP)}"]\\
    (M\otimes_BN)\otimes_CP\arrow[rr,"\overline\alpha_{M|N|P}"swap]&& M\otimes_B(N\otimes_CP)
\end{tikzcd}
\]
By \cref{CompatibilityTensorOnMaps},
it suffices to prove that the following diagram commutes:
\[
\begin{tikzcd}
    ((M\otimes_BN)\otimes P)\otimes D\arrow[rr,"\overline\alpha_{M|N,P}\otimes D"]\arrow[d,"r_{M\otimes_B(N\otimes P)}"swap]&&(M\otimes_B(N\otimes P))\otimes D\arrow[d,"r_{(M\otimes_BN)\otimes P}"]\\
    (M\otimes_BN)\otimes P\arrow[rr,"\overline\alpha_{M|N,P}"swap]&& M\otimes_B(N\otimes P)
\end{tikzcd}
\]
By \cref{CompatibilityTensorOnMaps}
it suffices to prove that the following diagram commutes:
\[
\begin{tikzcd}
    ((M\otimes N)\otimes P)\otimes D\arrow[rr,"\alpha_{M,N,P}\otimes D"]\arrow[d,"r_{M\otimes (N\otimes P)}"swap]&&(M\otimes(N\otimes P))\otimes D\arrow[d,"r_{(M\otimes N)\otimes P}"]\\
    (M\otimes N)\otimes P\arrow[rr,"\alpha_{M,N,P}"swap]&& M\otimes(N\otimes P)
\end{tikzcd}
\]
By \cref{FreeModule},
it suffices to prove that the following diagram commutes:
\[\begin{tikzcd}
{(M \otimes N) \otimes (P \otimes D)} &&&&& {M \otimes ((N \otimes P) \otimes D)} \\
{(M \otimes N) \otimes P} &&&&& {M \otimes (N \otimes P)}
\arrow["{\alpha_{M\otimes N, P, D} \bullet (\alpha_{M,N,P} \otimes D) \bullet \alpha_{M, N \otimes P,D} }", from=1-1, to=1-6]
\arrow["{(M \otimes N) \otimes r_P}"', from=1-1, to=2-1]
\arrow["{M \otimes r_{N \otimes P}}", from=1-6, to=2-6]
\arrow["{\alpha_{M,N,P}}"', from=2-1, to=2-6]
\end{tikzcd}\]
By \cref{FreeModule} and \cref{MonoidalCategory}(2),
it suffices to prove that the following diagram commutes:
\[
\begin{tikzcd}
    (M\otimes N)\otimes (P\otimes D)\arrow[rr,"\alpha_{M,N,P\otimes D}"]\arrow[d,"(M\otimes N)\otimes r_P"swap]&&M\otimes(N\otimes (P\otimes D))\arrow[d,"M\otimes (N\otimes r_P)"]\\
    (M\otimes N)\otimes P\arrow[rr,"\alpha_{M,N,P}"swap]&& M\otimes(N\otimes P)
\end{tikzcd}
\]
And this commutes by naturality of the associator on $r_P$.
The left action can be treated similarly.
\end{proof}

\subsection{Relevant bimodule maps}

In presence of well-behaved balanced tensor products, we can perform various constructions on bimodules.

\begin{rmk}
We'll often evoke a graphical representation in which:
\begin{enumerate}[leftmargin=*]
    \item monoids are represented as nodes,
    \[
    A
    \]
    \item bimodules are represented as $1$-dimensional arrows
    \[
    A\xrightarrow{M}B
    \]
    \item and bimodule maps are represented as $2$-dimensional arrows
\[\begin{tikzcd}
	A && B
	\arrow[""{name=0, anchor=center, inner sep=0}, "M", curve={height=-24pt}, from=1-1, to=1-3]
	\arrow[""{name=1, anchor=center, inner sep=0}, "N"', curve={height=24pt}, from=1-1, to=1-3]
	\arrow["\varphi", shorten <=6pt, shorten >=6pt, Rightarrow, from=0, to=1]
\end{tikzcd}\]
\end{enumerate}
This is of course a reminiscence of the fact that these should form a bicategory. While this is likely true, we do not prove it nor make use of anything else other than what we prove in this section.
\end{rmk}

The following construction essentially makes sense of configuration of the form
\[\begin{tikzcd}
	A && B && C \\
	&&& {}
	\arrow["{{M}}", from=1-1, to=1-3]
	\arrow[""{name=0, anchor=center, inner sep=0}, "N", curve={height=-24pt}, from=1-3, to=1-5]
	\arrow[""{name=1, anchor=center, inner sep=0}, "{{N'}}"', curve={height=24pt}, from=1-3, to=1-5]
	\arrow["\psi", shorten <=6pt, shorten >=6pt, Rightarrow, from=0, to=1]
\end{tikzcd}\]

\begin{prop}
Given monoids $A$, $B$, $C$, and $M$ and $(A,B)$-bimodule, $N$, $N'$ two $(B,C)$-bimodules, and $\psi\colon N\to N'$ a $(B,C)$-bimodule homomophism, the map
\[
M\otimes_B\psi\colon M\otimes_BN\to M\otimes_BN'
\]
from \cref{TensorFunctorial} is an $(A,C)$-bimodule map.
\end{prop}

\begin{proof}
In order to show that the proposed left action is associative, we need to show that the following diagram commutes:
\[
\begin{tikzcd}
    A\otimes (M\otimes_BN)\arrow[rr,"A\otimes (M\otimes_B\psi)"]\arrow[d,"\ell_{M\otimes_BN}"swap]&&A\otimes (M\otimes_BN')\arrow[d,"\ell_{M\otimes_BN'}"]\\
    M\otimes_BN\arrow[rr,"M\otimes_B\psi"swap]&&M\otimes_BN'
\end{tikzcd}
\]
By \cref{CompositeBimodule,Axioms},
it suffices to prove that the following diagram commutes:
\[
\begin{tikzcd}
    (A\otimes M)\otimes_BN
    \arrow[rr,"(A\otimes M)\otimes_B\psi"]
    \arrow[d,"\ell_M\otimes_BN"swap]
    &&
    (A\otimes M)\otimes_BN'
    \arrow[d,"\ell_M\otimes_BN'"]
    \\
M\otimes_BN
\arrow[rr,"M\otimes_B\psi"swap]
&&M\otimes_BN'
\end{tikzcd}
\]
By \cref{CompatibilityTensorOnMaps},
it suffices to prove that the following diagram commutes:
\[
\begin{tikzcd}
    (A\otimes M)\otimes N\arrow[rr,"(A\otimes M)\otimes\psi"]\arrow[d,"\ell_M\otimes N"swap]&&
    (A\otimes M)\otimes N'\arrow[d,"\ell_M\otimes N'"]\\
    M\otimes N\arrow[rr,"M\otimes\psi" swap]&&M\otimes N'
\end{tikzcd}
\]
This diagram commutes because $\otimes$ is a bifunctor.

In order to show that the proposed right action is associative, we need to show that the following diagram commutes:
\[
\begin{tikzcd} (M\otimes_BN)\otimes C\arrow[rr,"(M\otimes_B\psi)\otimes C"]\arrow[d,"r_{M\otimes_BN}"swap]&&(M\otimes_BN')\otimes C\arrow[d,"r_{M\otimes_BN"}"]\\
    M\otimes_BN\arrow[rr,"M\otimes_B\psi"swap]&&M\otimes_BN'
\end{tikzcd}
\]
By \cref{CompositeBimodule,Axioms},
it suffices to prove that the following diagram commutes:
\[
\begin{tikzcd} M\otimes_B(N\otimes C)\arrow[rr,"M\otimes_B(\psi\otimes C)"]\arrow[d,"M\otimes_Br_N"swap]&&M\otimes_B(N'\otimes C)\arrow[d,"M\otimes_B r_{N'}"]\\
    M\otimes_BN\arrow[rr,"M\otimes_B\psi"swap]&&M\otimes_BN'
\end{tikzcd}
\]
By \cref{CompatibilityTensorOnMaps},
it suffices to prove that the following diagram commutes:
\[
\begin{tikzcd} M\otimes(N\otimes C)\arrow[rr,"M\otimes(\psi\otimes C)"]\arrow[d,"M\otimes r_N"swap]&&M\otimes(N'\otimes C)\arrow[d,"M\otimes r_{N'}"]\\
    M\otimes N\arrow[rr,"M\otimes\psi" swap]&&M\otimes N'
\end{tikzcd}
\]
This diagram commutes because $M\otimes(-)$ is a functor.
\end{proof}

The following proposition essentially guarantees that, given a configuration of the form
\[\begin{tikzcd}
	A && B && C
	\arrow["{{M}}", from=1-1, to=1-3]
	\arrow[""{name=0, anchor=center, inner sep=0}, "{{N'}}"{description}, from=1-3, to=1-5]
	\arrow[""{name=1, anchor=center, inner sep=0}, "N", curve={height=-24pt}, from=1-3, to=1-5]
	\arrow[""{name=2, anchor=center, inner sep=0}, "{{N''}}"', curve={height=24pt}, from=1-3, to=1-5]
	\arrow["\gamma", shorten <=3pt, shorten >=3pt, Rightarrow, from=1, to=0]
	\arrow["{\gamma'}", shorten <=3pt, shorten >=3pt, Rightarrow, from=0, to=2]
\end{tikzcd}\]
there is a well defined composite.

\begin{rmk}
\label{TensorWithIso}
Let $\cC^{\otimes}$ be a monoidal category which admits a calculus of balanced tensor products. Given monoids $A$, $B$, $C$, an $(A,B)$-bimodule $M$ and a map of $(B,C)$-bimodules $\gamma\colon N\to N'$, we have
\[
M\otimes_B\id_{N}=\id_{M\otimes_BN}.
\]
By \cref{RmkEpi}, if $\gamma$ is an isomorphism then $M\otimes_B\gamma$ is an isomorphism and
\[
(M\otimes_B\gamma)^{-1}=M\otimes_B\gamma^{-1}.
\]
\end{rmk}

The following proposition essentially guarantees that, given a configuration of the form
\[\begin{tikzcd}[cramped]
	{A} & {B} & {C}
	\arrow[""{name=0, anchor=center, inner sep=0}, "{M'}"', curve={height=12pt}, from=1-1, to=1-2]
	\arrow[""{name=1, anchor=center, inner sep=0}, "M", curve={height=-12pt}, from=1-1, to=1-2]
	\arrow[""{name=2, anchor=center, inner sep=0}, "{N'}"', curve={height=12pt}, from=1-2, to=1-3]
	\arrow[""{name=3, anchor=center, inner sep=0}, "N", curve={height=-12pt}, from=1-2, to=1-3]
	\arrow["\varphi"', shorten <=3pt, shorten >=3pt, Rightarrow, from=1, to=0]
	\arrow["\psi"', shorten <=3pt, shorten >=3pt, Rightarrow, from=3, to=2]
\end{tikzcd}\]
there is a well defined composite.

\begin{prop}
\label{3d-interchange}
\label{pastingII}
Let $\cC^{\otimes}$ be a monoidal category which admits a calculus of balanced tensor products. Given monoids $A,B,C$, bimodules $M$, $M'$ two $(A,B)$-bimodule, $N$, $N'$, two $(B,C)$-bimodules, and bimodule maps nd $(A,B)$-bimodule map $\varphi\colon M\to M'$ and $\psi\colon N\to N'$, there is an equality of $(A,C)$-bimodule maps from $M\otimes_BN$ to $M'\otimes_BN'$:
\[(\varphi\otimes_{B} N)\bullet(M'\otimes_B\psi)= \varphi\otimes_B\psi=(M\otimes_B\psi)\bullet(\varphi\otimes_BN').
\]
\end{prop}

\begin{proof}
There are commutative diagrams in $\cC$
\[
\begin{tikzcd}
    M\otimes N\arrow[rd,"\varphi\otimes N" swap]\arrow[rr,"\varphi\otimes\psi"]&&M'\otimes N'\\
    &M'\otimes N\arrow[ru,"M'\otimes\psi" swap]&
\end{tikzcd}
\rightsquigarrow
\begin{tikzcd}
    M\otimes_B N\arrow[rd,"\varphi\otimes_B N" swap]\arrow[rr,"\varphi\otimes_B\psi"]&&M'\otimes_B N'\\
    &M'\otimes_B N\arrow[ru,"M'\otimes_B\psi" swap]&
\end{tikzcd}
\]
so first equality follows. The second equality can be treated similarly.
\end{proof}

We will use these two lemmas later:

\begin{lem}[Reduction I]
\label{ReductionI}
Let $\cC^{\otimes}$ be a monoidal category which admits a calculus of balanced tensor products.
Given monoids $A,B$, and $(A,B)$-bimodules $M, M'$ and a bimodule map $\varphi\colon M\to M'$,
we have an equality of $(A,B)$-bimodule maps
\[ \varphi = \overline r_{M}^{-1} \bullet (\varphi \otimes_B B) \bullet \overline r_{M'}\]
In particular, $\varphi$ is determined by $\varphi \otimes_B B$ and $\varphi$ is an isomorphism if and only if $\varphi \otimes_B B$ is.
\end{lem}

\begin{proof}
Since $\varphi$ is compatible with the $B$ action we have
\[
\begin{tikzcd}
M\otimes B\arrow[r,"r_{M}"]\arrow[d,"\varphi\otimes B" swap]&M\arrow[d,"\varphi"]\\
    M'\otimes B\arrow[r,"r_{M'}" swap]&M'
\end{tikzcd}
\quad\quad
\rightsquigarrow
\quad\quad
\begin{tikzcd}
 M\otimes_BB\arrow[r,"\overline r_{M}"]\arrow[d,"\varphi \otimes_B B" swap]&M\arrow[d,"\varphi"]\\
    M'\otimes_B B\arrow[r,"\overline r_{M'}" swap]&M'
\end{tikzcd}
\]
We are using that $\overline r_{M'}$ is an isomorphism by \cref{UnitalityForBimodules}.
\end{proof}

\begin{lem}[Reduction II]
\label{ReductionII}
Given monoids $A,B$, and $(A,B)$-bimodules $M, M'$ and bimodule maps $\varphi\colon M\to M'$, and $P$ an $(A,A)$-bimodule and $\epsilon\colon P\cong A$ an $(A,A)$-bimodule isomorphism
have equality of $(A,B)$-bimodule maps from $A\otimes_A M$ to $A\otimes_A M'$
\[
A\otimes_A\varphi=(\epsilon\otimes_AM)^{-1}\bullet(P\otimes_A\varphi)\bullet(\epsilon\otimes_AM').
\]
In particular, $A\otimes_A\varphi$ is determined by $P\otimes_A\varphi$ and $A\otimes_A\varphi$ is an isomorphism if and only if $P\otimes_A\varphi$ is.
\end{lem}

\begin{proof}
As an instance of \cref{pastingII} (twice) and some associator stuff should get equality of $(A,B)$-bimodule maps from $P\otimes_AM$ to $A\otimes_AM'$
\[
(\epsilon\otimes_AM)\bullet(A\otimes_A\varphi)=\epsilon\otimes_A\varphi=(P\otimes_A\varphi)\bullet(\epsilon\otimes_AM').
\]
By \cref{TensorWithIso},
$\epsilon\otimes_AM$ is an isomorphism, so the claim follows.
\end{proof}

\section{Nerve of a monoidal category}

\subsection{Nerve of a monoidal category}

Recall from that a simplicial set $X$ is \emph{$3$-coskeletal} if it is isomorphic it its $3$-skeleton; that is, for every integer $m\geq4$, every morphism of simplicial sets $\partial\Delta[m]\to X$ extends uniquely to an $m$-simplex of $X$. Every $3$-truncated simplicial set extends uniquely to a $3$-coskeletal simplicial set; see e.g.~\cite[\textsection3.2]{GJ}.

\begin{const}
\label{MonoidalNerve}
Let $\cC^{\otimes}$ be a monoidal category. The \emph{monoidal nerve} $\mrt(\cC^{\otimes})$ of $\cC^{\otimes}$ is the $3$-coskeletal simplicial set in which:
\begin{enumerate}[leftmargin=*]
\item[(0)] The set $\mrt_0(\cC^{\otimes})$ of $0$-simplices is given by the set of monoids in $\cC^{\otimes}$:
\[
\mrt_0(\cC^{\otimes})\coloneqq\{A\ |\ A\text{ monoid in }\cC^{\otimes}\}.
\]
\item The set $\mrt_1(\cC^{\otimes})$ of $1$-simplices is given by the set of bimodules in $\cC^{\otimes}$:
\[
\mrt_1(\cC^{\otimes})\coloneqq\{(A_0,A_1|M_{01})\ |\ A,B\text{ monoids in }\cC^{\otimes}\text{ and }M_{01}\text{ an $(A_0,A_1)$-bimodule in $\cC^{\otimes}$}\}.
\]
A $1$-simplex $(A_0,A_1|M_{01})$ will be depicted as
\[\begin{tikzcd}
	A_0 \arrow[r,"M_{01}"]& A_1,
\end{tikzcd}\]
indicating that its $0$-th and $1$-st faces are given, respectively, by $B$ and $A$.
Given a $0$-simplex $A_0$, its $0$-th degeneracy is the $1$-simplex
\[\begin{tikzcd}
	A_0 \arrow[r,"A_{0}"]& A_0
\end{tikzcd}\]
given by $A_0$ with the $(A_0,A_0)$-bimodule structure from \cref{IdentityBimodule}.
\item The set $\mrt_2(\cC^{\otimes})$ of $2$-simplices is given by bimodule maps with a decomposition of the domain bimodule as a balanced tensor product:
\[
\mrt_2(\cC^{\otimes})\coloneqq\{(A_0, A_1,A_2|M_{01},M_{12},M_{23}|\varphi)\ 
|\ \text{$A_i$ is monoid, $M_{ij}$ is $(A_i,A_j)$-bimodule},\]
\[
\quad\quad\quad\quad\quad\quad\quad\quad\quad\quad\quad\quad\quad\quad\quad\quad\quad\quad\text{$\varphi\colon M_{01}\otimes_{A_1}M_{12}\to M_{02}$ is $(A_0,A_2)$-bimodule map}\}.
\]
A $2$-simplex $(A_0, A_1,A_2|M_{01},M_{12},M_{23}|\varphi)$ will be depicted as
\[\begin{tikzcd}
	& {A_1} \\
	{A_0} && {A_2}
	\arrow["{M_{12}}", from=1-2, to=2-3]
	\arrow["{M_{01}}", from=2-1, to=1-2]
	\arrow[""{name=0, anchor=center, inner sep=0}, "{M_{02}}"', from=2-1, to=2-3]
	\arrow["\varphi", shorten <=5pt, shorten >=5pt, Rightarrow, from=1-2, to=0]
\end{tikzcd}\]
indicating that its $0$-th, $1$-st and $2$-nd faces are given, respectively, by\[
\begin{tikzcd}
	A_1 \arrow[r,"M_{12}"]& A_2,
\end{tikzcd}
\quad
\begin{tikzcd}
	A_0 \arrow[r,"M_{02}"]& A_2,
\end{tikzcd}
\quad
\begin{tikzcd}
	A_0 \arrow[r,"M_{01}"]& A_1.
\end{tikzcd}
\]
Given a $1$-simplex $(A_0,A_1|M_{01})$, its $0$-th and $1$-st degeneracies are given by the $2$-simplices
\[\begin{tikzcd}
	& {A_0} \\
	{A_0} && {A_1,}
	\arrow["{M_{01}}", from=1-2, to=2-3]
	\arrow["{A_0}", from=2-1, to=1-2]
	\arrow[""{name=0, anchor=center, inner sep=0}, "{M_{01}}"', from=2-1, to=2-3]
	\arrow["\overline\ell_{M_{01}}", shorten <=5pt, shorten >=5pt, Rightarrow, from=1-2, to=0]
\end{tikzcd}
\quad
\begin{tikzcd}
	& {A_1} \\
	{A_0} && {A_1}
	\arrow["{A_1}", from=1-2, to=2-3]
	\arrow["{M_{01}}", from=2-1, to=1-2]
	\arrow[""{name=0, anchor=center, inner sep=0}, "{M_{01}}"', from=2-1, to=2-3]
	\arrow["\overline r_{M_{01}}", shorten <=5pt, shorten >=5pt, Rightarrow, from=1-2, to=0]
\end{tikzcd}\]
where $\overline\ell_{M_{01}}\colon A\otimes_{A_0}M_{01}\to M_{01}$ and $\overline r_{M_{01}}\colon M_{01}\otimes_{A_1}A_1\to M_{01}$ are the bimodule homomorphisms from \cref{UnitalityForBimodules}.
\item The set $\mrt_3(\cC^{\otimes})$ of $3$-simplices is given by
\[
\mrt_3(\cC^{\otimes})\coloneqq\{(A_i,M_{ij},\varphi_{ijk})_{0\leq i\leq j\leq k\leq 3}\ |\ \text{$A_i$ is monoid, $M_{ij}$ is $(A_i,A_j)$-bimodule, }\]
\[
\quad\quad\quad\text{$\varphi_{ijk}\colon M_{ij}\otimes_{A_j}M_{ij}\to M_{jk}$ is $(A_i,A_k)$-bimodule map}\]
\[\quad\quad\quad\quad\quad\quad\text{such that }(\varphi_{012} \otimes_{A_{2}} M_{23}) \bullet \varphi_{023} =\overline{\alpha}_{M_{01}|M_{12}|M_{23}} \bullet (M_{01} \otimes_{A_{1}} \varphi_{123}) \bullet \varphi_{013}\}.
\]
A $3$-simplex $(A_i,M_{ij},\varphi_{ijk})_{0\leq i\leq j\leq k\leq 3}$ will be depicted as
\[\begin{tikzcd} [sep=large]
	{A_1} & {A_2} && {A_1} & {A_2} \\
	{A_0} & {A_3} && {A_0} & {A_3}
	\arrow["{M_{12}}", from=1-1, to=1-2]
	\arrow[""{name=0, anchor=center, inner sep=0}, "{M_{23}}", from=1-2, to=2-2]
	\arrow["{M_{12}}", from=1-4, to=1-5]
	\arrow[""{name=1, anchor=center, inner sep=0}, "{A_{13}}"{pos=0.2}, from=1-4, to=2-5]
	\arrow["{M_{23}}", from=1-5, to=2-5]
	\arrow["{M_{01}}", from=2-1, to=1-1]
	\arrow[""{name=2, anchor=center, inner sep=0}, "{A_{02}}"{pos=0.8}, from=2-1, to=1-2]
	\arrow[""{name=3, anchor=center, inner sep=0}, "M_{01}"', from=2-1, to=2-2]
	\arrow[""{name=4, anchor=center, inner sep=0}, "{M_{01}}", from=2-4, to=1-4]
	\arrow[""{name=5, anchor=center, inner sep=0}, "M_{03}"', from=2-4, to=2-5]
	\arrow["{\varphi_{012}}"'{pos=0.4}, shorten >=3pt, Rightarrow, from=1-1, to=2]
	\arrow[shorten <=26pt, shorten >=26pt, Rightarrow, scaling nfold=3, from=0, to=4]
	\arrow["{\varphi_{023}}"{pos=0.6}, shift right, shorten <=8pt, shorten >=4pt, Rightarrow, from=1-2, to=3]
	\arrow["{\varphi_{013}}"'{pos=0.6}, shift left, shorten <=8pt, shorten >=4pt, Rightarrow, from=1-4, to=5]
	\arrow["{\varphi_{123}}", shorten >=2pt, Rightarrow, from=1-5, to=1]
\end{tikzcd}\]
indicating that its $0$-th, $1$-st, $2$-nd and $3$-rd faces are given, respectively, by
\[\begin{tikzcd}
	& {A_2} \\
	{A_1} && {A_3}
	\arrow["{M_{23}}", from=1-2, to=2-3]
	\arrow["{M_{12}}", from=2-1, to=1-2]
	\arrow[""{name=0, anchor=center, inner sep=0}, "{M_{13}}"', from=2-1, to=2-3]
	\arrow["\varphi_{123}", shorten <=5pt, shorten >=5pt, Rightarrow, from=1-2, to=0]
\end{tikzcd}
\quad
\begin{tikzcd}
	& {A_2} \\
	{A_0} && {A_3}
	\arrow["{M_{23}}", from=1-2, to=2-3]
	\arrow["{M_{02}}", from=2-1, to=1-2]
	\arrow[""{name=0, anchor=center, inner sep=0}, "{M_{03}}"', from=2-1, to=2-3]
	\arrow["\varphi_{023}", shorten <=5pt, shorten >=5pt, Rightarrow, from=1-2, to=0]
\end{tikzcd}
\]
\[
\begin{tikzcd}
	& {A_1} \\
	{A_0} && {A_3}
	\arrow["{M_{13}}", from=1-2, to=2-3]
	\arrow["{M_{01}}", from=2-1, to=1-2]
	\arrow[""{name=0, anchor=center, inner sep=0}, "{M_{03}}"', from=2-1, to=2-3]
	\arrow["\varphi_{013}", shorten <=5pt, shorten >=5pt, Rightarrow, from=1-2, to=0]
\end{tikzcd}
\quad
\begin{tikzcd}
	& {A_1} \\
	{A_0} && {A_2}
	\arrow["{M_{12}}", from=1-2, to=2-3]
	\arrow["{M_{01}}", from=2-1, to=1-2]
	\arrow[""{name=0, anchor=center, inner sep=0}, "{M_{02}}"', from=2-1, to=2-3]
	\arrow["\varphi_{012}", shorten <=5pt, shorten >=5pt, Rightarrow, from=1-2, to=0]
\end{tikzcd}
\]
Given a $2$-simplex corresponding to the bimodule homomorphism $\varphi\colon M_{01}\otimes_{A_1} M_{12}\to M_{02}$, its $0$-th, $1$-st and $2$-nd degeneracies are given by the $3$-simplices depicted respectively by
\[
\begin{tikzcd} [sep=large]
	{A_0} & {A_1} && {A_0} & {A_1} \\
	{A_0} & {A_2} && {A_0} & {A_2}
	\arrow["{M_{01}}", from=1-1, to=1-2]
	\arrow[""{name=0, anchor=center, inner sep=0}, "{M_{12}}", from=1-2, to=2-2]
	\arrow["{M_{01}}", from=1-4, to=1-5]
	\arrow[""{name=1, anchor=center, inner sep=0}, "{M_{02}}"{pos=0.2}, from=1-4, to=2-5]
	\arrow["{M_{12}}", from=1-5, to=2-5]
	\arrow["{A_0}", from=2-1, to=1-1]
	\arrow[""{name=2, anchor=center, inner sep=0}, "{M_{01}}"{pos=0.8}, from=2-1, to=1-2]
	\arrow[""{name=3, anchor=center, inner sep=0}, "M_{02}"', from=2-1, to=2-2]
	\arrow[""{name=4, anchor=center, inner sep=0}, "{A_0}", from=2-4, to=1-4]
	\arrow[""{name=5, anchor=center, inner sep=0}, "M_{02}"', from=2-4, to=2-5]
	\arrow["{\overline\ell_{M_{01}}}"'{pos=0.4}, shorten >=3pt, Rightarrow, from=1-1, to=2]
	\arrow[shorten <=26pt, shorten >=26pt, Rightarrow, scaling nfold=3, from=0, to=4]
	\arrow["{\varphi}"{pos=0.6}, shift right, shorten <=8pt, shorten >=4pt, Rightarrow, from=1-2, to=3]
	\arrow["{\overline\ell_{M_{02}}}"'{pos=0.6}, shift left, shorten <=8pt, shorten >=4pt, Rightarrow, from=1-4, to=5]
	\arrow["{\varphi}", shorten >=2pt, Rightarrow, from=1-5, to=1]
\end{tikzcd}\]
\[\begin{tikzcd} [sep=large]
	{A_1} & {A_1} && {A_1} & {A_1} \\
	{A_0} & {A_2} && {A_0} & {A_2}
	\arrow["{A_1}", from=1-1, to=1-2]
	\arrow[""{name=0, anchor=center, inner sep=0}, "{M_{12}}", from=1-2, to=2-2]
	\arrow["{A_1}", from=1-4, to=1-5]
	\arrow[""{name=1, anchor=center, inner sep=0}, "{M_{12}}"{pos=0.2}, from=1-4, to=2-5]
	\arrow["{M_{12}}", from=1-5, to=2-5]
	\arrow["{M_{01}}", from=2-1, to=1-1]
	\arrow[""{name=2, anchor=center, inner sep=0}, "{M_{01}}"{pos=0.8}, from=2-1, to=1-2]
	\arrow[""{name=3, anchor=center, inner sep=0}, "M_{01}"', from=2-1, to=2-2]
	\arrow[""{name=4, anchor=center, inner sep=0}, "{M_{01}}", from=2-4, to=1-4]
	\arrow[""{name=5, anchor=center, inner sep=0}, "M_{02}"', from=2-4, to=2-5]
	\arrow["{\overline r_{M_{01}}}"'{pos=0.4}, shorten >=3pt, Rightarrow, from=1-1, to=2]
	\arrow[shorten <=26pt, shorten >=26pt, Rightarrow, scaling nfold=3, from=0, to=4]
	\arrow["{\varphi}"{pos=0.6}, shift right, shorten <=8pt, shorten >=4pt, Rightarrow, from=1-2, to=3]
	\arrow["{\varphi}"'{pos=0.6}, shift left, shorten <=8pt, shorten >=4pt, Rightarrow, from=1-4, to=5]
	\arrow["{\overline\ell_{M_{12}}}", shorten >=2pt, Rightarrow, from=1-5, to=1]
\end{tikzcd}\]
\[\begin{tikzcd} [sep=large]
	{A_1} & {A_2} && {A_1} & {A_2} \\
	{A_0} & {A_2} && {A_0} & {A_2}
	\arrow["{M_{12}}", from=1-1, to=1-2]
	\arrow[""{name=0, anchor=center, inner sep=0}, "{A_2}", from=1-2, to=2-2]
	\arrow["{M_{12}}", from=1-4, to=1-5]
	\arrow[""{name=1, anchor=center, inner sep=0}, "{M_{12}}"{pos=0.2}, from=1-4, to=2-5]
	\arrow["{A_2}", from=1-5, to=2-5]
	\arrow["{M_{01}}", from=2-1, to=1-1]
	\arrow[""{name=2, anchor=center, inner sep=0}, "{M_{02}}"{pos=0.8}, from=2-1, to=1-2]
	\arrow[""{name=3, anchor=center, inner sep=0}, "M_{02}"', from=2-1, to=2-2]
	\arrow[""{name=4, anchor=center, inner sep=0}, "{M_{01}}", from=2-4, to=1-4]
	\arrow[""{name=5, anchor=center, inner sep=0}, "M_{02}"', from=2-4, to=2-5]
	\arrow["{\varphi}"'{pos=0.4}, shorten >=3pt, Rightarrow, from=1-1, to=2]
	\arrow[shorten <=26pt, shorten >=26pt, Rightarrow, scaling nfold=3, from=0, to=4]
	\arrow["{\overline r_{M_{02}}}"{pos=0.6}, shift right, shorten <=8pt, shorten >=4pt, Rightarrow, from=1-2, to=3]
	\arrow["{\varphi}"'{pos=0.6}, shift left, shorten <=8pt, shorten >=4pt, Rightarrow, from=1-4, to=5]
	\arrow["{\overline r_{M_{02}}}", shorten >=2pt, Rightarrow, from=1-5, to=1]
\end{tikzcd}
\]
The fact that these indeed define $3$-simplices will be checked as \cref{Degenerate3Simplices}.
\end{enumerate}
\end{const}

\begin{prop}
\label{Degenerate3Simplices}
Let $\cC^{\otimes}$ be a monoidal category which admits a calculus of balanced tensor products. Suppose we are given monoids $A_0$, $A_1$, $A_2$, an $(A_0,A_1)$-bimodule $M_{01}$, an $(A_1,A_2)$-bimodule $M_{12}$ and an $(A_0,A_2)$-bimodule $M_{02}$.
Given a homomorphism $\varphi\colon M_{01}\otimes_{A_1}M_{12}\to M_{02}$ of $(A_0,A_2)$-bimodules (with respect to the structure from \cref{CompositeBimodule}), we have equalities of $(A_0,A_2)$-bimodule maps
\[
\left\{
\begin{array}{rclll}
(\overline\ell_{M_{01}}\otimes_{A_1}M_{12})\bullet\varphi&=&\overline\alpha_{A_0|M_{01}|M_{12}}\bullet(A_0\otimes_{A_0}\varphi)\bullet\overline\ell_{M_{02}}&(0)\\
(\overline r_{M_{01}}\otimes_{A_1}M_{12})\bullet\varphi&=&\overline\alpha_{M_{01}|A_1|M_{12}}\bullet(M_{01}\otimes_{A_1}\overline\ell_{M_{12}})\bullet\varphi&(1)\\
(\varphi\otimes_{A_2}A_2)\bullet\overline r_{M_{02}}&=&\overline\alpha_{M_{01}|M_{12}|A_2}\bullet(M_{01}\otimes_{A_1}\overline r_{M_{12}})\bullet\varphi.&(2)
\end{array}
\right.
\]
\end{prop}

\begin{proof}
We prove (2), and (0) is analogous. To show (2), we have to show that the following diagram commutes in $\cC$:
\[
\begin{tikzcd}
    M_{01}\otimes_{A_1}(M_{12}\otimes_{A_2}A_2)\arrow[rr,"M_{01}\otimes_{A_1}\overline r_{M_{12}}"]\arrow[d,"\overline\alpha^{-1}_{M_{01}|M_{12}|A_2}\bullet(\varphi\otimes_{A_2}A_2)"swap]&&M_{01}\otimes_{A_1}M_{12}\arrow[d,"\varphi"]\\
    M_{02}\otimes_{A_2}A_2\arrow[rr,"\overline r_{M_{02}}"swap]&&M_{12}
\end{tikzcd}
\]
By \cref{PentagonTensor} and definition of $\overline{r}$,
it suffices to check that the following diagram commutes in $\cC$:
\[
\begin{tikzcd}
(M_{01}\otimes_{A_1}M_{12})\otimes_{A_2}A_2\arrow[rr,"\overline r_{M_{01}\otimes_{A_1}M_{12}}"]\arrow[d,"\varphi\otimes_{A_2}A_2"swap]&&M_{01}\otimes_{A_1}M_{12}\arrow[d,"\varphi"]\\
    M_{02}\otimes_{A_2}A_2\arrow[rr,"\overline r_{M_{02}}"swap]&&M_{12}
\end{tikzcd}
\]
By \cref{CompatibilityTensorOnMaps,ActionIsmap}, it suffices to check that the following diagram commutes in $\cC$:
\[
\begin{tikzcd}
    (M_{01}\otimes_{A_1}M_{12})\otimes A_2\arrow[rr,"r_{M_{01}\otimes_{A_1}M_{12}}"]\arrow[d,"\varphi\otimes A_2"swap]&&M_{01}\otimes_{A_1}M_{12}\arrow[d,"\varphi"]\\
    M_{02}\otimes A_2\arrow[rr,"r_{M_{02}}"swap]&&M_{12}
\end{tikzcd}
\]
This diagram commutes in $\cC$ because $\varphi$ is compatible with the right action of $A_2$.

We now prove (1). Consider the diagram in $\cC$:
\[
\adjustbox{width=\textwidth,center}{
\begin{tikzcd}[ampersand replacement=\&,cramped]
	{(M_{01} \otimes_{A_1} A_1) \otimes_{A_1} M_{12}} \&\& {M_{01} \otimes_{A_1} (A_1 \otimes_{A_1} M_{12})} \\
	\\
	\&\& {(M_{01} \otimes A_1) \otimes M_{12}} \&\& {M_{01} \otimes (A_1 \otimes M_{12})} \\
	\&\& {M_{01} \otimes M_{12}} \&\&\& {M_{01} \otimes M_{12}} \& {M_{01} \otimes_{A_1} M_{12}} \\
	\&\& {M_{01} \otimes_{A_1} M_{12}} \&\&\&\& {M_{02}}
	\arrow["{\alpha_{M_{01}|A_{1}|M_{12}}}", from=1-1, to=1-3]
	\arrow["{ \overline r_{M_{01}} \otimes_{A_1} M_{12}}"'{pos=0.4}, curve={height=30pt}, from=1-1, to=5-3]
	\arrow["{M_{01} \otimes_{A_1} \overline \ell_{M_{12}}}"{pos=0.6}, curve={height=-30pt}, from=1-3, to=4-7]
	\arrow["{\pi_{M_{01},A_{1}} \otimes M_{12} \bullet \pi_{M_{01} \otimes_{A_1} A_1, M_{12}}}"{description}, from=3-3, to=1-1]
	\arrow["{\alpha_{M_{01},A_{1},M_{12}}}", from=3-3, to=3-5]
	\arrow["{r_{M_{01}} \otimes M_{12}}"', from=3-3, to=4-3]
	\arrow["{M_{01} \otimes \pi_{A_1 , M_{12}} \bullet \pi_{M_{01},A_1 \otimes_{A_1} M_{12} }}"{description}, from=3-5, to=1-3]
	\arrow["{M_{01} \otimes \ell_{M_{12}}}"', from=3-5, to=4-6]
	\arrow["{\pi_{M_{01},M_{12}}}"', from=4-3, to=5-3]
	\arrow["\pi_{M_{01},M_{12}}"', from=4-6, to=4-7]
	\arrow["\varphi", from=4-7, to=5-7]
	\arrow["\varphi"', from=5-3, to=5-7]
\end{tikzcd}
}
\]
Here, the lower diagram commutes by \cref{BalancedTensor}, and the commutativity of the other remaining diagrams is given by \cref{InducedAlpha2,UnitalityForBimodules} and the Fubini's theorem for colimits. In particular the outer rectangle shows the desired unitality, concluding the proof.
\end{proof}

\begin{rmk}
Unpacking the definition of $\mrt^{\natural}(\cC^{\otimes})$, we see that a $4$-simplex $\mrt(\cC^{\otimes})$ can be pictured as follows:
\[\begin{tikzcd}
	&&&& {A_2} \\
	&&& {A_1} && {A_3} \\
	& {A_2} && {A_0} && {A_4} && {A_2} \\
	{A_1} && {A_3} &&&& {A_1} && {A_3} \\
	{A_0} && {A_4} &&&& {A_0} && {A_4} \\
	&& {A_2} &&&& {A_2} \\
	& {A_1} && {A_3} && {A_1} && {A_3} \\
	& {A_0} && {A_4} && {A_0} && {A_4}
	\arrow["{M_{23}}", from=1-5, to=2-6]
	\arrow[""{name=0, anchor=center, inner sep=0}, "{M_{24}}"'{pos=0.4}, from=1-5, to=3-6]
	\arrow["{M_{12}}", from=2-4, to=1-5]
	\arrow[""{name=1, anchor=center, inner sep=0}, "{M_{34}}", from=2-6, to=3-6]
	\arrow[""{name=2, anchor=center, inner sep=0}, "{M_{23}}", from=3-2, to=4-3]
	\arrow[""{name=3, anchor=center, inner sep=0}, "{M_{02}}"'{pos=0.6}, from=3-4, to=1-5]
	\arrow[""{name=4, anchor=center, inner sep=0}, "{M_{01}}", from=3-4, to=2-4]
	\arrow[""{name=5, anchor=center, inner sep=0}, "{M_{04}}"', from=3-4, to=3-6]
	\arrow["{M_{23}}", from=3-8, to=4-9]
	\arrow[""{name=6, anchor=center, inner sep=0}, "{M_{24}}"'{pos=0.6}, from=3-8, to=5-9]
	\arrow["{M_{12}}", from=4-1, to=3-2]
	\arrow["{M_{34}}", from=4-3, to=5-3]
	\arrow[""{name=7, anchor=center, inner sep=0}, "{M_{12}}", from=4-7, to=3-8]
	\arrow[""{name=8, anchor=center, inner sep=0}, "{M_{14}}"'{pos=0.6}, from=4-7, to=5-9]
	\arrow["{M_{34}}", from=4-9, to=5-9]
	\arrow[""{name=9, anchor=center, inner sep=0}, "{M_{02}}"', from=5-1, to=3-2]
	\arrow["{M_{01}}", from=5-1, to=4-1]
	\arrow[""{name=10, anchor=center, inner sep=0}, "{M_{03}}"'{pos=0.4}, from=5-1, to=4-3]
	\arrow[""{name=11, anchor=center, inner sep=0}, "{M_{04}}"', from=5-1, to=5-3]
	\arrow["{M_{01}}", from=5-7, to=4-7]
	\arrow[""{name=12, anchor=center, inner sep=0}, "{M_{04}}"', from=5-7, to=5-9]
	\arrow["{M_{23}}", from=6-3, to=7-4]
	\arrow[""{name=13, anchor=center, inner sep=0}, "{M_{23}}", from=6-7, to=7-8]
	\arrow[""{name=14, anchor=center, inner sep=0}, "{M_{12}}", from=7-2, to=6-3]
	\arrow[""{name=15, anchor=center, inner sep=0}, "{M_{13}}", from=7-2, to=7-4]
	\arrow[""{name=16, anchor=center, inner sep=0}, "{M_{34}}", from=7-4, to=8-4]
	\arrow["{M_{12}}", from=7-6, to=6-7]
	\arrow[""{name=17, anchor=center, inner sep=0}, "{M_{13}}", from=7-6, to=7-8]
	\arrow[""{name=18, anchor=center, inner sep=0}, "{M_{14}}"{pos=0.4}, from=7-6, to=8-8]
	\arrow["{M_{34}}", from=7-8, to=8-8]
	\arrow["{M_{01}}", from=8-2, to=7-2]
	\arrow[""{name=19, anchor=center, inner sep=0}, "{M_{03}}"{pos=0.6}, from=8-2, to=7-4]
	\arrow[""{name=20, anchor=center, inner sep=0}, "{M_{04}}"', from=8-2, to=8-4]
	\arrow[""{name=21, anchor=center, inner sep=0}, "{M_{01}}", from=8-6, to=7-6]
	\arrow[""{name=22, anchor=center, inner sep=0}, "{M_{04}}"', from=8-6, to=8-8]
	\arrow["{\varphi_{024}}"'{pos=0.7}, shorten <=19pt, shorten >=4pt, Rightarrow, from=1-5, to=5]
	\arrow["{\varphi_{012}}"{pos=0.7}, shorten >=2pt, Rightarrow, from=2-4, to=3]
	\arrow["{\varphi_{234}}"'{pos=0.7}, shorten >=2pt, Rightarrow, from=2-6, to=0]
	\arrow["{(3)}", shorten <=15pt, shorten >=15pt, Rightarrow, scaling nfold=3, from=1, to=7]
	\arrow["{(1)}", shorten <=15pt, shorten >=15pt, Rightarrow, scaling nfold=3, from=2, to=4]
	\arrow["{\varphi_{023}}"{pos=0.6}, shorten <=8pt, shorten >=5pt, Rightarrow, from=3-2, to=10]
	\arrow["{\varphi_{124}}"', shorten <=8pt, shorten >=8pt, Rightarrow, from=3-8, to=8]
	\arrow["{\varphi_{012}}"{pos=0.7}, shorten >=2pt, Rightarrow, from=4-1, to=9]
	\arrow["{\varphi_{034}}"{pos=0.4}, shorten <=8pt, shorten >=8pt, Rightarrow, from=4-3, to=11]
	\arrow["{\varphi_{014}}"'{pos=0.4}, shorten <=8pt, shorten >=8pt, Rightarrow, from=4-7, to=12]
	\arrow["{\varphi_{234}}"'{pos=0.6}, shorten >=2pt, Rightarrow, from=4-9, to=6]
	\arrow["{(4)}"', shorten <=10pt, shorten >=10pt, Rightarrow, scaling nfold=3, from=11, to=14]
	\arrow["{\varphi_{123}}", shorten <=2pt, shorten >=8pt, Rightarrow, from=6-3, to=15]
	\arrow["{(0)}"', shorten <=10pt, shorten >=10pt, Rightarrow, scaling nfold=3, from=13, to=12]
	\arrow["{\varphi_{123}}", shorten <=2pt, shorten >=8pt, Rightarrow, from=6-7, to=17]
	\arrow["{\varphi_{013}}"'{pos=0.3}, shorten <=2pt, shorten >=7pt, Rightarrow, from=7-2, to=19]
	\arrow["{(2)}"', shorten <=19pt, shorten >=19pt, Rightarrow, scaling nfold=3, from=16, to=21]
	\arrow["{\varphi_{034}}"{pos=0.4}, shorten <=8pt, shorten >=8pt, Rightarrow, from=7-4, to=20]
	\arrow["{\varphi_{014}}"'{pos=0.4}, shorten <=8pt, shorten >=8pt, Rightarrow, from=7-6, to=22]
	\arrow["{\varphi_{134}}"{pos=0.3}, shorten <=2pt, shorten >=7pt, Rightarrow, from=7-8, to=18]
\end{tikzcd}\]
This amounts to having all displayed monoids $A_i$ for $i=0,\dots4$, all displayed $(A_i,A_j)$-bimodules $M_{ij}$ for $0\leq i< j\leq 4$, all displayed bimodule maps $\varphi_{ijk}\colon M_{ij}\otimes_{A_j}M_{jk}\to M_{ij}$ for $0\leq i<j<k\leq4$, together satisfying the conditions
\[
\left\{
\begin{array}{rcllll}
(\varphi_{123} \otimes_{A_3} M_{34}) \bullet \varphi_{134} &=& \overline{\alpha}_{M_{12}|M_{23}|M_{34}} \bullet (M_{12} \otimes_{A_2} \varphi_{234}) \bullet \varphi_{124}&(0)\\
(\varphi_{023} \otimes_{A_3} M_{34}) \bullet \varphi_{034}&=&\overline{\alpha}_{M_{02}|M_{23}|M_{34}} \bullet (M_{02} \otimes_{A_2} \varphi_{234}) \bullet \varphi_{024}&(1)\\
(\varphi_{013} \otimes_{A_3} M_{34}) \bullet \varphi_{034}&=&\overline{\alpha}_{M_{01}|M_{13}|M_{34}} \bullet (M_{01} \otimes_{A_1} \varphi_{134}) \bullet \varphi_{014}&(2)\\
(\varphi_{012} \otimes_{A_2} M_{24}) \bullet \varphi_{024}&=&\overline{\alpha}_{M_{01}|M_{12}|M_{24}} \bullet (M_{01} \otimes_{A_1} \varphi_{124}) \bullet \varphi_{014}&(3)\\
(\varphi_{012} \otimes_{A_{2}} M_{23}) \bullet \varphi_{023}&=&\overline{\alpha}_{M_{01}|M_{12}|M_{23}} \bullet (M_{01} \otimes_{A_{1}} \varphi_{123}) \bullet \varphi_{013}.&(4)\\
\end{array}\right.
\]
\end{rmk}

\begin{ex}
The construction $\mrt(\cC^{\otimes})$ recovers some known instances.
    \begin{itemize}[leftmargin=*]
        \item $\mrt(\cA b^{\otimes})$ is isomorphic to the Duskin nerve of a variant of the bicategory $\cA lg^{bi}$ of abelian groups, bimodules and bimodule maps considered e.g.~in \cite[\textsection3]{SPclassification}.
        \item $\mrt(\set^{\amalg})$ is isomorphic to the Duskin nerve of the bicategory of sets, spans and maps of spans
    \end{itemize}
\end{ex}

\subsection{Marked nerve of a monoidal category}

We briefly recall terminology and notation about marked simplicial sets and $2$-complicial sets. See e.g.~\cite{VerityComplicial,RiehlComplicial,ORms} for more details.

\begin{defn} 
A \emph{simplicial set with marking} is a simplicial set $X$, and for $m>0$ a subset $X^{\mathrm{th}}_m\subseteq X_m$ of simplices of $X$, called \emph{marked} simplices, which contain all degenerate simplices of $X$.
\end{defn}

\begin{notn}
\label{preliminarynotation}
We denote
\begin{itemize}[leftmargin=*]
    \item by $\Delta^k[m]$, for $0\leq k \leq m$, the simplicial set given by the standard $m$-simplex in which a non-degenerate simplex is marked if and only if it contains the vertices $\{k-1,k,k+1\}\cap [m]$;
    \item by $\Delta^k[m]'$, for $0\leq k \leq m$, the simplicial set given by the standard $m$-simplex with marking obtained from $\Delta^k[m]$ by additionally marking the $(k-1)$-st and $(k+1)$-st $(m-1)$-dimensional face of $\Delta[m]$, whenever defined;
    \item by $\Delta^k[m]''$, for $0\leq k \leq m$, the simplicial set given by the standard $m$-simplex with marking obtained from $\Delta^k[m]'$ by additionally marking the $k$-th face of $\Delta[m]$;
    \item by $\Lambda^k[m]$, for $0\leq k \leq m$, the simplicial set given by the usual $k$-horn $\Lambda^k[m]$ with marking inherited from $\Delta^k[m]$.
    \item by $\Delta[m]^{\sharp}$ the simplicial set given by the standard $m$-simplex with the maximal marking.
    \item by $\Delta[m]_{t}$ the standard $m$-simplex with the in which the top $m$-dimensional simplex is marked, as well as all degenerate simplices.
    \item by $\Delta[3]^{\mathrm{eq}}$ the standard $3$-simplex in which the $1$-simplices $[0,2]$ and $[1,3]$ are marked, as well as all degenerate $1$-simplices and all simplices in dimension $2$ or higher.
    \item by $\Delta[3]^{\mathrm{eq}}\star\Delta[\ell]$ the simplicial set given by the join of a standard $3$-simplex with a standard $\ell$-simplex, which is isomorphic to the standard $(3+1+\ell)$-simplex, in which a simplex $\sigma\star\tau$ is marked if and only if $\sigma$ is marked in $\Delta[3]^{\mathrm{eq}}$ or $\tau$ is a degenerate simplex of $\Delta[\ell]$.
    \item by $\Delta[3]^{\sharp}\star\Delta[\ell]$ the simplicial set given by the join of a standard $3$-simplex with a standard $\ell$-simplex, which is isomorphic to the standard $(3+1+\ell)$-simplex, in which a simplex $\sigma\star\tau$ is marked if and only if $\sigma$ is marked in $\Delta[3]^{\sharp}$ or $\tau$ is a degenerate simplex of $\Delta[\ell]$.
\end{itemize}
\end{notn}

In order to describe the intended marking on $\mrt(\cC^{\otimes})$, we consider the following notion.

\begin{defn}
\label{complicialextensions}
A marked simplicial set $X$ is a \emph{$2$-complicial set} if it has the right lifting property in the category of marked simplicial sets and marking preserving maps with respect to the following classes of maps:
\begin{enumerate}[leftmargin=*]
\item For $m> 1$ and $0< k< m$, the \emph{complicial inner horn extension} is the inclusion
\[\Lambda^k[m]\to \Delta^k[m].\]
 \item For $m\geq 2$ and $0< k < m$, the \emph{complicial thinness extension} is the inclusion
\[\Delta^k[m]' \to \Delta^k[m]''.\]
\item For $m>n$, the \emph{triviality extension} is the inclusion
\[\Delta[m]\to\Delta[m]_t.\]
\item For $m\ge-1$, the \emph{complicial saturation extension} is the inclusion
\[\Delta[3]^{\mathrm{eq}}\star\Delta[m]\to\Delta[3]^{\sharp}\star\Delta[m].\]
\end{enumerate}
\end{defn}

\begin{defn}
\label{Equivalence}
Let $\cC^{\otimes}$ be a monoidal category.
Given monoids $A,B$, an $(A,B)$-bimodule is said to be an \emph{$(A,B)$-equivalence} if there exist a $(B,A)$-bimodule $M'$, as well as an isomorphism of $(A,A)$-bimodules and an isomorphism of $(B,B)$-bimodules
\[
M\otimes_B M'\cong A\quad \text{ and }\quad M'\otimes_AM\cong B,
\]
with respect to the bimodule structures from \cref{CompositeBimodule,IdentityBimodule}.
\end{defn}

\begin{const}
Let $\cC^{\otimes}$ be a monoidal category. The \emph{naturally marked monoidal nerve} $ \mrt^{\natural}(\cC^{\otimes})$ of $\cC^{\otimes}$ is the marked simplicial set whose underlying simplicial set is the simplicial set $\mrt(\cC^{\otimes})$ from \cref{MonoidalNerve} and in which:
\begin{enumerate}[leftmargin=*]
\item The set $\mrt_1^{\mathrm{th}}(\cC^{\otimes})$ of marked $1$-simplices is given by the bimodules which are equivalences in the sense of \cref{Equivalence}:
\[
\mrt_1^{\mathrm{th}}(\cC^{\otimes})\coloneqq\{(A,B|M)\ |\ \text{$M$ is an $(A,B)$-equivalence}\}
\]
A marked $1$-simplex will be depicted in blue, as follows:
\[\begin{tikzcd}[ampersand replacement=\&]
	A \& A
	\arrow["M", color={rgb,255:red,65;green,51;blue,255}, from=1-1, to=1-2]
\end{tikzcd}\]
\item The set $\mrt_2^{\mathrm{th}}(\cC^{\otimes})$ of marked $2$-simplices is given by the $2$-simplices inhabited by an isomorphism of bimodules:
\[
\mrt_2^{\mathrm{th}}(\cC^{\otimes})\coloneqq\{(A_0, A_1,A_2|M_{01},M_{12},M_{23}|\varphi)\ 
|\ \text{$A_i$ is monoid, $M_{ij}$ is $(A_i,A_j)$-bimodule},\]
\[
\quad\quad\quad\quad\quad\quad\quad\quad\quad\quad\quad\quad\text{$\varphi\colon M_{01}\otimes_{A_1}M_{12}\to M_{02}$ is $(A_0\otimes{A_1},A_2)$-bimodule isomorphism}\}.
\]
A marked $2$-simplex will be depicted in blue, as follows:
\[\begin{tikzcd}
	& {A_1} \\
	{A_0} && {A_2}
	\arrow["{M_{12}}", from=1-2, to=2-3]
	\arrow["{M_{01}}", from=2-1, to=1-2]
	\arrow[""{name=0, anchor=center, inner sep=0}, "{M_{02}}"', from=2-1, to=2-3]
	\arrow["\varphi", color={rgb,255:red,65;green,51;blue,255}, shorten <=5pt, shorten >=5pt, Rightarrow, from=1-2, to=0]
\end{tikzcd}\]
\item For $k>2$, every $k$-simplex is marked; that is
\[
\mrt^{\mathrm{th}}_k(\cC^{\otimes})=\mrt_k(\cC^{\otimes}).
\]
\end{enumerate}
\end{const}

\begin{thm}
\label{maintheorem}
Given $\cC^{\otimes}$ a monoidal category which admits a calculus of balanced tensor products, the monoidal nerve $\mrt^{\natural}(\cC^{\otimes})$ is a $2$-complicial set.
\end{thm}

\begin{proof}
We will show as \cref{horn,thin,saturation} that the marked simplicial set $\mrt^{\natural}(\cC^{\otimes})$ has the right lifting properties with respect to all the anodyne extensions from \cref{complicialextensions}.
\end{proof}

\begin{rmk}
The construction $\mrt^\natural(\cC^{\otimes})$ most likely recovers some known instances.

For instance, the marked simplicial set $\mrt(\cA b^{\otimes})$ is likely equal to the Duskin nerve of a variant of the bicategory $\cA lg^{bi}$ of abelian groups, bimodules and bimodule maps considered e.g.~in \cite[\textsection3]{SPclassification} with the marking described in \cite{Gurski}. We believe the two are in agreement, although the marking is described using different terms, and we have not verified their equivalence.

Further, if the monoidal category $\cC^{\otimes}$ is strict, then the marked simplicial set $\mrt(\cC^{\otimes})$ is isomorphic to the natural nerve from \cite{ORnerve} of the Morita $2$-category of $\cC^{\otimes}$.
\end{rmk}

\subsection{Complicial horns}

In this subsection we verify that $\mrt^\natural(\cC^{\otimes})$ lifts against all complicial horn inclusions:

\begin{thm}
\label{horn}
The marked simplicial set $\mrt^{\natural}(\cC^{\otimes})$ has the right lifting property with respect to the complicial horn inclusion $\Lambda^k[m]\to\Delta^k[m]$ for $m\geq0$ and $0\leq k\leq m$. 
\end{thm}

\begin{proof}
We treat the various cases as \cref{horn10,horn20,horn21,horn30,horn31,horn40,horn41,horn42,hornHigh}.
\end{proof}

We now treat the $1$-dimensional horns:

\begin{prop}
\label{horn10}
The marked simplicial set $\mrt^{\natural}(\cC^{\otimes})$ has the right lifting property with respect to the complicial horn inclusion $\Lambda^k[1]\to\Delta^k[1]$ for $k=0,1$.
\end{prop}

\begin{proof}
We treat the case $k=0$, the case $k=1$ being analogous. Given a map of marked simplicial sets
$\Lambda^0[1]\to \mrt^{\natural}(\cC^{\otimes})$, which amounts to a monoid $A$ in $\cC^{\otimes}$, a lift
$\Delta^0[1]\to \mrt^{\natural}(\cC^{\otimes})$ is given by the degenerate $1$-simplex
\[\begin{tikzcd}[ampersand replacement=\&]
	A \& A
	\arrow["A", color={rgb,255:red,65;green,51;blue,255}, from=1-1, to=1-2]
\end{tikzcd}\]
obtained by considering the bimodule structure from \cref{IdentityBimodule}.
\end{proof}

We now treat the $2$-dimensional horns:

\begin{prop}
\label{horn21}
The marked simplicial set $\mrt^{\natural}(\cC^{\otimes})$ has the right lifting property with respect to the complicial horn inclusion $\Lambda^1[2]\to\Delta^1[2]$.
\end{prop}

\begin{proof} 
Given a map of marked simplicial sets $\Lambda^1[2]\to \mrt^{\natural}(\cC^{\otimes})$, which can be depicted as
\[\begin{tikzcd}[ampersand replacement=\&]
	\& B \\
	A \&\& C
	\arrow["N", from=1-2, to=2-3]
	\arrow["M", from=2-1, to=1-2]
\end{tikzcd}\]
a lift $\Delta^1[2]\to \mrt^{\natural}(\cC^{\otimes})$ is the map of marked simplicial sets depicted as
\[\begin{tikzcd}[ampersand replacement=\&]
	\& B \\
	A \&\& C
	\arrow["N", from=1-2, to=2-3]
	\arrow["M", from=2-1, to=1-2]
	\arrow[""{name=0, anchor=center, inner sep=0}, "{M \otimes_B N}"', from=2-1, to=2-3]
	\arrow["\text{id}"', color={rgb,255:red,65;green,51;blue,255}, shorten <=5pt, shorten >=5pt, Rightarrow, from=1-2, to=0]
\end{tikzcd}\]
Here, the object $M\otimes_BN$ has an $(A,C)$-bimodule structure by \cref{CompositeBimodule}. Further, we know that the identity morphism
\[\id_{M\otimes_BN}\colon M\otimes_BN\to M\otimes_BN\]
is a bimodule isomorphism.
\end{proof}

\begin{prop}
\label{horn20}
The marked simplicial set $\mrt^{\natural}(\cC^{\otimes})$ has the right lifting property with respect to the complicial horn inclusion $\Lambda^k[2]\to\Delta^k[2]$ for $k=0,2$.
\end{prop}

\begin{proof} 
We treat the case $k=0$, the case $k=2$ being analogous. Given a map of marked simplicial sets $\Lambda^0[2]\to \mrt^{\natural}(\cC^{\otimes})$, which can be depicted as
\[\begin{tikzcd}
	& B \\
	A && C
	\arrow["{ M}", color={rgb,255:red,65;green,51;blue,255}, from=2-1, to=1-2]
	\arrow["N"', from=2-1, to=2-3]
\end{tikzcd}\]
we construct a lift $\Delta^0[2]\to \mrt^{\natural}(\cC^{\otimes})$ depicted as
\[\begin{tikzcd}[ampersand replacement=\&]
	\& B \\
	A \&\& C
	\arrow["N", from=1-2, to=2-3]
	\arrow["M", color={rgb,255:red,65;green,51;blue,255}, from=2-1, to=1-2]
	\arrow[""{name=0, anchor=center, inner sep=0}, "P"', from=2-1, to=2-3]
	\arrow["\varphi", color={rgb,255:red,65;green,51;blue,255}, shorten <=5pt, shorten >=5pt, Rightarrow, from=1-2, to=0]
\end{tikzcd}\]
Since $M$ is by assumption an equivalence, there exists an $(B,A)$-bimodule $M'$ and an isomorphism of $(A,A)$-bimodules
\[
M \otimes_B M'\cong A\]
with respect to the bimodule structures from \cref{IdentityBimodule,CompositeBimodule}.
We define $N\coloneqq M'\otimes_AP$, with the $(B,C)$-bimodule structure from \cref{CompositeBimodule}. Further, we set $\alpha$
as the composite $(A,C)$-bimodule isomorphism
\[\begin{array}{lll}
    \alpha\colon  M \otimes_B (M' \otimes_A P) &\cong (M \otimes_B M') \otimes_A P&\text{\cref{AssociativityForBimodules}}\\
     &\cong A \otimes_A P&\text{\cref{TensorWithIso}}\\
     &\cong P&\text{\cref{UnitalityForBimodules}}
\end{array}
\]
which concludes the proof.
\end{proof}

We now treat the $3$-dimensional horns:

\begin{prop}
\label{horn31}
The marked simplicial set $\mrt^{\natural}(\cC^{\otimes})$ has the right lifting property with respect to the complicial horn inclusion $\Lambda^k[3]\to\Delta^k[3]$ for $k=1,2$.
\end{prop}

\begin{proof}
We treat the case $k=1$, the case $k=2$ being analogous. Given a map of marked simplicial sets $\Lambda^1[3]\to \mrt^{\natural}(\cC^{\otimes})$, which can be depicted as
\[\begin{tikzcd}[sep=large, ampersand replacement=\&]
	A_{1} \& A_{2} \& A_{1} \& A_{2} \\
	A_0 \& A_{3} \& A_{0} \& A_{3}
	\arrow["M_{12}", from=1-1, to=1-2]
	\arrow["M_{23}", from=1-2, to=2-2]
	\arrow["M_{12}", from=1-3, to=1-4]
	\arrow[""{name=0, anchor=center, inner sep=0}, "M_{13}"{pos=0.2}, from=1-3, to=2-4]
	\arrow["M_{23}", from=1-4, to=2-4]
	\arrow["M_{01}", from=2-1, to=1-1]
	\arrow[""{name=1, anchor=center, inner sep=0}, "M_{02}"{pos=0.8}, from=2-1, to=1-2]
	\arrow["M_{03}"', from=2-1, to=2-2]
	\arrow["M_{01}", from=2-3, to=1-3]
	\arrow[""{name=2, anchor=center, inner sep=0}, "M_{03}"', from=2-3, to=2-4]
	\arrow["\alpha"'{pos=0.4}, color={rgb,255:red,65;green,51;blue,255}, shorten >=3pt, Rightarrow, from=1-1, to=1]
	\arrow["\gamma"'{pos=0.6}, shift left, shorten <=8pt, shorten >=4pt, Rightarrow, from=1-3, to=2]
	\arrow["\delta", shorten >=3pt, Rightarrow, from=1-4, to=0]
\end{tikzcd}\]
we construct a lift $\Delta^1[3]\to \mrt^{\natural}(\cC^{\otimes})$, which can be depicted as
\[\begin{tikzcd}[sep=large, ampersand replacement=\&]
	A_{1} \& A_{2} \& A_{1} \& A_{2} \\
	A_{0} \& A_{3} \& A_{0} \& A_{3}
	\arrow["M_{12}", from=1-1, to=1-2]
	\arrow[""{name=0, anchor=center, inner sep=0}, "M_{23}", from=1-2, to=2-2]
	\arrow["M_{12}", from=1-3, to=1-4]
	\arrow[""{name=1, anchor=center, inner sep=0}, "M_{13}"{pos=0.2}, from=1-3, to=2-4]
	\arrow["M_{23}", from=1-4, to=2-4]
	\arrow["M_{01}", color={rgb,255:red,65;green,51;blue,255}, from=2-1, to=1-1]
	\arrow[""{name=2, anchor=center, inner sep=0}, "M_{02}"{pos=0.8}, from=2-1, to=1-2]
	\arrow[""{name=3, anchor=center, inner sep=0}, "M_{03}"', from=2-1, to=2-2]
	\arrow[""{name=4, anchor=center, inner sep=0}, "M_{01}", color={rgb,255:red,65;green,51;blue,255}, from=2-3, to=1-3]
	\arrow[""{name=5, anchor=center, inner sep=0}, "M_{03}"', from=2-3, to=2-4]
	\arrow["\alpha"'{pos=0.4}, color={rgb,255:red,65;green,51;blue,255}, shorten >=3pt, Rightarrow, from=1-1, to=2]
	\arrow[color={rgb,255:red,65;green,51;blue,255}, shorten <=19pt, shorten >=19pt, Rightarrow, scaling nfold=3, from=0, to=4]
	\arrow["\beta"{pos=0.6}, shift right, shorten <=8pt, shorten >=4pt, Rightarrow, from=1-2, to=3]
	\arrow["\gamma"'{pos=0.6}, shift left, color={rgb,255:red,65;green,51;blue,255}, shorten <=8pt, shorten >=4pt, Rightarrow, from=1-3, to=5]
	\arrow["\delta", shorten >=3pt, Rightarrow, from=1-4, to=1]
\end{tikzcd}\]
Since $\alpha$ is an isomorphism, by \cref{TensorWithIso} we have that $\alpha\otimes_{A_{2}} M_{23}$ is an isomorphism. We can then define $\beta$ to be the map of $(A_{0},A_{3})$-bimodules from $M_{02} \otimes_{A_{2}} M_{23}$ to $M_{03}$ given by 
\[ \beta\coloneqq (\alpha \otimes_{A_{2}} M_{23})^{-1} \bullet \overline{\alpha}_{M_{01}|M_{12}|M_{23}} \bullet (M_{01} \otimes_{A_{1}} \delta) \bullet \gamma.\]
Hence, we obtain the equality map of $(A_{0},A_{3})$-bimodules from $ (M_{01} \otimes_{A_{1}} M_{12})\otimes_{A_{2}} M_{23}$ to $M_{03}$
\[\overline{\alpha}_{M_{01}|M_{12}|M_{23}}\bullet (M_{01} \otimes_{A_{1}} \delta) \bullet \gamma= (\alpha \otimes_{A_{2}} M_{23}) \bullet \beta,\]
concluding the proof.
\end{proof}

\begin{prop}
\label{horn30}
The marked simplicial set $\mrt^{\natural}(\cC^{\otimes})$ has the right lifting property with respect to the complicial horn inclusion $\Lambda^k[3]\to\Delta^k[3]$ for $k=0,3$.
\end{prop}

\begin{proof}
We treat the case $k=0$, the case $k=3$ being analogous. Given a map of marked simplicial sets $\Lambda^0[3]\to \mrt^{\natural}(\cC^{\otimes})$, which can be depicted as
\[\begin{tikzcd}[sep=large, ampersand replacement=\&]
	A_{1} \& A_{2} \& A_{1} \& A_{2} \\
	A_{0} \& A_{3} \& A_{0} \& A_{3}
	\arrow["M_{12}", from=1-1, to=1-2]
	\arrow["M_{23}", from=1-2, to=2-2]
	\arrow["M_{12}", from=1-3, to=1-4]
	\arrow["M_{13}"{pos=0.2}, from=1-3, to=2-4]
	\arrow["M_{23}", from=1-4, to=2-4]
	\arrow["M_{01}", color={rgb,255:red,65;green,51;blue,255}, from=2-1, to=1-1]
	\arrow[""{name=0, anchor=center, inner sep=0}, "M_{02}"{pos=0.8}, from=2-1, to=1-2]
	\arrow[""{name=1, anchor=center, inner sep=0}, "M_{03}"', from=2-1, to=2-2]
	\arrow["M_{01}", color={rgb,255:red,65;green,51;blue,255}, from=2-3, to=1-3]
	\arrow[""{name=2, anchor=center, inner sep=0}, "M_{03}"', from=2-3, to=2-4]
	\arrow["\alpha"'{pos=0.4}, color={rgb,255:red,65;green,51;blue,255}, shorten >=3pt, Rightarrow, from=1-1, to=0]
	\arrow["\beta"{pos=0.6}, shift right, shorten <=8pt, shorten >=4pt, Rightarrow, from=1-2, to=1]
	\arrow["\gamma"'{pos=0.6}, shift left, color={rgb,255:red,65;green,51;blue,255}, shorten <=8pt, shorten >=4pt, Rightarrow, from=1-3, to=2]
\end{tikzcd}\]
we construct a lift $\Delta^0[3]\to \mrt^{\natural}(\cC^{\otimes})$, which can be depicted as
\[\begin{tikzcd}[sep=large, ampersand replacement=\&]
	A_{1} \& A_{2} \& A_{1} \& A_{2} \\
	A_{0} \& A_{3} \& A_{0} \& A_{3}
	\arrow["M_{12}", from=1-1, to=1-2]
	\arrow[""{name=0, anchor=center, inner sep=0}, "M_{23}", from=1-2, to=2-2]
	\arrow["M_{12}", from=1-3, to=1-4]
	\arrow[""{name=1, anchor=center, inner sep=0}, "M_{13}"{pos=0.2}, from=1-3, to=2-4]
	\arrow["M_{23}", from=1-4, to=2-4]
	\arrow["M_{01}", color={rgb,255:red,65;green,51;blue,255}, from=2-1, to=1-1]
	\arrow[""{name=2, anchor=center, inner sep=0}, "M_{02}"{pos=0.8}, from=2-1, to=1-2]
	\arrow[""{name=3, anchor=center, inner sep=0}, "M_{03}"', from=2-1, to=2-2]
	\arrow[""{name=4, anchor=center, inner sep=0}, "M_{01}", color={rgb,255:red,65;green,51;blue,255}, from=2-3, to=1-3]
	\arrow[""{name=5, anchor=center, inner sep=0}, "M_{03}"', from=2-3, to=2-4]
	\arrow["\alpha"'{pos=0.4}, color={rgb,255:red,65;green,51;blue,255}, shorten >=3pt, Rightarrow, from=1-1, to=2]
	\arrow[color={rgb,255:red,65;green,51;blue,255}, shorten <=19pt, shorten >=19pt, Rightarrow, scaling nfold=3, from=0, to=4]
	\arrow["\beta"{pos=0.6}, shift right, shorten <=8pt, shorten >=4pt, Rightarrow, from=1-2, to=3]
	\arrow["\gamma"'{pos=0.6}, shift left, color={rgb,255:red,65;green,51;blue,255}, shorten <=8pt, shorten >=4pt, Rightarrow, from=1-3, to=5]
	\arrow["\delta", shorten >=3pt, Rightarrow, from=1-4, to=1]
\end{tikzcd}\]
Since $M_{01}$ is an equivalence and there exists a bimodule isomorphism
\[\omega\colon M_{01}'\otimes_{A_{0}} M_{01} \cong A_{1}.\]
We can then define $\delta$ to be the map of $(A_{1},A_{3})$-bimodules from $M_{12} \otimes_{A_{2}} M_{23}$ to $M_{13}$ given by
\begin{align*}\delta &\coloneqq \overline\ell^{-1}_{M_{12}\otimes_{A_{2}} M_{23}}\bullet (\omega\otimes_{A_{1}}(M_{12}\otimes_{A_{2}} M_{23}))^{-1}\bullet \overline\alpha_{M_{01}'|M_{01}|M_{12}\otimes_{A_2}M_{23}} \\&\qquad\bullet (M_{01}' \otimes_{A_{0}} (\overline\alpha^{-1}_{M_{01}|M_{12}|M_{23}} \bullet (\alpha \otimes_{A_{2}} M_{23}) \bullet \beta \bullet \gamma^{-1}))\bullet (\omega\otimes_{A_{1}} M_{13})\bullet \overline\ell_{M_{13}}.\end{align*}
By \cref{ReductionI}, we deduce
\begin{align*}
    A_{1} \otimes_{A_{1}} \delta &= (\omega\otimes_{A_{1}} (M_{12} \otimes_{A_{2}} M_{23}))^{-1}\bullet \overline\alpha_{M_{01}'|M_{01}|M_{12}\otimes_{A_2}M_{23}} \\
    &\qquad\bullet (M_{01}' \otimes_{A_{0}} (\overline\alpha_{M_{01}|M_{12}|M_{23}}^{-1} \bullet (\alpha \otimes_{A_{2}} M_{23}) \bullet \beta\bullet \gamma^{-1}))\bullet (\omega\otimes_{A_{1}} M_{13}).
\end{align*}
By \cref{ReductionII}, we deduce
\[
(M_{01}'\otimes_{A_0}M_{01})\otimes_{A_1}\delta= \overline\alpha_{M_{01}'|M_{01}|M_{12}\otimes_{A_2}M_{23}}\bullet (M_{01}' \otimes_{A_{0}} (\overline\alpha_{M_{01}|M_{12}|M_{23}}^{-1} \bullet (\alpha \otimes_{A_{2}} M_{23}) \bullet \beta\bullet \gamma^{-1})),
\]
and then by \cref{NaturalityAssociatorTensor}, we have
\[
M_{01}'\otimes_{A_0}(M_{01}\otimes_{A_1}\delta)= M_{01}' \otimes_{A_{0}} (\overline\alpha_{M_{01}|M_{12}|M_{23}}^{-1} \bullet (\alpha \otimes_{A_{2}} M_{23}) \bullet \beta\bullet \gamma^{-1}).
\]
Since $M_{01}'$ is an equivalence, we have that
\[\overline\alpha_{M_{01}|M_{12}|M_{23}}\bullet (M_{01} \otimes_{A_{1}} \delta) \bullet \gamma= (\alpha \otimes_{A_{2}} M_{23}) \bullet \beta,\]
concluding the proof.
\end{proof}

\begin{prop}
\label{horn40}
The marked simplicial set $ \mrt^\natural(\cC^{\otimes})$ has the right lifting property with respect to the complicial horn inclusion $\Lambda^k[4]\to\Delta^k[4]$ for $k=0,4$.
\end{prop}

\begin{proof}
We treat the case $k=0$, the case $k=4$ being analogous. Given a map of marked simplicial sets $\Lambda^0[4]\to \mrt^{\natural}(\cC^{\otimes})$, which can be depicted as
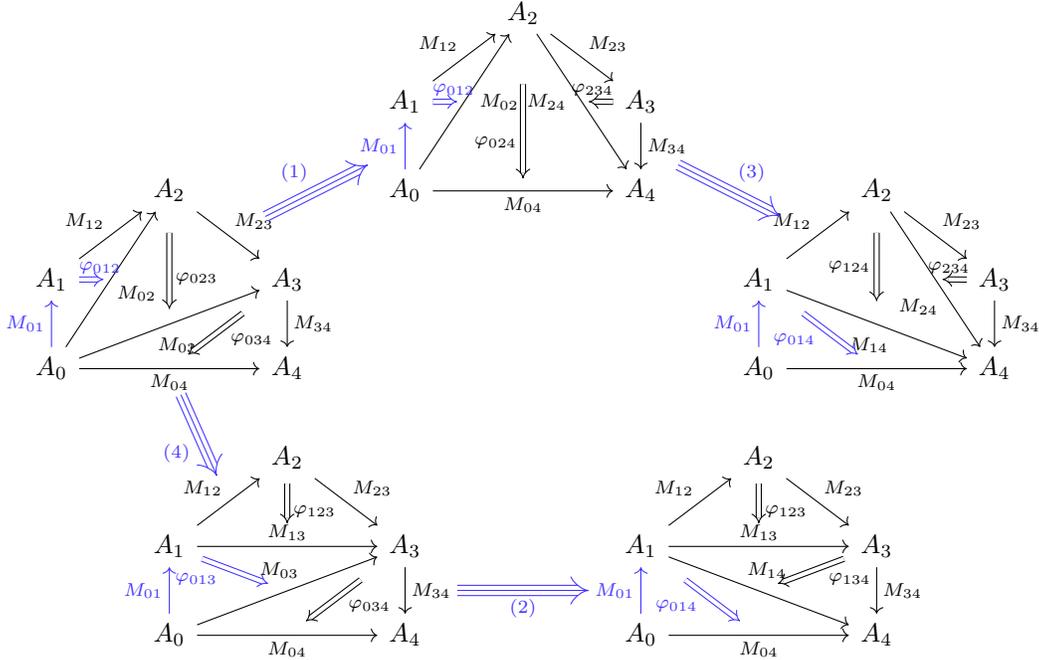
\[\begin{tikzcd}[ampersand replacement=\&]
	\&\&\&\& {A_2} \\
	\&\&\& {A_1} \&\& {A_3} \\
	\& {A_2} \&\& {A_0} \&\& {A_4} \&\& {A_2} \\
	{A_1} \&\& {A_3} \&\&\&\& {A_1} \&\& {A_3} \\
	{A_0} \&\& {A_4} \&\&\&\& {A_0} \&\& {A_4} \\
	\&\& {A_2} \&\&\&\& {A_2} \\
	\& {A_1} \&\& {A_3} \&\& {A_1} \&\& {A_3} \\
	\& {A_0} \&\& {A_4} \&\& {A_0} \&\& {A_4}
	\arrow["{M_{23}}", from=1-5, to=2-6]
	\arrow[""{name=0, anchor=center, inner sep=0}, "{M_{24}}"'{pos=0.4}, from=1-5, to=3-6]
	\arrow["{M_{12}}", from=2-4, to=1-5]
	\arrow[""{name=1, anchor=center, inner sep=0}, "{M_{34}}", from=2-6, to=3-6]
	\arrow[""{name=2, anchor=center, inner sep=0}, "{M_{23}}", from=3-2, to=4-3]
	\arrow[""{name=3, anchor=center, inner sep=0}, "{M_{02}}"'{pos=0.6}, from=3-4, to=1-5]
	\arrow[""{name=4, anchor=center, inner sep=0}, "{M_{01}}", color={rgb,255:red,65;green,51;blue,255}, from=3-4, to=2-4]
	\arrow[""{name=5, anchor=center, inner sep=0}, "{M_{04}}"', from=3-4, to=3-6]
	\arrow["{M_{23}}", from=3-8, to=4-9]
	\arrow[""{name=6, anchor=center, inner sep=0}, "{M_{24}}"'{pos=0.6}, from=3-8, to=5-9]
	\arrow["{M_{12}}", from=4-1, to=3-2]
	\arrow["{M_{34}}", from=4-3, to=5-3]
	\arrow[""{name=7, anchor=center, inner sep=0}, "{M_{12}}", from=4-7, to=3-8]
	\arrow[""{name=8, anchor=center, inner sep=0}, "{M_{14}}"'{pos=0.6}, from=4-7, to=5-9]
	\arrow["{M_{34}}", from=4-9, to=5-9]
	\arrow[""{name=9, anchor=center, inner sep=0}, "{M_{02}}"', from=5-1, to=3-2]
	\arrow["{M_{01}}", color={rgb,255:red,65;green,51;blue,255}, from=5-1, to=4-1]
	\arrow[""{name=10, anchor=center, inner sep=0}, "{M_{03}}"'{pos=0.4}, from=5-1, to=4-3]
	\arrow[""{name=11, anchor=center, inner sep=0}, "{M_{04}}"', from=5-1, to=5-3]
	\arrow["{M_{01}}", color={rgb,255:red,65;green,51;blue,255}, from=5-7, to=4-7]
	\arrow[""{name=12, anchor=center, inner sep=0}, "{M_{04}}"', from=5-7, to=5-9]
	\arrow["{M_{23}}", from=6-3, to=7-4]
	\arrow["{M_{23}}", from=6-7, to=7-8]
	\arrow[""{name=13, anchor=center, inner sep=0}, "{M_{12}}", from=7-2, to=6-3]
	\arrow[""{name=14, anchor=center, inner sep=0}, "{M_{13}}", from=7-2, to=7-4]
	\arrow[""{name=15, anchor=center, inner sep=0}, "{M_{34}}", from=7-4, to=8-4]
	\arrow["{M_{12}}", from=7-6, to=6-7]
	\arrow[""{name=16, anchor=center, inner sep=0}, "{M_{13}}", from=7-6, to=7-8]
	\arrow[""{name=17, anchor=center, inner sep=0}, "{M_{14}}"{pos=0.4}, from=7-6, to=8-8]
	\arrow["{M_{34}}", from=7-8, to=8-8]
	\arrow["{M_{01}}", color={rgb,255:red,65;green,51;blue,255}, from=8-2, to=7-2]
	\arrow[""{name=18, anchor=center, inner sep=0}, "{M_{03}}"{pos=0.6}, from=8-2, to=7-4]
	\arrow[""{name=19, anchor=center, inner sep=0}, "{M_{04}}"', from=8-2, to=8-4]
	\arrow[""{name=20, anchor=center, inner sep=0}, "{M_{01}}", color={rgb,255:red,65;green,51;blue,255}, from=8-6, to=7-6]
	\arrow[""{name=21, anchor=center, inner sep=0}, "{M_{04}}"', from=8-6, to=8-8]
	\arrow["{\varphi_{024}}"'{pos=0.7}, shorten <=19pt, shorten >=4pt, Rightarrow, from=1-5, to=5]
	\arrow["{\varphi_{012}}"{pos=0.7}, color={rgb,255:red,65;green,51;blue,255}, shorten >=2pt, Rightarrow, from=2-4, to=3]
	\arrow["{\varphi_{234}}"'{pos=0.7}, shorten >=2pt, Rightarrow, from=2-6, to=0]
	\arrow["(3)", color={rgb,255:red,65;green,51;blue,255}, shorten <=15pt, shorten >=15pt, Rightarrow, scaling nfold=3, from=1, to=7]
	\arrow["(1)", color={rgb,255:red,65;green,51;blue,255}, shorten <=15pt, shorten >=15pt, Rightarrow, scaling nfold=3, from=2, to=4]
	\arrow["{\varphi_{023}}"{pos=0.6}, shorten <=8pt, shorten >=5pt, Rightarrow, from=3-2, to=10]
	\arrow["{\varphi_{124}}"', shorten <=8pt, shorten >=8pt, Rightarrow, from=3-8, to=8]
	\arrow["{\varphi_{012}}"{pos=0.7}, color={rgb,255:red,65;green,51;blue,255}, shorten >=2pt, Rightarrow, from=4-1, to=9]
	\arrow["{\varphi_{034}}"{pos=0.4}, shorten <=8pt, shorten >=8pt, Rightarrow, from=4-3, to=11]
	\arrow["{\varphi_{014}}"'{pos=0.4}, color={rgb,255:red,65;green,51;blue,255}, shorten <=8pt, shorten >=8pt, Rightarrow, from=4-7, to=12]
	\arrow["{\varphi_{234}}"'{pos=0.6}, shorten >=2pt, Rightarrow, from=4-9, to=6]
	\arrow["(4)"', color={rgb,255:red,65;green,51;blue,255}, shorten <=10pt, shorten >=10pt, Rightarrow, scaling nfold=3, from=11, to=13]
	\arrow["{\varphi_{123}}", shorten <=2pt, shorten >=8pt, Rightarrow, from=6-3, to=14]
	\arrow["{\varphi_{123}}", shorten <=2pt, shorten >=8pt, Rightarrow, from=6-7, to=16]
	\arrow["{\varphi_{013}}"'{pos=0.3}, color={rgb,255:red,65;green,51;blue,255}, shorten <=2pt, shorten >=7pt, Rightarrow, from=7-2, to=18]
	\arrow["(2)"', color={rgb,255:red,65;green,51;blue,255}, shorten <=19pt, shorten >=19pt, Rightarrow, scaling nfold=3, from=15, to=20]
	\arrow["{\varphi_{034}}"{pos=0.4}, shorten <=8pt, shorten >=8pt, Rightarrow, from=7-4, to=19]
	\arrow["{\varphi_{014}}"'{pos=0.4}, color={rgb,255:red,65;green,51;blue,255}, shorten <=8pt, shorten >=8pt, Rightarrow, from=7-6, to=21]
	\arrow["{\varphi_{134}}"{pos=0.3}, shorten <=2pt, shorten >=7pt, Rightarrow, from=7-8, to=17]
\end{tikzcd}\]
we show that there is an equality of $(A_1,A_4)$-bimodule maps from $(M_{12}\otimes_{A_{2}}M_{23})\otimes_{A_3}M_{34}$ to $M_{14}$
\[(\varphi_{123} \otimes_{A_{3}} M_{34}) \bullet \varphi_{134} =\overline{\alpha}_{M_{12}|M_{23}|M_{34}} \bullet (M_{12} \otimes_{A_{2}} \varphi_{234}) \bullet \varphi_{124},\]
obtaining the desired lift $\Delta^0[4]\to \mrt^{\natural}(\cC^{\otimes})$.

Through various instances of \cref{NaturalityAssociatorTensor,pastingII,PentagonTensor}, we obtain
\[
\begin{array}{lllll}
(M_{01} \otimes_{A_{1}} (\varphi_{123} \otimes_{A_{3}} M_{34})) \bullet (M_{01} \otimes_{A_{1}} \varphi_{134}) \bullet \varphi_{014}\quad\quad\quad\quad\quad\quad\quad\quad \quad\quad\quad\quad\quad\quad\quad\quad \\
\quad\quad=\overline{\alpha}_{M_{01}|M_{12}\otimes_{A_2} M_{23}|M_{34}}^{-1} \bullet ((M_{01} \otimes_{A_{1}} \varphi_{123}) \otimes_{A_{3}} M_{34}) \bullet (\varphi_{013} \otimes_{A_{3}} M_{34}) \bullet \varphi_{034}&(2)\\
\quad\quad=\overline{\alpha}_{M_{01}|M_{12}\otimes_{A_2} M_{23}|M_{34}}^{-1} \bullet (\overline{\alpha}_{M_{01}|M_{12}|M_{23}} \otimes_{A_{3}} M_{34})^{-1} \bullet\\
\quad\quad\quad\quad\bullet((\varphi_{012} \otimes_{A_{2}} M_{23}) \otimes_{A_{3}} M_{34}) \bullet (\varphi_{023} \otimes_{A_{3}} M_{34}) \bullet \varphi_{034}&(4)\\
\quad\quad=\overline{\alpha}_{M_{01}|M_{12}\otimes_{A_2} M_{23}|M_{34}}^{-1} \bullet (\overline{\alpha}_{M_{01}|M_{12}|M_{23}} \otimes_{A_{3}} M_{34})^{-1} \bullet \overline{\alpha}_{M_{01}\otimes_{A_1} M_{12}|M_{23}|M_{34}} \bullet \\
\quad\quad\quad\quad\bullet(\varphi_{012} \otimes_{A_{2}} (M_{23} \otimes_{A_{3}} M_{34})) \bullet (M_{02} \otimes_{A_{2}} \varphi_{234}) \bullet \varphi_{024}&(1)\\
\quad\quad=\overline{\alpha}_{M_{01}|M_{12}\otimes_{A_2} M_{23}|M_{34}}^{-1} \bullet (\overline{\alpha}_{M_{01}|M_{12}|M_{23}} \otimes_{A_{3}} M_{34})^{-1}\bullet\overline{\alpha}_{M_{01}\otimes_{A_1} M_{12}|M_{23}|M_{34}}  \bullet \\
\quad\quad\quad\quad\bullet ((M_{01} \otimes_{A_{1}} M_{12}) \otimes_{A_{2}} \varphi_{234}) \bullet (\varphi_{012} \otimes_{A_{2}} M_{24}) \bullet \varphi_{024}&\text{Prop.~\ref{pastingII}}\\
\quad\quad=\overline{\alpha}_{M_{01}|M_{12}\otimes_{A_2} M_{23}|M_{34}}^{-1} \bullet (\overline{\alpha}_{M_{01}|M_{12}|M_{23}} \otimes_{A_{3}} M_{34})^{-1} \bullet \overline{\alpha}_{M_{01}\otimes_{A_1} M_{12}| M_{23}|M_{34}} \bullet \\
\quad\quad\quad\quad\bullet\overline{\alpha}_{M_{01}|M_{12}|M_{23} \otimes_{A_3} M_{34}}\bullet (M_{01} \otimes_{A_{1}} (M_{12} \otimes_{A_{2}} \varphi_{234})) \bullet (\varphi_{012} \otimes_{A_{2}} M_{24}) \bullet \varphi_{024}&(3)\\
\quad\quad=(M_{01} \otimes_{A_{1}} \overline\alpha_{M_{12}|M_{23}|M_{34}}) \bullet (M_{01} \otimes_{A_{1}} (M_{12} \otimes_{A_{2}} \varphi_{234})) \bullet (M_{01} \otimes_{A_{1}} \varphi_{124}) \bullet \varphi_{014}. &\text{Prop.~\ref{PentagonTensor}}
\end{array}
\]

Since $\varphi_{014}$ is an isomorphism, we obtain
\[(M_{01} \otimes_{A_{1}} (\varphi_{123} \otimes_{A_{3}} M_{34})) \bullet (M_{01} \otimes_{A_{1}} \varphi_{134}) = (M_{01} \otimes_{A_{1}} \overline\alpha_{M_{12}|M_{23}|M_{34}}) \bullet (M_{01} \otimes_{A_{1}} (M_{12} \otimes_{A_{2}} \varphi_{234})) \bullet (M_{01} \otimes_{A_{1}} \varphi_{124}).\]
By \cref{pastingII} we obtain
\[M_{01} \otimes_{A_{1}} ((\varphi_{123} \otimes_{A_{3}} M_{34}) \bullet \varphi_{134}) = M_{01} \otimes_{A_{1}} (\overline\alpha_{M_{12}|M_{23}|M_{34}} \bullet (M_{12} \otimes_{A_{2}} \varphi_{234}) \bullet \varphi_{124}).\]
By \cref{ReductionII} we obtain
\[A_1 \otimes_{A_{1}} ((\varphi_{123} \otimes_{A_{3}} M_{34}) \bullet \varphi_{134}) = A_1 \otimes_{A_{1}} (\overline\alpha_{M_{12}|M_{23}|M_{34}} \bullet (M_{12} \otimes_{A_{2}} \varphi_{234}) \bullet \varphi_{124}).\]
By \cref{ReductionI} we obtain
\[(\varphi_{123} \otimes_{A_{3}} M_{34}) \bullet \varphi_{134} =\overline\alpha_{M_{12}|M_{23}|M_{34}} \bullet (M_{12} \otimes_{A_{2}} \varphi_{234}) \bullet \varphi_{124},\]
as desired.
\end{proof}

\begin{prop}
\label{horn41}
The marked simplicial set $\mrt^{\natural}(\cC^{\otimes})$ has the right lifting property with respect to the complicial horn inclusion $\Lambda^k[4]\to\Delta^k[4]$ for $k=1,3$.
\end{prop}

\begin{proof}
We treat the case $k=1$, the case $k=3$ being analogous. 
Given a map of marked simplicial sets $\Lambda^1[4]\to \mrt^{\natural}(\cC^{\otimes})$, which can be depicted as
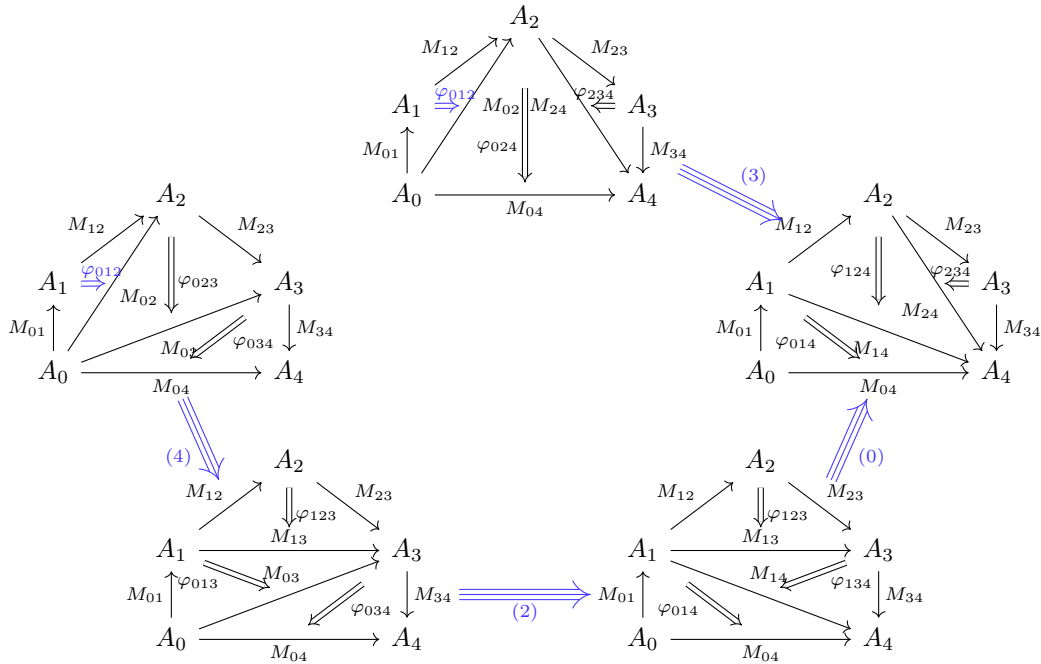
\[\begin{tikzcd}[ampersand replacement=\&]
	\&\&\&\& {A_2} \\
	\&\&\& {A_1} \&\& {A_3} \\
	\& {A_2} \&\& {A_0} \&\& {A_4} \&\& {A_2} \\
	{A_1} \&\& {A_3} \&\&\&\& {A_1} \&\& {A_3} \\
	{A_0} \&\& {A_4} \&\&\&\& {A_0} \&\& {A_4} \\
	\&\& {A_2} \&\&\&\& {A_2} \\
	\& {A_1} \&\& {A_3} \&\& {A_1} \&\& {A_3} \\
	\& {A_0} \&\& {A_4} \&\& {A_0} \&\& {A_4}
	\arrow["{M_{23}}", from=1-5, to=2-6]
	\arrow[""{name=0, anchor=center, inner sep=0}, "{M_{24}}"'{pos=0.4}, from=1-5, to=3-6]
	\arrow["{M_{12}}", from=2-4, to=1-5]
	\arrow[""{name=1, anchor=center, inner sep=0}, "{M_{34}}", from=2-6, to=3-6]
	\arrow["{M_{23}}", from=3-2, to=4-3]
	\arrow[""{name=2, anchor=center, inner sep=0}, "{M_{02}}"'{pos=0.6}, from=3-4, to=1-5]
	\arrow["{M_{01}}", from=3-4, to=2-4]
	\arrow[""{name=3, anchor=center, inner sep=0}, "{M_{04}}"', from=3-4, to=3-6]
	\arrow["{M_{23}}", from=3-8, to=4-9]
	\arrow[""{name=4, anchor=center, inner sep=0}, "{M_{24}}"'{pos=0.6}, from=3-8, to=5-9]
	\arrow["{M_{12}}", from=4-1, to=3-2]
	\arrow["{M_{34}}", from=4-3, to=5-3]
	\arrow[""{name=5, anchor=center, inner sep=0}, "{M_{12}}", from=4-7, to=3-8]
	\arrow[""{name=6, anchor=center, inner sep=0}, "{M_{14}}"'{pos=0.6}, from=4-7, to=5-9]
	\arrow["{M_{34}}", from=4-9, to=5-9]
	\arrow[""{name=7, anchor=center, inner sep=0}, "{M_{02}}"', from=5-1, to=3-2]
	\arrow["{M_{01}}", from=5-1, to=4-1]
	\arrow[""{name=8, anchor=center, inner sep=0}, "{M_{03}}"'{pos=0.4}, from=5-1, to=4-3]
	\arrow[""{name=9, anchor=center, inner sep=0}, "{M_{04}}"', from=5-1, to=5-3]
	\arrow["{M_{01}}", from=5-7, to=4-7]
	\arrow[""{name=10, anchor=center, inner sep=0}, "{M_{04}}"', from=5-7, to=5-9]
	\arrow["{M_{23}}", from=6-3, to=7-4]
	\arrow[""{name=11, anchor=center, inner sep=0}, "{M_{23}}", from=6-7, to=7-8]
	\arrow[""{name=12, anchor=center, inner sep=0}, "{M_{12}}", from=7-2, to=6-3]
	\arrow[""{name=13, anchor=center, inner sep=0}, "{M_{13}}", from=7-2, to=7-4]
	\arrow[""{name=14, anchor=center, inner sep=0}, "{M_{34}}", from=7-4, to=8-4]
	\arrow["{M_{12}}", from=7-6, to=6-7]
	\arrow[""{name=15, anchor=center, inner sep=0}, "{M_{13}}", from=7-6, to=7-8]
	\arrow[""{name=16, anchor=center, inner sep=0}, "{M_{14}}"{pos=0.4}, from=7-6, to=8-8]
	\arrow["{M_{34}}", from=7-8, to=8-8]
	\arrow["{M_{01}}", from=8-2, to=7-2]
	\arrow[""{name=17, anchor=center, inner sep=0}, "{M_{03}}"{pos=0.6}, from=8-2, to=7-4]
	\arrow[""{name=18, anchor=center, inner sep=0}, "{M_{04}}"', from=8-2, to=8-4]
	\arrow[""{name=19, anchor=center, inner sep=0}, "{M_{01}}", from=8-6, to=7-6]
	\arrow[""{name=20, anchor=center, inner sep=0}, "{M_{04}}"', from=8-6, to=8-8]
	\arrow["{\varphi_{024}}"'{pos=0.7}, shorten <=19pt, shorten >=4pt, Rightarrow, from=1-5, to=3]
	\arrow["{\varphi_{012}}"{pos=0.7}, color={rgb,255:red,65;green,51;blue,255}, shorten >=2pt, Rightarrow, from=2-4, to=2]
	\arrow["{\varphi_{234}}"'{pos=0.7}, shorten >=2pt, Rightarrow, from=2-6, to=0]
	\arrow["(3)", color={rgb,255:red,65;green,51;blue,255}, shorten <=15pt, shorten >=15pt, Rightarrow, scaling nfold=3, from=1, to=5]
	\arrow["{\varphi_{023}}"{pos=0.6}, shorten <=8pt, shorten >=5pt, Rightarrow, from=3-2, to=8]
	\arrow["{\varphi_{124}}"', shorten <=8pt, shorten >=8pt, Rightarrow, from=3-8, to=6]
	\arrow["{\varphi_{012}}"{pos=0.7}, color={rgb,255:red,65;green,51;blue,255}, shorten >=2pt, Rightarrow, from=4-1, to=7]
	\arrow["{\varphi_{034}}"{pos=0.4}, shorten <=8pt, shorten >=8pt, Rightarrow, from=4-3, to=9]
	\arrow["{\varphi_{014}}"'{pos=0.4}, shorten <=8pt, shorten >=8pt, Rightarrow, from=4-7, to=10]
	\arrow["{\varphi_{234}}"'{pos=0.6}, shorten >=2pt, Rightarrow, from=4-9, to=4]
	\arrow["(4)"', color={rgb,255:red,65;green,51;blue,255}, shorten <=10pt, shorten >=10pt, Rightarrow, scaling nfold=3, from=9, to=12]
	\arrow["{\varphi_{123}}", shorten <=2pt, shorten >=8pt, Rightarrow, from=6-3, to=13]
	\arrow["(0)"', color={rgb,255:red,65;green,51;blue,255}, shorten <=10pt, shorten >=10pt, Rightarrow, scaling nfold=3, from=11, to=10]
	\arrow["{\varphi_{123}}", shorten <=2pt, shorten >=8pt, Rightarrow, from=6-7, to=15]
	\arrow["{\varphi_{013}}"'{pos=0.3}, shorten <=2pt, shorten >=7pt, Rightarrow, from=7-2, to=17]
	\arrow["(2)"', color={rgb,255:red,65;green,51;blue,255}, shorten <=19pt, shorten >=19pt, Rightarrow, scaling nfold=3, from=14, to=19]
	\arrow["{\varphi_{034}}"{pos=0.4}, shorten <=8pt, shorten >=8pt, Rightarrow, from=7-4, to=18]
	\arrow["{\varphi_{014}}"'{pos=0.4}, shorten <=8pt, shorten >=8pt, Rightarrow, from=7-6, to=20]
	\arrow["{\varphi_{134}}"{pos=0.3}, shorten <=2pt, shorten >=7pt, Rightarrow, from=7-8, to=16]
\end{tikzcd}\]
we show that there is an equality of $(A_0,A_4)$-bimodule maps from $(M_{02}\otimes_{A_{23}}M_{23})\otimes_{A_{34}}M_{34}$ to $M_{04}$
\begin{equation}
\label{eq41}
(\varphi_{023} \otimes_{A_{3}} M_{34}) \bullet \varphi_{034} = \overline\alpha_{M_{02}|M_{23}|M_{34}} \bullet (M_{02} \otimes_{A_{2}} \varphi_{234}) \bullet \varphi_{024},    
\end{equation}
obtaining the desired lift $\Delta^1[4]\to \mrt^{\natural}(\cC^{\otimes})$.

Through various instances of \cref{NaturalityAssociatorTensor,RmkEpi,pastingII,PentagonTensor}, we obtain
\[
\begin{array}{lllll}
((\varphi_{012} \otimes_{A_{2}} M_{23}) \otimes_{A_{3}} M_{34}) \bullet (\varphi_{023} \otimes_{A_{3}} M_{34}) \bullet \varphi_{034}\quad\quad\quad\quad\quad\quad\quad\quad \quad\quad\quad\quad\quad\quad\quad\quad \\
\quad\quad=(\overline{\alpha}_{M_{01}|M_{12}| M_{23}} \otimes_{A_{3}} M_{34}) \bullet ((M_{01} \otimes_{A_{1}} \varphi_{123}) \otimes_{A_{3}} M_{34}) \bullet (\varphi_{013} \otimes_{A_{3}} M_{34}) \bullet \varphi_{034}&(4)\\
\quad\quad=(\overline{\alpha}_{M_{01}|M_{12}|M_{23}} \otimes_{A_{3}} M_{34}) \bullet\overline{\alpha}_{M_{01}|M_{12} \otimes_{A_2} M_{23}|M_{34}} \bullet\\
\quad\quad\quad\quad\bullet (M_{01} \otimes_{A_{1}} (\varphi_{123} \otimes_{A_{3}} M_{34})) \bullet (M_{01} \otimes_{A_{1}} \varphi_{134}) \bullet \varphi_{014}&(2)\\
\quad\quad=(\overline{\alpha}_{M_{01}|M_{12}|M_{23}} \otimes_{A_{3}} M_{34}) \bullet \overline{\alpha}_{M_{01}| M_{12} \otimes_{A_2} M_{23}| M_{34}}\bullet (M_{01} \otimes_{A_{1}} \overline{\alpha}_{M_{12}|M_{23}|M_{34}})  \bullet\\
\quad\quad\quad\quad\bullet (M_{01} \otimes_{A_{1}} (M_{12} \otimes_{A_{2}} \varphi_{234})) \bullet (M_{01} \otimes_{A_{1}} \varphi_{124}) \bullet \varphi_{014}&(0)\\
\quad\quad=(\overline{\alpha}_{M_{01}|M_{12}|M_{23}} \otimes_{A_{3}} M_{34}) \bullet \overline{\alpha}_{M_{01}|M_{12} \otimes_{A_2} M_{23}|M_{34}} \bullet (M_{01} \otimes_{A_{1}} \overline{\alpha}_{M_{12}|M_{23}|M_{34}}) \bullet \\
\quad\quad\quad\quad\bullet\overline{\alpha}_{M_{01}|M_{12}| M_{23} \otimes_{A_3} M_{34}}^{-1} \bullet ((M_{01} \otimes_{A_{1}} M_{12}) \otimes_{A_{2}} \varphi_{234}) \bullet (\varphi_{012} \otimes_{A_{2}} M_{24}) \bullet \varphi_{024}&(3)\\
\quad\quad=(\overline{\alpha}_{M_{01}|M_{12}|M_{23}} \otimes_{A_{3}} M_{34}) \bullet \overline{\alpha}_{M_{01}|M_{12} \otimes_{A_2} M_{23}|M_{34}} \bullet (M_{01} \otimes_{A_{1}} \overline{\alpha}_{M_{12}|M_{23}|M_{34}}) \bullet\\
\quad\quad\quad\quad\bullet \overline{\alpha}_{M_{01}|M_{12}|M_{23} \otimes_{A_3} M_{34}}^{-1} \bullet (\varphi_{012} \otimes_{A_{2}} (M_{23} \otimes_{A_{3}} M_{34})) \bullet (M_{02} \otimes_{A_{2}} \varphi_{234}) \bullet \varphi_{024}&\text{Prop.~\ref{pastingII}}\\
\quad\quad=\overline\alpha_{M_{01} \otimes M_{12}|M_{23}|M_{34}} \bullet (\varphi_{012} \otimes_{A_{2}} (M_{23} \otimes_{A_{3}} M_{34})) \bullet (M_{02} \otimes_{A_{2}} \varphi_{234}) \bullet \varphi_{024} &\text{Prop.~\ref{NaturalityAssociatorTensor}}\\
\quad\quad= ((\varphi_{012} \otimes_{A_{2}} M_{23}) \otimes_{A_{3}} M_{34}) \bullet \overline\alpha_{M_{02}|M_{23}|M_{34}} \bullet (M_{02} \otimes_{A_{2}} \varphi_{234}) \bullet \varphi_{024}.
\end{array}
\]
Since $(\varphi_{012} \otimes_{A_{2}} M_{23}) \otimes_{A_{3}} M_{34}$ is an isomorphism by \cref{RmkEpi},
we obtain \eqref{eq41}, as desired.
\end{proof}

\begin{prop}
\label{horn42}
The marked simplicial set $\mrt^{\natural}(\cC^{\otimes})$ has the right lifting property with respect to the complicial horn inclusion $\Lambda^2[4]\to\Delta^2[4]$.
\end{prop}

\begin{proof}
Given a map of marked simplicial sets $\Lambda^2[4]\to \mrt^{\natural}(\cC^{\otimes})$, which can be depicted as
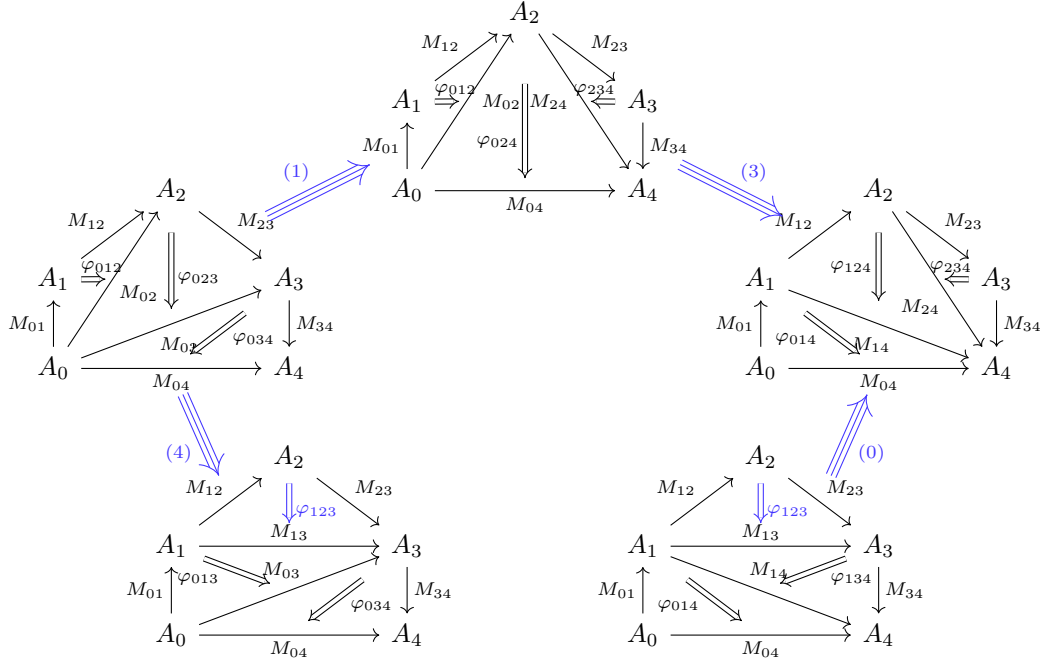
\[\begin{tikzcd}[ampersand replacement=\&]
	\&\&\&\& {A_2} \\
	\&\&\& {A_1} \&\& {A_3} \\
	\& {A_2} \&\& {A_0} \&\& {A_4} \&\& {A_2} \\
	{A_1} \&\& {A_3} \&\&\&\& {A_1} \&\& {A_3} \\
	{A_0} \&\& {A_4} \&\&\&\& {A_0} \&\& {A_4} \\
	\&\& {A_2} \&\&\&\& {A_2} \\
	\& {A_1} \&\& {A_3} \&\& {A_1} \&\& {A_3} \\
	\& {A_0} \&\& {A_4} \&\& {A_0} \&\& {A_4}
	\arrow["{M_{23}}", from=1-5, to=2-6]
	\arrow[""{name=0, anchor=center, inner sep=0}, "{M_{24}}"'{pos=0.4}, from=1-5, to=3-6]
	\arrow["{M_{12}}", from=2-4, to=1-5]
	\arrow[""{name=1, anchor=center, inner sep=0}, "{M_{34}}", from=2-6, to=3-6]
	\arrow[""{name=2, anchor=center, inner sep=0}, "{M_{23}}", from=3-2, to=4-3]
	\arrow[""{name=3, anchor=center, inner sep=0}, "{M_{02}}"'{pos=0.6}, from=3-4, to=1-5]
	\arrow[""{name=4, anchor=center, inner sep=0}, "{M_{01}}", from=3-4, to=2-4]
	\arrow[""{name=5, anchor=center, inner sep=0}, "{M_{04}}"', from=3-4, to=3-6]
	\arrow["{M_{23}}", from=3-8, to=4-9]
	\arrow[""{name=6, anchor=center, inner sep=0}, "{M_{24}}"'{pos=0.6}, from=3-8, to=5-9]
	\arrow["{M_{12}}", from=4-1, to=3-2]
	\arrow["{M_{34}}", from=4-3, to=5-3]
	\arrow[""{name=7, anchor=center, inner sep=0}, "{M_{12}}", from=4-7, to=3-8]
	\arrow[""{name=8, anchor=center, inner sep=0}, "{M_{14}}"'{pos=0.6}, from=4-7, to=5-9]
	\arrow["{M_{34}}", from=4-9, to=5-9]
	\arrow[""{name=9, anchor=center, inner sep=0}, "{M_{02}}"', from=5-1, to=3-2]
	\arrow["{M_{01}}", from=5-1, to=4-1]
	\arrow[""{name=10, anchor=center, inner sep=0}, "{M_{03}}"'{pos=0.4}, from=5-1, to=4-3]
	\arrow[""{name=11, anchor=center, inner sep=0}, "{M_{04}}"', from=5-1, to=5-3]
	\arrow["{M_{01}}", from=5-7, to=4-7]
	\arrow[""{name=12, anchor=center, inner sep=0}, "{M_{04}}"', from=5-7, to=5-9]
	\arrow["{M_{23}}", from=6-3, to=7-4]
	\arrow[""{name=13, anchor=center, inner sep=0}, "{M_{23}}", from=6-7, to=7-8]
	\arrow[""{name=14, anchor=center, inner sep=0}, "{M_{12}}", from=7-2, to=6-3]
	\arrow[""{name=15, anchor=center, inner sep=0}, "{M_{13}}", from=7-2, to=7-4]
	\arrow["{M_{34}}", from=7-4, to=8-4]
	\arrow["{M_{12}}", from=7-6, to=6-7]
	\arrow[""{name=16, anchor=center, inner sep=0}, "{M_{13}}", from=7-6, to=7-8]
	\arrow[""{name=17, anchor=center, inner sep=0}, "{M_{14}}"{pos=0.4}, from=7-6, to=8-8]
	\arrow["{M_{34}}", from=7-8, to=8-8]
	\arrow["{M_{01}}", from=8-2, to=7-2]
	\arrow[""{name=18, anchor=center, inner sep=0}, "{M_{03}}"{pos=0.6}, from=8-2, to=7-4]
	\arrow[""{name=19, anchor=center, inner sep=0}, "{M_{04}}"', from=8-2, to=8-4]
	\arrow["{M_{01}}", from=8-6, to=7-6]
	\arrow[""{name=20, anchor=center, inner sep=0}, "{M_{04}}"', from=8-6, to=8-8]
	\arrow["{\varphi_{024}}"'{pos=0.7}, shorten <=19pt, shorten >=4pt, Rightarrow, from=1-5, to=5]
	\arrow["{\varphi_{012}}"{pos=0.7}, shorten >=2pt, Rightarrow, from=2-4, to=3]
	\arrow["{\varphi_{234}}"'{pos=0.7}, shorten >=2pt, Rightarrow, from=2-6, to=0]
	\arrow["(3)", color={rgb,255:red,65;green,51;blue,255}, shorten <=15pt, shorten >=15pt, Rightarrow, scaling nfold=3, from=1, to=7]
	\arrow["(1)", color={rgb,255:red,65;green,51;blue,255}, shorten <=15pt, shorten >=15pt, Rightarrow, scaling nfold=3, from=2, to=4]
	\arrow["{\varphi_{023}}"{pos=0.6}, shorten <=8pt, shorten >=5pt, Rightarrow, from=3-2, to=10]
	\arrow["{\varphi_{124}}"', shorten <=8pt, shorten >=8pt, Rightarrow, from=3-8, to=8]
	\arrow["{\varphi_{012}}"{pos=0.7}, shorten >=2pt, Rightarrow, from=4-1, to=9]
	\arrow["{\varphi_{034}}"{pos=0.4}, shorten <=8pt, shorten >=8pt, Rightarrow, from=4-3, to=11]
	\arrow["{\varphi_{014}}"'{pos=0.4}, shorten <=8pt, shorten >=8pt, Rightarrow, from=4-7, to=12]
	\arrow["{\varphi_{234}}"'{pos=0.6}, shorten >=2pt, Rightarrow, from=4-9, to=6]
	\arrow["(4)"', color={rgb,255:red,65;green,51;blue,255}, shorten <=10pt, shorten >=10pt, Rightarrow, scaling nfold=3, from=11, to=14]
	\arrow["{\varphi_{123}}", color={rgb,255:red,65;green,51;blue,255}, shorten <=2pt, shorten >=8pt, Rightarrow, from=6-3, to=15]
	\arrow["(0)"', color={rgb,255:red,65;green,51;blue,255}, shorten <=10pt, shorten >=10pt, Rightarrow, scaling nfold=3, from=13, to=12]
	\arrow["{\varphi_{123}}", color={rgb,255:red,65;green,51;blue,255}, shorten <=2pt, shorten >=8pt, Rightarrow, from=6-7, to=16]
	\arrow["{\varphi_{013}}"'{pos=0.3}, shorten <=2pt, shorten >=7pt, Rightarrow, from=7-2, to=18]
	\arrow["{\varphi_{034}}"{pos=0.4}, shorten <=8pt, shorten >=8pt, Rightarrow, from=7-4, to=19]
	\arrow["{\varphi_{014}}"'{pos=0.4}, shorten <=8pt, shorten >=8pt, Rightarrow, from=7-6, to=20]
	\arrow["{\varphi_{134}}"{pos=0.3}, shorten <=2pt, shorten >=7pt, Rightarrow, from=7-8, to=17]
\end{tikzcd}\]
we show that there is an equality of $(A_0,A_4)$-bimodule maps from $(M_{01}\otimes_{A_1}M_{13})\otimes_{A_3}M_{3,4}$ to $M_{0,4}$
\begin{equation}
\label{eq42}
(\varphi_{013} \otimes_{A_{3}} M_{34}) \bullet \varphi_{034} = \overline\alpha_{M_{01}|M_{13}|M_{34}}\bullet (M_{01} \otimes_{A_{1}} \varphi_{134}) \bullet \varphi_{014},
\end{equation}
obtaining the desired lift $\Delta^2[4]\to \mrt^{\natural}(\cC^{\otimes})$.

Through various instances of \cref{NaturalityAssociatorTensor,RmkEpi,pastingII,PentagonTensor}, we obtain
we have the equality of $(A_0,A_4)$-bimodules from $M_{01}\otimes_{A_1}M_{13}\otimes_{A_3}M_{34}$ to $M_{04}$:
\[
\begin{array}{llll}
((M_{01} \otimes_{A_{1}} \varphi_{123}) \otimes_{A_{3}} M_{34}) \bullet (\varphi_{013} \otimes_{A_{3}} M_{34}) \bullet \varphi_{034}\quad\quad\quad\quad\quad\quad\quad\quad\\
\quad\quad=(\overline{\alpha}_{M_{01}|M_{12}|M_{23}} \otimes_{A_{3}} M_{34})^{-1} \bullet ((\varphi_{012} \otimes_{A_{2}} M_{23}) \otimes_{A_{3}} M_{34}) \bullet (\varphi_{023} \otimes_{A_{3}} M_{34}) \bullet \varphi_{034}&(4)\\
\quad\quad=(\overline{\alpha}_{M_{01}|M_{12}|M_{23}} \otimes_{A_{3}} M_{34})^{-1} \bullet \overline{\alpha}_{M_{01} \otimes_{A_1} M_{12}|M_{23}|M_{34}} \bullet \\
\quad\quad\quad\quad\bullet(\varphi_{012} \otimes_{A_{2}} (M_{23} \otimes_{A_{3}} M_{34})) \bullet (M_{02} \otimes_{A_{2}} \varphi_{234}) \bullet \varphi_{024}&(1)\\
\quad\quad=(\overline{\alpha}_{M_{01}|M_{12}|M_{23}} \otimes_{A_{3}} M_{34})^{-1} \bullet
\overline{\alpha}_{M_{01} \otimes_{A_1} M_{12}|M_{23}|M_{34}} \bullet \\
\quad\quad\quad\quad\bullet((M_{01} \otimes_{A_{1}} M_{12}) \otimes_{A_{2}} \varphi_{234}) \bullet (\varphi_{012} \otimes_{A_{2}} M_{24}) \bullet \varphi_{024}&\text{Prop.~\ref{pastingII}}\\
\quad\quad=(\overline{\alpha}_{M_{01}|M_{12}|M_{23}} \otimes_{A_{3}} M_{34})^{-1} \bullet \overline{\alpha}_{M_{01} \otimes_{A_1} M_{12}|M_{23}|M_{34}} \bullet \overline{\alpha}_{M_{01}|M_{12}|M_{23} \otimes_{A_3} M_{34}} \bullet\\
\quad\quad\quad\quad\bullet (M_{01} \otimes_{A_{1}} (M_{12} \otimes_{A_{2}} \varphi_{234})) \bullet (M_{01} \otimes_{A_{1}} \varphi_{124}) \bullet \varphi_{014}&(3)\\
\quad\quad=(\overline{\alpha}_{M_{01}|M_{12}|M_{23}} \otimes_{A_{3}} M_{34})^{-1} \bullet \overline{\alpha}_{M_{01} \otimes_{A_1} M_{12}|M_{23}|M_{34}} \bullet  \overline{\alpha}_{M_{01}|M_{12}|M_{23} \otimes_{A_3} M_{34}}\bullet \\
   \quad\quad\quad\quad\bullet (M_{01} \otimes_{A_{1}} \overline{\alpha}_{M_{12}|M_{23}|M_{34}})^{-1} \bullet (M_{01} \otimes_{A_{1}} (\varphi_{123} \otimes_{A_{3}} M_{34})) \bullet (M_{01} \otimes_{A_{1}} \varphi_{134}) \bullet \varphi_{014}&(0)\\
\quad\quad= \overline\alpha_{M_{01}|M_{12} \otimes M_{23}|M_{34}} \bullet (M_{01} \otimes_{A_{1}} (\varphi_{123} \otimes_{A_{3}} M_{34})) \bullet (M_{01} \otimes_{A_{1}} \varphi_{134}) \bullet \varphi_{014} &\text{Prop.~\ref{PentagonTensor}}\\
\quad\quad= ((M_{01} \otimes_{A_{1}} \varphi_{123}) \otimes_{A_{3}} M_{34}) \bullet \overline\alpha_{M_{01}|M_{13}|M_{34}} \bullet (M_{01} \otimes_{A_{1}} \varphi_{134}) \bullet \varphi_{014}.
\end{array}
\]
Since $(M_{01} \otimes_{A_{1}} \varphi_{123}) \otimes_{A_{3}} M_{34}$ is an isomorphism by \cref{RmkEpi},
we obtain
\eqref{eq42}, as desired.
\end{proof}

\begin{prop}
\label{hornHigh}
The marked simplicial set $\mrt^{\natural}(\cC^{\otimes})$ has the right lifting property with respect to the map $\Lambda^k[m]\to\Delta^k[m]$ for $m>4$ and $0\leq k\leq m$.
\end{prop}

\begin{proof}
Since $\mrt^{\natural}(\cC^{\otimes})$ is $3$-coskeletal, the two lifting problems below are equivalent:
\[
\begin{tikzcd}
    \Lambda^k[m]\arrow[r]\arrow[d]&\mrt^{\natural}(\cC^{\otimes})\\
    \Delta^k[m]\arrow[ru,dashed]&
\end{tikzcd}
\quad\quad
\leftrightsquigarrow
\quad\quad
\begin{tikzcd}
\mathrm{sk}_3\Lambda^k[m]\arrow[r]\arrow[d]&\mrt^{\natural}(\cC^{\otimes})\\
\mathrm{sk}_3\Delta^k[m]\arrow[ru,dashed]&
\end{tikzcd}
\]
Given that the vertical map in the right hand lifting problem is an isomorphism, this lifting problem has a solution. Hence, so does the left hand lifting problem, as desired.
\end{proof}

\subsection{Complicial thinness}

In this subsection we verify that $\mrt^\natural(\cC^{\otimes})$ lifts against all complicial thinness extensions:

\begin{thm}
\label{thin}
The marked simplicial set $ \mrt^\natural(\cC^{\otimes})$ has the right lifting property with respect to the map $\Delta^k[m]'\to\Delta^k[m]''$ for $m\geq0$ and $0\leq k\leq m$. 
\end{thm}

\begin{proof}
We treat the various cases as \cref{thin20,thin21,thin30,thin31,thinhigh}.
\end{proof}

\begin{prop}
\label{thin21}
The marked simplicial set $\mrt^{\natural}(\cC^{\otimes})$ has the right lifting property with respect to the map $\Delta^1[2]'\to\Delta^1[2]''$
\end{prop}
\begin{proof}
Given a map of marked simplicial sets $\Delta^1[2]\to \mrt^{\natural}(\cC^{\otimes})$, which can be depicted as
\[\begin{tikzcd}[ampersand replacement=\&]
	{} \& A_1 \\
	A_0 \&\& A_2
	\arrow["M_{12}", color={rgb,255:red,65;green,51;blue,255}, from=1-2, to=2-3]
	\arrow["M_{01}", color={rgb,255:red,65;green,51;blue,255}, from=2-1, to=1-2]
	\arrow[""{name=0, anchor=center, inner sep=0}, "M_{02}"', from=2-1, to=2-3]
	\arrow["\varphi", color={rgb,255:red,65;green,51;blue,255}, shorten <=5pt, shorten >=5pt, Rightarrow, from=1-2, to=0]
\end{tikzcd}\]
we show that $M_{02}$ is an $(A_0,A_1)$-equivalence, obtaining the desired lift $\Delta^1[2]''\to \mrt^{\natural}(\cC^{\otimes})$.
By assumption there exist $M_{01}'$ a $(A_1,A_0)$-bimodule, $M_{12}'$ a $(A_2,A_1)$-bimodule, as well as an isomorphism of $(A_1,A_1)$-bimodules and an isomorphism of $(A_2,A_2)$-bimodules
\[M_{01}'\otimes_{A_0}M_{01}\cong A_1\quad\text{ and }\quad
M_{12}'\otimes_{A_1}M_{12}\cong A_2.
\]
Define $M_{02}' \coloneqq M_{12}' \otimes_B M_{01}'$, with the $(A_2,A_0)$-bimodule structure from \cref{CompositeBimodule}. On the one hand there is an isomorphism of $(A_2,A_2)$-bimodules
\[\begin{array}{lll}
  M_{02}'\otimes_{A_0}M_{02} &=(M_{12}' \otimes_{A_1} M_{01}') \otimes_{A_0} M_{02} &\text{definition}\\
  & \cong M_{12}' \otimes_{A_1} (M_{01}' \otimes_{A_0} M_{02}) &\text{\cref{AssociativityForBimodules}}\\
   & \cong M_{12}' \otimes_{A_1} (M_{01}' \otimes_{A_0} (M_{01}\otimes_{A_1}M_{12})) &\text{$\varphi$ marked}\\
   & \cong M_{12}' \otimes_{A_1} ((M_{01}' \otimes_{A_0} M_{01})\otimes_{A_1}M_{12}) &\text{\cref{AssociativityForBimodules}}\\
   & \cong M_{12}' \otimes_{A_1} (A_1\otimes_{A_1}M_{12}) &\text{}\\
   & \cong M_{12}' \otimes_{A_1} M_{12}\cong A_2&\text{\cref{TensorWithIso}}. 
\end{array}
\]
Similarly, there is an isomorphism of $(A_0,A_0)$-bimodules
\[M_{02} \otimes_{A_2} M_{02}' \cong A_0.\]
So $M_{02}$ is an $(A_0,A_2)$-equivalence, as desired.
\end{proof}

\begin{prop}
\label{thin20}
The marked simplicial set $\mrt^{\natural}(\cC^{\otimes})$ has the right lifting property with respect to the map $\Delta^k[2]'\to\Delta^k[2]''$ for $k=0,1$.
\end{prop}

\begin{proof}
We treat the case $k=0$, the case $k=2$ being analogous.
Given a map of marked simplicial sets $\Delta^0[2]'\to \mrt^{\natural}(\cC^{\otimes})$, which can be depicted as   
\[\begin{tikzcd}[ampersand replacement=\&]
	\& A_1 \\
	A_0 \&\& A_2
	\arrow["M_{12}", from=1-2, to=2-3]
	\arrow["M_{01}", color={rgb,255:red,65;green,51;blue,255}, from=2-1, to=1-2]
	\arrow[""{name=0, anchor=center, inner sep=0}, "M_{02}"', color={rgb,255:red,65;green,51;blue,255}, from=2-1, to=2-3]
	\arrow["\varphi", color={rgb,255:red,65;green,51;blue,255}, shorten <=5pt, shorten >=5pt, Rightarrow, from=1-2, to=0]
\end{tikzcd}\]
we show that $M_{12}$ is an equivalence, obtaining a lift $\Delta^0[2]''\to \mrt^{\natural}(\cC^{\otimes})$.
By assumption there exist $M_{01}'$ a $(A_1,A_0)$-bimodule, $M_{02}'$ a $(A_2,A_0)$-bimodule, as well as an isomorphism of $(A_1,A_1)$-bimodules and an isomorphism of $(A_2,A_2)$-bimodules
\[M_{01}'\otimes_{A_0}M_{01}\cong A_1\quad\text{ and }\quad
M_{02}'\otimes_{A_0}M_{02}\cong A_2,
\]
and
an $(A_0,A_2)$-bimodule isomorphism
\[\varphi\colon M_{01}\otimes_{A_1}M_{12}\cong M_{02}.\]
Define $M_{12}' \coloneqq M_{02}' \otimes_{A_0} M_{01}$ with the $(A_2,A_1)$-bimodule structure from \cref{CompositeBimodule}.
On the one hand, there is an isomorphism of $(A_2,A_2)$-bimodules
\[
\begin{array}{llllll}
M_{12}'\otimes_{A_1}M_{12}&=(M_{02}' \otimes_{A_0} M_{01}) \otimes_{A_1} M_{12}&\text{definition}\\
&\cong M_{02}' \otimes_{A_0} (M_{01} \otimes_{A_1} M_{12}) &\text{\cref{AssociativityForBimodules}}\\
&\cong M_{02}' \otimes_{A_0} M_{02}\cong A_2.&\text{\cref{TensorWithIso}}
\end{array}
\]
We obtain a $(A_1,A_0)$-bimodule isomorphism in $\cC$:
\[
\begin{array}{llll}
   M_{12}\otimes_{A_2}M_{02}'  & \cong A_1\otimes_{A_1} (M_{12}\otimes_{A_2}M_{02}')&\text{\cref{UnitalityForBimodules}}\\
& \cong  (M_{01}'\otimes_{A_0}M_{01})\otimes_{A_1} (M_{12}\otimes_{A_2}M_{02}') \\
& \cong M_{01}'\otimes_{A_0}((M_{01}\otimes_{A_1} M_{12})\otimes_{A_2}M_{02}') &\text{\cref{AssociativityForBimodules}}\\
& \cong M_{01}'\otimes_{A_0} (M_{02}\otimes_{A_2}M_{02}') \\
& \cong M_{01}'\otimes_{A_0} A_0 \cong M_{01}'&\text{\cref{UnitalityForBimodules}}\\
\end{array}
\]
On the other hand, there is an isomorphism of $(A_1,A_1)$-bimodules:
\[
\begin{array}{llll}
M_{12}\otimes_{A_2}M_{12}'&=M_{12} \otimes_{A_2} (M_{02}' \otimes_{A_0} M_{01})&\text{definition}\\
     & \cong (M_{12} \otimes_{A_2} M_{02}') \otimes_{A_0} M_{01}&\text{\cref{AssociativityForBimodules}}\\
     &\cong  M_{01}' \otimes_{A_0} M_{01}\cong A_2.&\text{\cref{TensorWithIso}}
\end{array}
\]
So $M_{12}$ is a $(A_1,A_2)$-equivalence, as desired.
\end{proof}

\begin{prop}
\label{thin30}
The marked simplicial set $\mrt^{\natural}(\cC^{\otimes})$ has the right lifting property with respect to the map $\Delta^k[3]'\to\Delta^k[3]''$ for $k=0,3$.
\end{prop}

\begin{proof}
We treat the case $k=0$, the case $k=3$ being analogous.
Given a map of marked simplicial sets $\Delta^0[3]'\to \mrt^{\natural}(\cC^{\otimes})$, which can be depicted as
\[\begin{tikzcd}[sep=large, ampersand replacement=\&]
	A_1 \& A_2 \& A_1 \& A_2 \\
	A_0 \& A_3 \& A_0 \& A_3
	\arrow["M_{12}", from=1-1, to=1-2]
	\arrow[""{name=0, anchor=center, inner sep=0}, "M_{23}", from=1-2, to=2-2]
	\arrow["M_{12}", from=1-3, to=1-4]
	\arrow[""{name=1, anchor=center, inner sep=0}, "M_{13}"{pos=0.2}, from=1-3, to=2-4]
	\arrow["M_{23}", from=1-4, to=2-4]
	\arrow["M_{01}", color={rgb,255:red,65;green,51;blue,255}, from=2-1, to=1-1]
	\arrow[""{name=2, anchor=center, inner sep=0}, "M_{02}"{pos=0.8}, from=2-1, to=1-2]
	\arrow[""{name=3, anchor=center, inner sep=0}, "M_{03}"', from=2-1, to=2-2]
	\arrow[""{name=4, anchor=center, inner sep=0}, "M_{01}", color={rgb,255:red,65;green,51;blue,255}, from=2-3, to=1-3]
	\arrow[""{name=5, anchor=center, inner sep=0}, "M_{03}"', from=2-3, to=2-4]
	\arrow["\varphi_{012}"'{pos=0.4}, color={rgb,255:red,65;green,51;blue,255}, shorten >=3pt, Rightarrow, from=1-1, to=2]
	\arrow[color={rgb,255:red,65;green,51;blue,255}, shorten <=19pt, shorten >=19pt, Rightarrow, scaling nfold=3, from=0, to=4]
	\arrow["\varphi_{023}"{pos=0.6}, shift right, color={rgb,255:red,65;green,51;blue,255}, shorten <=8pt, shorten >=4pt, Rightarrow, from=1-2, to=3]
	\arrow["\varphi_{013}"'{pos=0.6}, shift left, color={rgb,255:red,65;green,51;blue,255}, shorten <=8pt, shorten >=4pt, Rightarrow, from=1-3, to=5]
	\arrow["\varphi_{123}", shorten >=3pt, Rightarrow, from=1-4, to=1]
\end{tikzcd}\]
we show that $\varphi_{123}$ is an isomorphism, obtaining the desired lift $\Delta^0[3]''\to \mrt^{\natural}(\cC^{\otimes})$.
By assumption we have the equality of $(A_0,A_3)$-bimodules from $M_{01}\otimes_{A_1}M_{12}\otimes_{A_2}M_{23}$ to $M_{03}$
\[\overline{\alpha}_{M_{01},M_{12},M_{23}} \bullet (M_{01} \otimes_{A_1} \varphi_{123}) \bullet \varphi_{013}= (\varphi_{012} \otimes_{A_{2}} M_{23}) \bullet \varphi_{023}.\]
Since $\varphi_{013}$ is an isomorphism we obtain the equality of $(A_0,A_3)$-bimodules from $M_{01}\otimes_{A_1} (M_{12}\otimes_{A_2}M_{23})$ to $M_{01}\otimes_{A_1}M_{13}$
\[M_{01} \otimes_{A_1} \varphi_{123} = \overline{\alpha}^{-1}_{M_{01}|M_{12}|M_{23}} \bullet (\varphi_{012} \otimes_{A_{2}} M_{23}) \bullet \varphi_{023}\bullet \varphi_{013}^{-1}.\]

It then follows that $M_{01} \otimes_{A_1} \varphi_{123}$ is an isomorphism. By \cref{ReductionI}, it follows that $A_{1} \otimes_{A_1} \varphi_{123}$ is an isomorphism. By \cref{ReductionII}, it follows that $ \varphi_{123}$ is an isomorphism, as desired.
\end{proof}

\begin{prop}
\label{thin31}
The marked simplicial set $\mrt^{\natural}(\cC^{\otimes})$ has the right lifting property with respect to the map $\Delta^k[3]'\to\Delta^k[3]''$ for $k=1,2$.
\end{prop}

\begin{proof}
We treat the case $k=1$, the case $k=2$ being analogous.
Given a map of marked simplicial sets $\Delta^1[3]'\to \mrt^{\natural}(\cC^{\otimes})$, which can be depicted as
\[\begin{tikzcd}[sep=large, ampersand replacement=\&]
	A_1 \& A_2 \& A_1 \& A_2 \\
	A_0 \& A_3 \& A_0 \& A_3
	\arrow["M_{12}", from=1-1, to=1-2]
	\arrow[""{name=0, anchor=center, inner sep=0}, "M_{23}", from=1-2, to=2-2]
	\arrow["M_{12}", from=1-3, to=1-4]
	\arrow[""{name=1, anchor=center, inner sep=0}, "M_{13}"{pos=0.2}, from=1-3, to=2-4]
	\arrow["M_{23}", from=1-4, to=2-4]
	\arrow["M_{01}", from=2-1, to=1-1]
	\arrow[""{name=2, anchor=center, inner sep=0}, "M_{02}"{pos=0.8}, from=2-1, to=1-2]
	\arrow[""{name=3, anchor=center, inner sep=0}, "M_{03}"', from=2-1, to=2-2]
	\arrow[""{name=4, anchor=center, inner sep=0}, "M_{01}", from=2-3, to=1-3]
	\arrow[""{name=5, anchor=center, inner sep=0}, "M_{03}"', from=2-3, to=2-4]
	\arrow["\varphi_{012}"'{pos=0.4}, color={rgb,255:red,65;green,51;blue,255}, shorten >=3pt, Rightarrow, from=1-1, to=2]
	\arrow[color={rgb,255:red,65;green,51;blue,255}, shorten <=19pt, shorten >=19pt, Rightarrow, scaling nfold=3, from=0, to=4]
	\arrow["\varphi_{023}"{pos=0.6}, shift right, shorten <=8pt, shorten >=4pt, Rightarrow, from=1-2, to=3]
	\arrow["\varphi_{013}"'{pos=0.6}, shift left, color={rgb,255:red,65;green,51;blue,255}, shorten <=8pt, shorten >=4pt, Rightarrow, from=1-3, to=5]
	\arrow["\varphi_{123}", color={rgb,255:red,65;green,51;blue,255}, shorten >=3pt, Rightarrow, from=1-4, to=1]
\end{tikzcd}\]
we show that $\varphi_{023}$ is an isomorphism, obtaining the desired lift $\Delta^1[3]''\to \mrt^{\natural}(\cC^{\otimes})$.
By assumption we have an equality of $(A_0,A_3)$-bimodule maps from $(M_{01}\otimes_{A_1} M_{12})\otimes_{A_2} M_{23}$ to $M_{03}$
\[\overline\alpha_{M_{01}|M_{12}|M_{23}}\bullet (M_{01} \otimes_{A_1} \varphi_{023}) \bullet \varphi_{013}= (\varphi_{012} \otimes_{A_2} M_{23}) \bullet \varphi_{023}.\]
We then deduce the equality of bimodule maps from $M_{02}\otimes_{A_2} M_{23}$ to $M_{03}$
\[(\varphi_{012} \otimes_{A_2} M_{23})^{-1} \bullet \overline\alpha_{M_{01}|M_{12}|M_{23}}\bullet (M_{01} \otimes_{A_1} \varphi_{023}) \bullet \varphi_{013}=  \varphi_{023}.
\]
By \cref{ReductionII,ReductionI}, follows that $\varphi_{023}$ is an isomorphism, as desired.
\end{proof}

\begin{prop}
\label{thinhigh}
The marked simplicial set $\mrt^{\natural}(\cC^{\otimes})$ has the right lifting property with respect to the map $\Delta^k[m]'\to\Delta^k[m]''$ for $m>3$ and $0\leq k\leq m$.
\end{prop}

\begin{proof}
Recall from \cite[Notation 13]{VerityComplicialI} that there is an adjunction
\[\mathrm{th}_3\colon m\sset\rightleftarrows\colon m\sset\colon \mathrm{sp}_3.\]
Here, $m\sset$ denotes the category of marked simplicial sets and marking preserving simplicial maps, the functor $\mathrm{th}_3$ adds to the existing marking of a marked simplicial sets all the simplices in dimension bigger or equal than $3$, and the functor $\mathrm{sp}_3$ selects the largest subsimplicial set with the property that all the simplices in dimension bigger or equal than $3$ are marked. Since $\mrt^{\natural}(\cC^{\otimes})$ is $2$-trivial, by construction we have
$\mrt^{\natural}(\cC^{\otimes})=\mathrm{sp}_3(\mrt^{\natural}(\cC^{\otimes}))$.
Hence, the two lifting problems below are equivalent:
\[
\begin{tikzcd}
    \Delta^k[m]'\arrow[r]\arrow[d]&\mrt^{\natural}(\cC^{\otimes})\\
    \Delta^k[m]''\arrow[ru,dashed]&
\end{tikzcd}
\quad\quad
\leftrightsquigarrow
\quad\quad
\begin{tikzcd}
\mathrm{th}_3\Delta^k[m]'\arrow[r]\arrow[d]&\mrt^{\natural}(\cC^{\otimes})\\
\mathrm{th}_3\Delta^k[m]''\arrow[ru,dashed]&
\end{tikzcd}
\]
Given that the vertical map in the right hand lifting problem is an isomorphism, this lifting problem has a solution. Hence, so does the left hand lifting problem, as desired.
\end{proof}

\subsection{Saturation}

In this subsection we verify that $\mrt^\natural(\cC^{\otimes})$ lifts against all saturation extensions:

\begin{thm}
\label{saturation}
The marked simplicial set $\mrt^{\natural}(\cC^{\otimes})$ has the right lifting property with respect to the map $\Delta[3]^{\mathrm{eq}}\star\Delta[\ell]\to\Delta[3]^{\sharp}\star\Delta[\ell]$ for $\ell\geq-1$. 
\end{thm}

\begin{proof}
We treat the various cases as \cref{saturation-1,saturation0,saturationHigh}.
\end{proof}

\begin{prop}
\label{saturation-1}
The marked simplicial set $\mrt^{\natural}(\cC^{\otimes})$ has the right lifting property with respect to the map $\Delta[3]^{\mathrm{eq}}\to\Delta[3]^{\sharp}$.
\end{prop}

\begin{proof}
Given a map of marked simplicial sets $\Delta[3]^{\mathrm{eq}}\to \mrt^{\natural}(\cC^{\otimes})$, which can be depicted as
\[\begin{tikzcd}[sep=large, ampersand replacement=\&]
	A_{1} \& A_{2} \& A_{1} \& A_{2} \\
	A_{0} \& A_{3} \& A_{0} \& A_{3}
	\arrow["M_{12}", from=1-1, to=1-2]
	\arrow[""{name=0, anchor=center, inner sep=0}, "M_{23}", from=1-2, to=2-2]
	\arrow["M_{12}", from=1-3, to=1-4]
	\arrow[""{name=1, anchor=center, inner sep=0}, "M_{13}"{pos=0.2}, color={rgb,255:red,65;green,51;blue,255}, from=1-3, to=2-4]
	\arrow["M_{23}", from=1-4, to=2-4]
	\arrow["M_{01}", from=2-1, to=1-1]
	\arrow[""{name=2, anchor=center, inner sep=0}, "M_{02}"{pos=0.8}, color={rgb,255:red,65;green,51;blue,255}, from=2-1, to=1-2]
	\arrow[""{name=3, anchor=center, inner sep=0}, "M_{03}"', from=2-1, to=2-2]
	\arrow[""{name=4, anchor=center, inner sep=0}, "M_{01}", from=2-3, to=1-3]
	\arrow[""{name=5, anchor=center, inner sep=0}, "M_{03}"', from=2-3, to=2-4]
	\arrow["\alpha"'{pos=0.4}, color={rgb,255:red,65;green,51;blue,255}, shorten >=3pt, Rightarrow, from=1-1, to=2]
	\arrow[color={rgb,255:red,65;green,51;blue,255}, shorten <=19pt, shorten >=19pt, Rightarrow, scaling nfold=3, from=0, to=4]
	\arrow["\beta"{pos=0.6}, shift right, color={rgb,255:red,65;green,51;blue,255}, shorten <=8pt, shorten >=4pt, Rightarrow, from=1-2, to=3]
	\arrow["\gamma"'{pos=0.6}, shift left, color={rgb,255:red,65;green,51;blue,255}, shorten <=8pt, shorten >=4pt, Rightarrow, from=1-3, to=5]
	\arrow["\delta", color={rgb,255:red,65;green,51;blue,255}, shorten >=3pt, Rightarrow, from=1-4, to=1]
\end{tikzcd}\]
we show that $M_{01}$, $M_{12}$, $M_{23}$ and $M_{03}$ are equivalences, obtaining the desired lift $\Delta[3]^\sharp\to \mrt^{\natural}(\cC^{\otimes})$.

Since $M_{02}$ and $M_{13}$ are equivalences, there exists $M_{02}'$ a $(A_{2},A_{0})$-bimodule, $M_{13}'$ a $(A_{3},A_{1})$-bimodule, as well as 
isomorphisms of bimodules
\[
M_{02}\otimes_{A_{2}} M_{02}'\cong A_{0} \quad M_{02}'\otimes_{A_{0}} M_{02}\cong A_{2} \quad M_{13} \otimes_{A_{3}} M_{13}'\cong A_{1} \quad M_{13}'\otimes_{A_{1}} M_{13} \cong A_{3}.
\]
We obtain a $(A_{2},A_{1})$-bimodule isomorphism in $\cC$:
\[
\begin{array}{lll}
    M_{02}'\otimes_{A_{0}} M_{01} &\cong (M_{02}'\otimes_{A_{0}} M_{01})\otimes_{A_{1}} A_{1} &\text{\cref{UnitalityForBimodules}}\\
    &\cong (M_{02}'\otimes_{A_{0}} M_{01})\otimes_{A_{1}} (M_{13}\otimes_{A_{3}} M_{13}')& \\
    &\cong (M_{02}'\otimes_{A_{0}}(M_{01} \otimes_{A_{1}} M_{13}))\otimes_{A_{3}} M_{13}'&\text{\cref{AssociativityForBimodules}} \\
    &\cong (M_{02}'\otimes_{A_{0}}(M_{02}\otimes_{A_{2}}M_{23}))\otimes_{A_{3}}M_{13}'&\text{\cref{TensorWithIso}}\\
    &\cong (M_{02}'\otimes_{A_{0}}M_{02})\otimes_{A_{2}}(M_{23}\otimes_{A_{3}}M_{13}')\\
    &\cong A_{2}\otimes_{A_{2}}(M_{23}\otimes_{A_{3}}M_{13}')\\
    &\cong M_{23}\otimes_{A_{3}}M_{13}'&\text{\cref{UnitalityForBimodules}}.
\end{array}
\]
We then define $M_{12}' \coloneqq M_{02}' \otimes_{A_{0}} M_{01}$.
On one hand, we have an isomorphism of $(A_{1},A_{1})$-bimodules
\[
\begin{array}{lllll}
M_{12}\otimes_{A_{2}} M_{12}' &= M_{12} \otimes_{A_{2}} (M_{02}' \otimes_{A_{0}} M_{01})&\text{definition}\\
&\cong M_{12} \otimes_{A_{2}} (M_{23} \otimes_{A_{3}} M_{13}')&\text{\cref{TensorWithIso}}\\
&\cong (M_{12} \otimes_{A_{2}} M_{23}) \otimes_{A_{3}} M_{13}'&\text{\cref{AssociativityForBimodules}}\\
&\cong M_{13} \otimes_{A_{3}} M_{13}'\cong A_{1}&\text{\cref{TensorWithIso}}
\end{array}\]
On the other hand, we have an isomorphism of $(A_{2},A_{2})$-bimodules
\[\begin{array}{llllll}
M_{12}'\otimes M_{12}&=(M_{02}' \otimes_{A_{0}} M_{01}) \otimes_{A_{1}} M_{12} &\text{definition}\\
&\cong M_{02}' \otimes_{A_{0}} (M_{01} \otimes_{A_{1}} M_{12})&\text{\cref{AssociativityForBimodules}}\\
&\cong M_{02}' \otimes_{A_{0}} M_{02}\cong A_{2}&\text{\cref{TensorWithIso}}
\end{array}\]
So $M_{12}$ is an equivalence. Further, by \cref{thin20} it follows that $M_{12}$ and $M_{23}$ are equivalences. Finally, by \cref{thin21}, it follows that $M_{03}$ is an equivalence, concluding the proof.
\end{proof}

\begin{prop}
\label{saturation0}
The marked simplicial set $\mrt^{\natural}(\cC^{\otimes})$ has the right lifting property with respect to the map $\Delta[3]^{\mathrm{eq}}\star\Delta[0]\to\Delta[3]^{\sharp}\star\Delta[0]$.
\end{prop}

\begin{proof}
Given a map of marked simplicial sets $\Delta[3]^{\mathrm{eq}}\star\Delta[0]\to \mrt^{\natural}(\cC^{\otimes})$, which can be depicted as
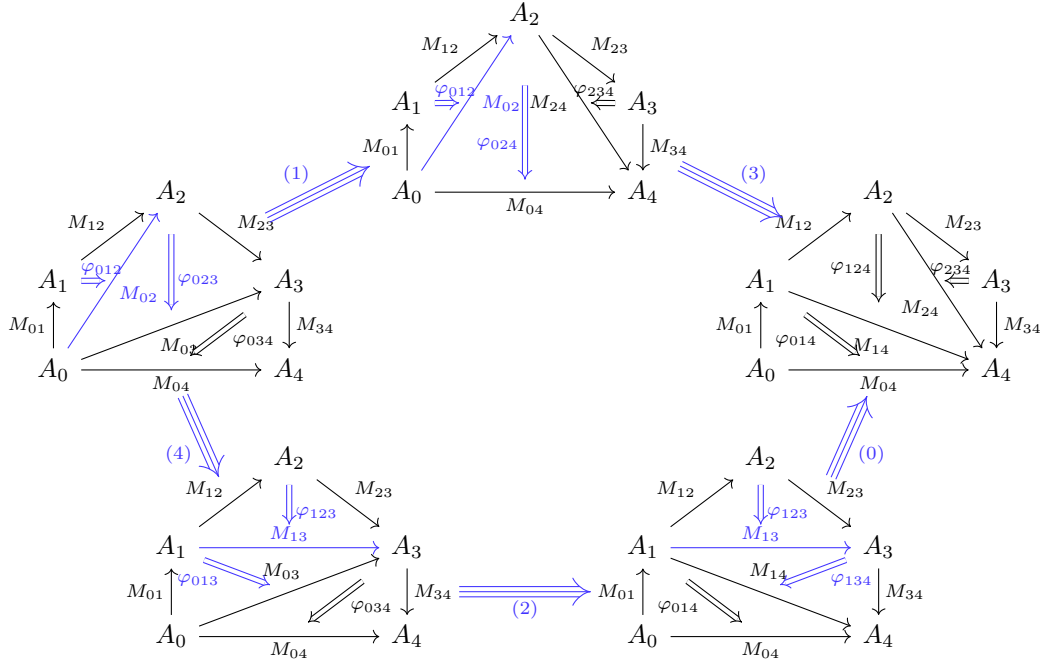
\[\begin{tikzcd}[ampersand replacement=\&]
	\&\&\&\& {A_2} \\
	\&\&\& {A_1} \&\& {A_3} \\
	\& {A_2} \&\& {A_0} \&\& {A_4} \&\& {A_2} \\
	{A_1} \&\& {A_3} \&\&\&\& {A_1} \&\& {A_3} \\
	{A_0} \&\& {A_4} \&\&\&\& {A_0} \&\& {A_4} \\
	\&\& {A_2} \&\&\&\& {A_2} \\
	\& {A_1} \&\& {A_3} \&\& {A_1} \&\& {A_3} \\
	\& {A_0} \&\& {A_4} \&\& {A_0} \&\& {A_4}
	\arrow["{M_{23}}", from=1-5, to=2-6]
	\arrow[""{name=0, anchor=center, inner sep=0}, "{M_{24}}"'{pos=0.4}, from=1-5, to=3-6]
	\arrow["{M_{12}}", from=2-4, to=1-5]
	\arrow[""{name=1, anchor=center, inner sep=0}, "{M_{34}}", from=2-6, to=3-6]
	\arrow[""{name=2, anchor=center, inner sep=0}, "{M_{23}}", from=3-2, to=4-3]
	\arrow[""{name=3, anchor=center, inner sep=0}, "{M_{02}}"'{pos=0.6}, color={rgb,255:red,65;green,51;blue,255}, from=3-4, to=1-5]
	\arrow[""{name=4, anchor=center, inner sep=0}, "{M_{01}}", from=3-4, to=2-4]
	\arrow[""{name=5, anchor=center, inner sep=0}, "{M_{04}}"', from=3-4, to=3-6]
	\arrow["{M_{23}}", from=3-8, to=4-9]
	\arrow[""{name=6, anchor=center, inner sep=0}, "{M_{24}}"'{pos=0.6}, from=3-8, to=5-9]
	\arrow["{M_{12}}", from=4-1, to=3-2]
	\arrow["{M_{34}}", from=4-3, to=5-3]
	\arrow[""{name=7, anchor=center, inner sep=0}, "{M_{12}}", from=4-7, to=3-8]
	\arrow[""{name=8, anchor=center, inner sep=0}, "{M_{14}}"'{pos=0.6}, from=4-7, to=5-9]
	\arrow["{M_{34}}", from=4-9, to=5-9]
	\arrow[""{name=9, anchor=center, inner sep=0}, "{M_{02}}"', color={rgb,255:red,65;green,51;blue,255}, from=5-1, to=3-2]
	\arrow["{M_{01}}", from=5-1, to=4-1]
	\arrow[""{name=10, anchor=center, inner sep=0}, "{M_{03}}"'{pos=0.4}, from=5-1, to=4-3]
	\arrow[""{name=11, anchor=center, inner sep=0}, "{M_{04}}"', from=5-1, to=5-3]
	\arrow["{M_{01}}", from=5-7, to=4-7]
	\arrow[""{name=12, anchor=center, inner sep=0}, "{M_{04}}"', from=5-7, to=5-9]
	\arrow["{M_{23}}", from=6-3, to=7-4]
	\arrow[""{name=13, anchor=center, inner sep=0}, "{M_{23}}", from=6-7, to=7-8]
	\arrow[""{name=14, anchor=center, inner sep=0}, "{M_{12}}", from=7-2, to=6-3]
	\arrow[""{name=15, anchor=center, inner sep=0}, "{M_{13}}", color={rgb,255:red,65;green,51;blue,255}, from=7-2, to=7-4]
	\arrow[""{name=16, anchor=center, inner sep=0}, "{M_{34}}", from=7-4, to=8-4]
	\arrow["{M_{12}}", from=7-6, to=6-7]
	\arrow[""{name=17, anchor=center, inner sep=0}, "{M_{13}}", color={rgb,255:red,65;green,51;blue,255}, from=7-6, to=7-8]
	\arrow[""{name=18, anchor=center, inner sep=0}, "{M_{14}}"{pos=0.4}, from=7-6, to=8-8]
	\arrow["{M_{34}}", from=7-8, to=8-8]
	\arrow["{M_{01}}", from=8-2, to=7-2]
	\arrow[""{name=19, anchor=center, inner sep=0}, "{M_{03}}"{pos=0.6}, from=8-2, to=7-4]
	\arrow[""{name=20, anchor=center, inner sep=0}, "{M_{04}}"', from=8-2, to=8-4]
	\arrow[""{name=21, anchor=center, inner sep=0}, "{M_{01}}", from=8-6, to=7-6]
	\arrow[""{name=22, anchor=center, inner sep=0}, "{M_{04}}"', from=8-6, to=8-8]
	\arrow["{\varphi_{024}}"'{pos=0.7}, color={rgb,255:red,65;green,51;blue,255}, shorten <=19pt, shorten >=4pt, Rightarrow, from=1-5, to=5]
	\arrow["{\varphi_{012}}"{pos=0.7}, color={rgb,255:red,65;green,51;blue,255}, shorten >=2pt, Rightarrow, from=2-4, to=3]
	\arrow["{\varphi_{234}}"'{pos=0.7}, shorten >=2pt, Rightarrow, from=2-6, to=0]
	\arrow["(3)", color={rgb,255:red,65;green,51;blue,255}, shorten <=15pt, shorten >=15pt, Rightarrow, scaling nfold=3, from=1, to=7]
	\arrow["(1)", color={rgb,255:red,65;green,51;blue,255}, shorten <=15pt, shorten >=15pt, Rightarrow, scaling nfold=3, from=2, to=4]
	\arrow["{\varphi_{023}}"{pos=0.6}, color={rgb,255:red,65;green,51;blue,255}, shorten <=8pt, shorten >=5pt, Rightarrow, from=3-2, to=10]
	\arrow["{\varphi_{124}}"', shorten <=8pt, shorten >=8pt, Rightarrow, from=3-8, to=8]
	\arrow["{\varphi_{012}}"{pos=0.7}, color={rgb,255:red,65;green,51;blue,255}, shorten >=2pt, Rightarrow, from=4-1, to=9]
	\arrow["{\varphi_{034}}"{pos=0.4}, shorten <=8pt, shorten >=8pt, Rightarrow, from=4-3, to=11]
	\arrow["{\varphi_{014}}"'{pos=0.4}, shorten <=8pt, shorten >=8pt, Rightarrow, from=4-7, to=12]
	\arrow["{\varphi_{234}}"'{pos=0.6}, shorten >=2pt, Rightarrow, from=4-9, to=6]
	\arrow["(4)"', color={rgb,255:red,65;green,51;blue,255}, shorten <=10pt, shorten >=10pt, Rightarrow, scaling nfold=3, from=11, to=14]
	\arrow["{\varphi_{123}}", color={rgb,255:red,65;green,51;blue,255}, shorten <=2pt, shorten >=8pt, Rightarrow, from=6-3, to=15]
	\arrow["(0)"', color={rgb,255:red,65;green,51;blue,255}, shorten <=10pt, shorten >=10pt, Rightarrow, scaling nfold=3, from=13, to=12]
	\arrow["{\varphi_{123}}", color={rgb,255:red,65;green,51;blue,255}, shorten <=2pt, shorten >=8pt, Rightarrow, from=6-7, to=17]
	\arrow["{\varphi_{013}}"'{pos=0.3}, color={rgb,255:red,65;green,51;blue,255}, shorten <=2pt, shorten >=7pt, Rightarrow, from=7-2, to=19]
	\arrow["(2)"', color={rgb,255:red,65;green,51;blue,255}, shorten <=19pt, shorten >=19pt, Rightarrow, scaling nfold=3, from=16, to=21]
	\arrow["{\varphi_{034}}"{pos=0.4}, shorten <=8pt, shorten >=8pt, Rightarrow, from=7-4, to=20]
	\arrow["{\varphi_{014}}"'{pos=0.4}, shorten <=8pt, shorten >=8pt, Rightarrow, from=7-6, to=22]
	\arrow["{\varphi_{134}}"{pos=0.3}, color={rgb,255:red,65;green,51;blue,255}, shorten <=2pt, shorten >=7pt, Rightarrow, from=7-8, to=18]
\end{tikzcd}\]
we show that the maps $\varphi_{124}$, $\varphi_{014}$, $\varphi_{234}$, and $\varphi_{034}$ are isomorphisms, obtaining the desired lift $\Delta[3]^{\sharp}\star\Delta[0]\to \mrt^{\natural}(\cC^{\otimes})$.

By \cref{saturation-1}, we know that $M_{01}$, $M_{12}$, $M_{23}$, and $M_{03}$ are equivalences.
By (3), we have 
\[(\varphi_{012} \otimes_{A_{2}} M_{24}) \bullet \varphi_{024} = \overline{\alpha}_{M_{01}|M_{12}|M_{24}}\bullet (M_{01} \otimes_{A_{1}} \varphi_{124}) \bullet \varphi_{014}.\] 
Then
\[\overline{\alpha}^{-1}_{M_{01}|M_{12}|M_{24}}\bullet(\varphi_{012} \otimes_{A_{2}} M_{24}) \bullet \varphi_{024} = (M_{01} \otimes_{A_{1}} \varphi_{124}) \bullet \varphi_{014}.\] 
Since the left-hand side is an isomorphism, it follows that $M_{01}\otimes_{A_1}\varphi_{124}$ admits a right inverse.
Further, by (0), we have
\[(\varphi_{123} \otimes_{A_{3}} M_{34}) \bullet \varphi_{134} = \overline{\alpha}_{M_{12}|M_{23}|M_{34}}\bullet (M_{12} \otimes_{A_{2}} \varphi_{234}) \bullet \varphi_{124}.\]
Since the left-hand side is an isomorphism, it follows that $M_{01}\otimes_{A_1}\varphi_{124}$ admits a left inverse.
In total, we have that $M_{01}\otimes_{A_1}\varphi_{124}$ is a bimodule isomorphism. Since $M_{01}$ is an equivalence, by \cref{ReductionI} we obtain that $\varphi_{124}$ is a bimodule isomorphism.
By \cref{thin31,thin30}, it follows that $\varphi_{014}$ and $\varphi_{234}$ are isomorphisms. By \cref{thin31}, it follows that $\varphi_{034}$ is an isomorphism, concluding the proof.
\end{proof}

\begin{prop}
\label{saturationHigh}
The marked simplicial set $\mrt^{\natural}(\cC^{\otimes})$ has the right lifting property with respect to the map $\Delta[3]^{\mathrm{eq}}\star\Delta[\ell]\to\Delta[3]^{\sharp}\star\Delta[\ell]$ for $\ell>0$.
\end{prop}

\begin{proof}
Recall from the proof of \cref{thinhigh} that there is an adjunction
\[\mathrm{th}_3\colon m\sset\rightleftarrows m\sset\colon\mathrm{sp}_3.\]
Since $\mrt^{\natural}(\cC^{\otimes})$ is $2$-trivial, by construction we have
$\mrt^{\natural}(\cC^{\otimes})=\mathrm{sp}_3(\mrt^{\natural}(\cC^{\otimes}))$.
Hence, the two lifting problems below are equivalent:
\[
\begin{tikzcd}
\mathrm{th}_3(\Delta[3]^{\mathrm{eq}}\star\Delta[\ell]))\arrow[r]\arrow[d]&\mrt^{\natural}(\cC^{\otimes})\\
\mathrm{th}_3(\Delta[3]^{\sharp}\star\Delta[\ell])\arrow[ru,dashed]&
\end{tikzcd}
\quad\quad
\leftrightsquigarrow
\quad\quad
\begin{tikzcd}
\Delta[3]^{\mathrm{eq}}\star\Delta[\ell]\arrow[r]\arrow[d]&\mrt^{\natural}(\cC^{\otimes})\\
    \Delta[3]^{\sharp}\star\Delta[\ell]\arrow[ru,dashed]&
\end{tikzcd}
\]
Now, the vertical map in the left hand lifting problem is an isomorphism, because the marking on $\Delta[3]^{\mathrm{eq}}\star\Delta[\ell]$ and $\Delta[3]^{\sharp}\star\Delta[\ell]$ only differs in dimension higher than $3$, this lifting problem has a solution. Hence, so does the right hand lifting problem, as desired.
\end{proof}

\bibliographystyle{amsalpha}
\bibliography{ref}

\end{document}